\documentclass[11pt,reqno]{amsart}

\usepackage{amsfonts,amsmath, amssymb,amsthm,amscd,mathrsfs}
\usepackage{lmodern}
\usepackage[left=2.5cm, right=2.5cm, top=2.5cm, bottom=3cm]{geometry}
\usepackage{hyperref}
\usepackage{enumitem}
\usepackage{bbm}
\usepackage [english]{babel}
\usepackage [autostyle, english = american]{csquotes}
\MakeOuterQuote{"}
\usepackage{dsfont}

\newtheorem{theorem}{Theorem}[section]

\newtheorem{lemma}[theorem]{Lemma}
\newtheorem{proposition}[theorem]{Proposition}
\newtheorem{corollary}[theorem]{Corollary}

\theoremstyle{definition}
\newtheorem{remark}[theorem]{Remark}

\theoremstyle{definition}

\theoremstyle{definition}
\newtheorem{definition}[theorem]{Definition}

\newcommand{\R}{\mathbb{R}}
\renewcommand{\P}{\mathbb{P}}
\newcommand{\E}{\mathbb{E}}
\newcommand{\C}{\mathbb{C}}

\newcommand{\Z}{\mathbb{Z}}

\renewcommand{\o}[1]{\overline{#1}}

\newcommand{\msf}{\mathsf}
\newcommand{\mc}{\mathcal}

\renewcommand{\Im}{\mathfrak{Im}}
\renewcommand{\Re}{\mathfrak{Re}}

\newcommand{\eeq}{\end{equation}}

\newcommand{\K}{\ensuremath{\mathsf{K}}}
\renewcommand{\i}{\mathbf{i}}
\newcommand{\I}{\mathbf{i}}
\def \Ai {{\rm Ai}}

\renewcommand{\epsilon}{\varepsilon}
\newcommand{\e}{\varepsilon}
\newcommand{\f}{\frac}

\renewcommand{\emph}[1]{\textsf{\textit{#1}}}
\renewcommand{\geq}{\geqslant}
\renewcommand{\leq}{\leqslant}

\usepackage[dvipsnames]{xcolor}

\usepackage{tikz}

\title{Large deviations for sticky Brownian motions}
\author[G. Barraquand]{Guillaume Barraquand}
\address{G. Barraquand, Laboratoire de Physique de l'ecole Normale Sup\'erieure, ENS, Universit\'e PSL, CNRS,
	Sorbonne Universit\'e, Universit\'e de Paris, Paris, France.}
\email{guillaume.barraquand@ens.fr}

\author[M. Rychnovsky]{Mark Rychnovsky}
\address{M. Rychnovsky, Columbia University,
	Department of Mathematics,
	2990 Broadway,
	New York, NY 10027, USA.}
\email{mrychnov@gmail.com}
\begin{document}
	
\begin{abstract}  
We consider $n$-point sticky Brownian motions: a family of $n$ diffusions that evolve as independent Brownian motions when they are apart, and interact locally so that the set of coincidence times has positive Lebesgue measure with positive probability. These diffusions can also be seen as $n$ random motions in a random environment whose distribution is given by so-called stochastic flows of kernels. 
For a specific type of sticky interaction, we prove exact formulas characterizing the stochastic flow and show that in the large deviations regime, the random fluctuations of these stochastic flows are Tracy-Widom GUE distributed. An equivalent formulation of this result states that the extremal particle among $n$ sticky Brownian motions has Tracy-Widom distributed fluctuations in the large $n$ and large time limit. These results are proved by viewing sticky Brownian motions as a (previously known) limit of the exactly solvable beta random walk in random environment.
\end{abstract}
\maketitle

\setcounter{tocdepth}{1}
\tableofcontents

\section{Introduction and main results}

\subsection{Introduction}

Families of interacting Brownian motions have been related to random matrix theory in a number of works. For instance at any fixed time nonintersecting Brownian motions have the same distribution as the eigenvalues of a matrix from the Gaussian unitary ensemble (GUE) \cite{dyson1962brownian}. Certain statistics of families of Brownian motions with asymmetric reflections also have Tracy-Widom GUE distributed fluctuations \cite{warren2007dyson} as the number of particles goes to $+\infty$. There are many other examples (see for instance \cite{baryshnikov2001gues, gravner2001limit, o2002representation, Spohn2015, Ferrari2015, o2012directed, macdonaldprocesses,borodin2014free}), and the ubiquitous occurrence of the GUE can be understood in the framework of the Kardar-Parisi-Zhang (KPZ) universality class. This framework predicts that in spatial dimension $1$, many growth models, interacting particle systems and directed polymer models have Tracy-Widom fluctuations in the cube-root time scale, for appropriate initial data. This class is extremely broad and is not yet clearly delineated.  In particular one may expect that many families of interacting  Brownian motions fall in the KPZ universality class and are related to random matrix theoretic distributions. The examples cited above all deal with families of Brownian motions with repulsive interaction; in this paper we study a family of Brownian motions with attractive interaction called sticky Brownian motions.

\bigskip 
\noindent In 1952 Feller introduced a reflected Brownian motion sticky at the origin which evolves as a Brownian motion everywhere except at origin, and has its reflection off the origin slowed down so that the total time its trajectory spends at the origin has positive Lebesgue measure \cite{Feller52}. This motion's law can be characterized by a single stickiness parameter which determines how much time it spends at the origin. More recently, using stochastic flows and Dirichlet forms \cite{LeJanRaimond04a, le2004sticky} or through a  martingale problem \cite{HowittWarren09a,HowittWarren09b}, several authors have defined families of $n$-particle diffusions where the distance between each pair of particles is a reflected Brownian motion sticky at the origin.  

\bigskip 
\noindent These $n$-point sticky Brownian motions describe the evolution of mesoscopic particles with attractive interaction at a scale smaller than their radius; this situation is common in the study of colloids  \cite{Holmes17}. Sticky Brownian motions are the diffusive scaling limit of various models:  discrete random walks in random environment \cite{HowittWarren09b, amir1991sticky}, certain  families of exclusion processes with a  tunable interaction \cite{RaczShkolnikov15}, and storage processes \cite{harrison1981sticky}. Using the language of stochastic flows of kernels, sticky Brownians motion can be described as independent motions in a space-time i.i.d. random environment \cite{le2004sticky, LeJanLemaire04, schertzer2009special, brownianwebandnet}.

\bigskip 
\noindent In this paper we restrict our attention to a specific one-parameter family of sticky Brownian motions which we will call uniform sticky Brownian motions where the multiparticle interactions are completely determined by the two particle interactions. Within this restricted class, we prove a quenched large deviation principle (Theorem \ref{th:LDP}) for the random heat kernel (referred to below as the uniform Howitt-Warren stochastic flow of kernels). We then prove that the random lower order corrections to the large deviation principle, which come from the random environment, are Tracy-Widom GUE distributed in the large time limit (Theorem \ref{th:tracy widom fluctuations}). This gives a positive answer, in the case of uniform sticky Brownian motions, to a question posed in \cite[Section 8.3 (4)]{brownianwebandnet}. Our results can be rephrased to say that as time and the number of particles $n$ are simultaneously sent to infinity, the position of the extremal particle of $n$ uniform sticky Brownian motions has Tracy-Widom GUE distributed fluctuations (Corollary \ref{cor:LDP}).

\bigskip 
\noindent We prove these results by viewing uniform sticky Brownian motions as the limit of a discrete exactly solvable model: the beta random walk in random environment (RWRE). Using exact formulas for the latter, we prove a Fredholm determinant formula for the Laplace transform of the random heat kernel associated to sticky Brownian motions. We then perform rigorous saddle point asymptotics to prove the Tracy-Widom GUE limit theorem. We also provide mixed moment formulas for the stochastic flow of kernels, which yield concise formulas for the probability distribution at time $t$ of the maximum of $n$-point sticky Brownian motions started from arbitrary particle positions (Proposition \ref{prop:moment formula sticky bm}). Though we uncover the integrability of the model by degenerating earlier results, this allows us to bring the techniques of integrable probability to bear on sticky Brownian motions and stochastic flows, which occur as the scaling limit of many stochastic processes. On a more technical side the asymptotic analysis of the Fredholm determinant formula for the beta RWRE was challenging and could only be performed for a very specific choice of parameters; we overcome some of these challenges in Section 3 through a careful analysis of the level lines of a meromorphic function with infinitely many poles. 

 \bigskip 
\noindent We also describe intriguing connections (see Remark \ref{rem:relationtorandomdiffusion}) between the uniform Howitt-Warren (or Le Jan-Raimond) stochastic flow of kernels and the a priori ill-posed diffusion (considered in physics \cite{le2017diffusion})
$$ dX_t = \xi(t,X_t)dt+dB_t,$$
where $\xi$ is a space time white noise independent from the driving Brownian motion $B$ or of the stochastic PDE
$$ \partial_tv = \frac 1 2 \partial_{xx} v + \xi \partial_x v,$$
associated to the above diffusion via the Kolmogorov backward equation.

\subsection{Definitions}
Before stating our main results, we need to introduce the notions of sticky Brownian motions and stochastic flows of kernels. Recall that the local time of a Brownian motion $B_t$ at the point $a$ is  defined by the almost-sure limit
$$ \ell^a_t(B)  = \lim_{\e \to 0} \frac{1}{2 \e} \int_0^t \mathds{1}_{a-\epsilon \leq B_s \leq a+\epsilon} ds =\lim_{\e \to 0} \frac{1}{\e} \int_0^t \mathds{1}_{a \leq B_s \leq a+\epsilon} ds. $$
For a continuous semimartingale $X_t$, the natural time scale is given by its quadratic variation $\langle X, X\rangle_t$ and we define the local time as the almost sure limit \cite[Corollary 1.9, Chap. VI]{revuz2013continuous}
$$\ell_t^a(X)=\lim_{\e \to 0} \frac{1}{\e} \int_0^t \mathds{1}_{a\leq X_s \leq a+\epsilon} d\langle X, X\rangle_t.$$
Feller initiated the study of Brownian motions sticky at the origin in \cite{Feller52}, while studying general boundary conditions for diffusions on the half line.
\begin{definition}
\emph{Brownian motion sticky at the origin} can be defined as the weak solution to the system of stochastic differential equations
\begin{align} 
\label{eq:stochastic equation sticky brownian}
dX_t&=\mathds{1}_{\{X_t \neq 0\}}dB_t,\\
\int_0^t\mathds{1}_{X_s=0} ds&= \frac{1}{2 \lambda} \ell_t^0(X), \nonumber
\end{align}
where $B_t$ is a Brownian motion. Reflected Brownian motion sticky at the origin can be defined as $Y_t=|X_t|$ where $X_t$ is a Brownian motion sticky at the origin.

\begin{remark}[Time change] \label{time change} Brownian motion sticky at the origin can be viewed as a time change of  Brownian motion in a construction due to Ito and McKean \cite{ItoMcKean63}. Consider the Brownian motion $B_t$, and define the continuous increasing function $A(t)=t+\frac{1}{2 \lambda} \ell_t^0(B)$. Let $T(t)=A^{-1}(t)$ and set $X_t=B_{T(t)}$. We see that $X_t$ is a usual Brownian motion when $X_t \neq 0$, because the local time of $B_t$ only increases when $B_t=0$. When $X_t=0$ time slows down. We know $\int_0^t \mathbbm{1}_{X_s>0}ds=T(t)$, so $\int_0^t \mathds{1}_{X_s=0} ds=t-T(t)=\frac{1}{2 \lambda} \ell^{0}_{T(t)}(B)=\frac{1}{2 \lambda} \ell^0_{t}(X)$.  This type of time change can be used to produce many processes with sticky interactions. 
\end{remark}

\begin{remark}[Discrete limit]
Reflected Brownian motion sticky at the origin $Y_t$ can also be viewed as the diffusive limit of a sequence of random walks which tend to stay at $0$. For small $\epsilon>0$, let $Z^{\epsilon}_t$ be a discrete time random walk on $\mathbb{Z}_{\geq 0}$, which behaves as a simple symmetric random walk when it is not at the point $0$. When $Z_t^{\epsilon}$ is at the point $0$, at each time step it travels to $1$ with probability $\epsilon$ and stays at $0$ with probability $1-\epsilon$. The diffusive limit $\epsilon Z^{2 \lambda \epsilon}_{\epsilon^{-2} t}$ converges to $Y_t$ weakly as $\epsilon \to 0$. To understand this convergence see equation \eqref{eq:stochasticequationreflectedstickybrownianslope}, and note that the drift of the limiting motion at $0$ is equal to $2 \lambda$ because in each unit of time there are $\epsilon^{-2}$ opportunities to jump from $0$ to $\epsilon$ and the proportion of these opportunities that is taken is approximately $2 \lambda \epsilon$. The analogous statement is also true for Brownian motion sticky at the origin. See Figure \ref{fig:stickyreflected} where a simulation of $Z^{1/5}_t$ is shown alongside $Y_{t}$.
\end{remark}

\begin{figure}
\begin{tikzpicture}
\begin{scope}
\node{\includegraphics[width=7.5cm]{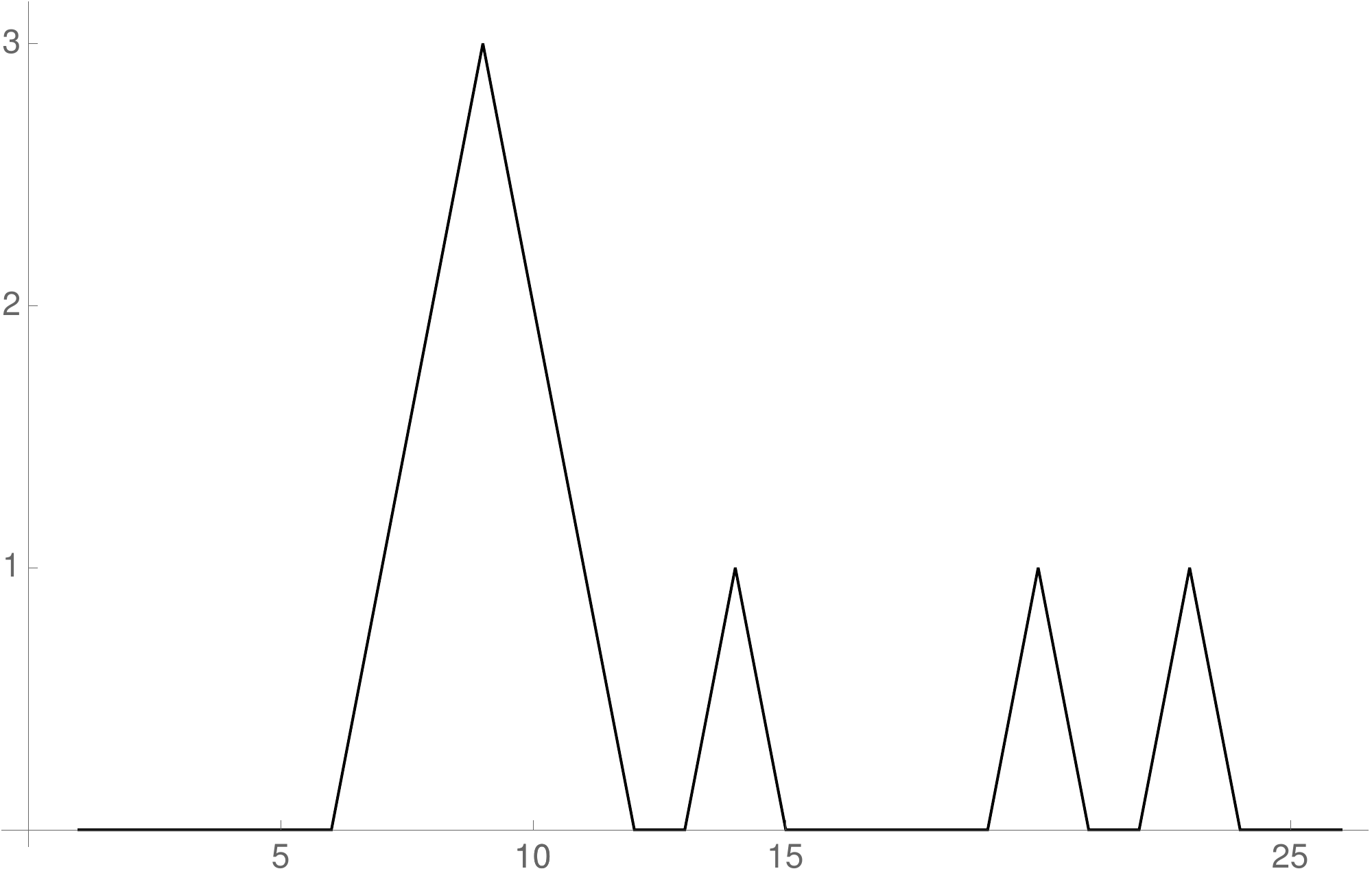}};
\end{scope}
\begin{scope}[xshift=8cm, yshift=0.65cm]
\node{\includegraphics[width=7.5cm, trim=0 0 0 
	-3cm]{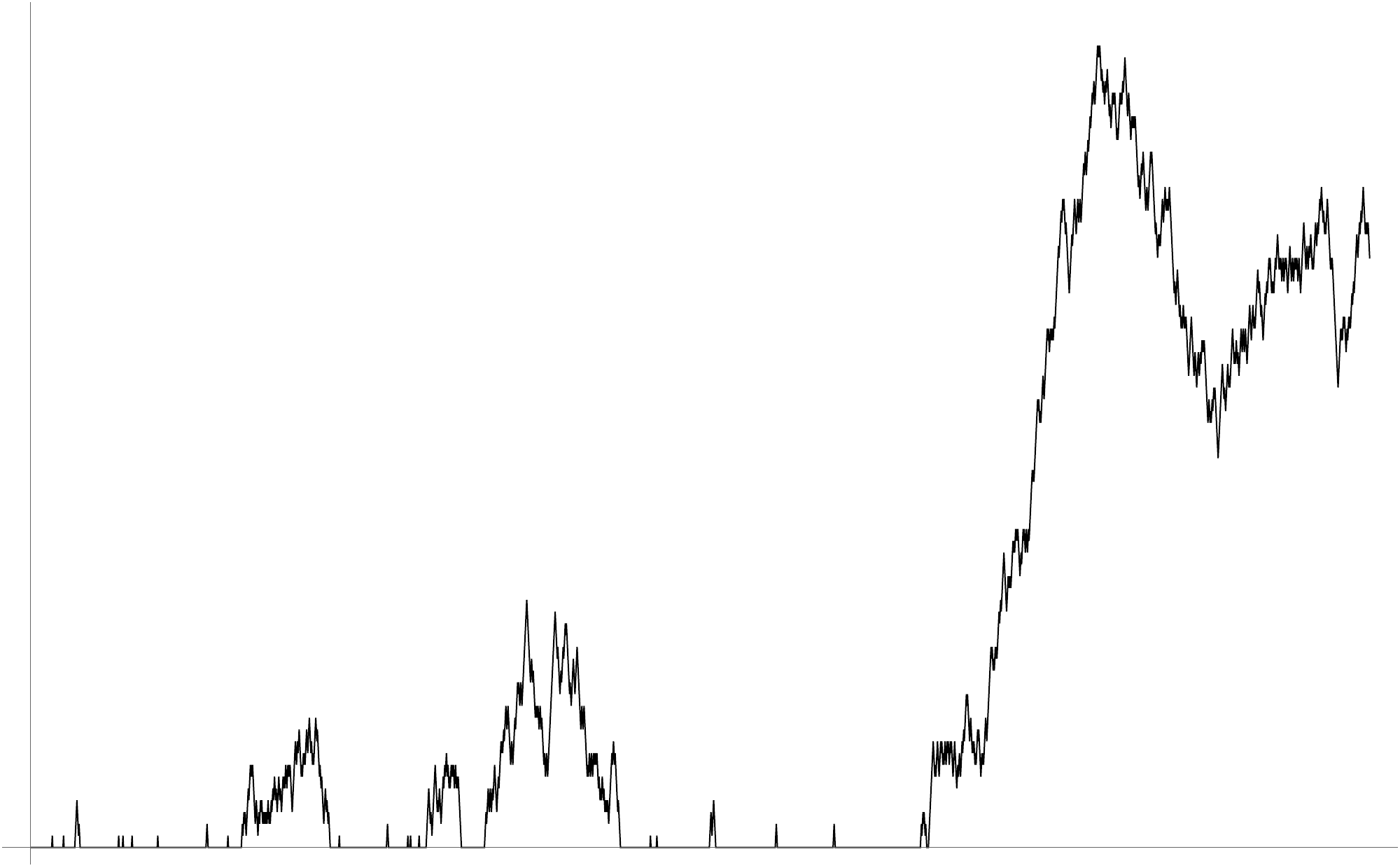}};
\end{scope}
\end{tikzpicture}
\caption{Left panel: Random walk $Z_t^{1/5}$ leaving $0$ with probability $1/5$, up to time $25$. Right panel: Reflected Brownian motion sticky at $0$ obtained by the scaling limit of $Z_t^{\epsilon}$.}
\label{fig:stickyreflected}
\end{figure}
From Remark \ref{time change} and the Tanaka Formula for reflected Brownian motion it is easy to see that $Y_t$ is a weak solution to the system of stochastic differential equations
\begin{align} \label{eq:stochastic equation reflected sticky brownian}
dY_t&= \frac{1}{2} d \ell_t^0(Y) +\mathds{1}_{\{Y_t>0\}} dB_t,\\
\mathds{1}_{\{Y_t=0\}}&=\frac{1}{ 4 \lambda} d \ell_t^0(Y), \nonumber 
\end{align}
Equations \eqref{eq:stochastic equation reflected sticky brownian} is equivalent to the single SDE
\begin{equation} \label{eq:stochasticequationreflectedstickybrownianslope}
dY_t=2 \lambda \mathds{1}_{\{Y_t=0\}} dt + \mathds{1}_{\{Y_t>0\}} dB_t,
\end{equation}

in the sense that a weak solution to one is a weak solution to the other \cite{EngelbertPeskir14}. Existence and uniqueness of weak solutions to \eqref{eq:stochastic equation sticky brownian} and \eqref{eq:stochastic equation reflected sticky brownian} can be found in \cite{EngelbertPeskir14} and references therein.

\end{definition}
 Nonexistence of strong solutions to equations \eqref{eq:stochastic equation sticky brownian} and \eqref{eq:stochastic equation reflected sticky brownian} was first shown in \cite{Chitashvili89} and \cite{Warren97} (see also \cite{EngelbertPeskir14} for a more canonical arguments which would more easily generalize to other sticky processes). 
 Several other works have been published on the existence of solutions to similar SDEs with indicator functions as the coefficient of $dB_t$ or $dt$ including \cite{Karatzas11,Bass14}. A more complete history of these SDEs can be found in \cite{EngelbertPeskir14}.

\bigskip 
We wish to study the evolution of $n$ particles in one spatial dimension where the difference between any pair of particles is a Brownian motion sticky at the origin. First we do this for a pair of sticky Brownian motions.

\begin{definition}
The stochastic process $(X_1(t),X_2(t))$ is a pair of Brownian motions with sticky interaction if each $X_i$ is marginally distributed as a Brownian motion and
\begin{align} \label{eq:sticky 2 particle eq 1}\langle X_1, X_2 \rangle(t)&=\int_0^t \mathds{1}_{X_1(s)=X_2(s)} ds,\\
\label{eq:sticky 2 particle eq 2}\int_0^t \mathds{1}_{X_1(s)=X_2(s)} ds&=\frac{1}{2 \lambda}\ell^0_t(X_1-X_2).
\end{align}
In other words $(X_1(t),X_2(t))$ are sticky Brownian motions if they evolve as independent Brownian motions when they are at different positions and their difference is a Brownian motion sticky at $0$ (see a simulation in Fig. \ref{fig:simulations}). The parameter $\lambda$ can be understood as the rate (in a certain excursion theoretic sense) at which the two particles split when they are at the same position.
\end{definition}
 \begin{figure}
\begin{tikzpicture}
\begin{scope}
\node{\includegraphics[width=7.5cm]{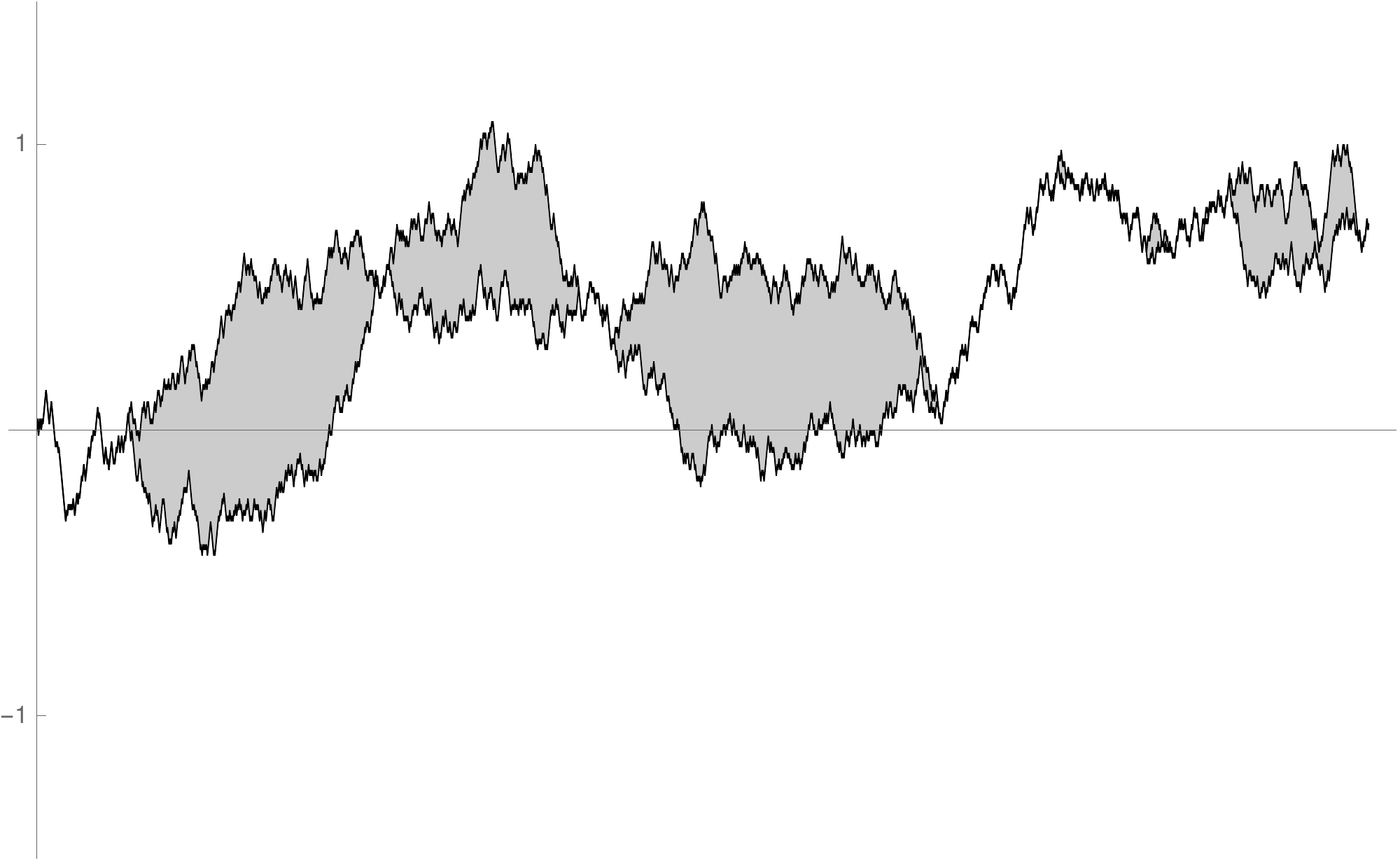}};
\end{scope}
\begin{scope}[xshift=8cm]
\node{\includegraphics[width=7.5cm, ]{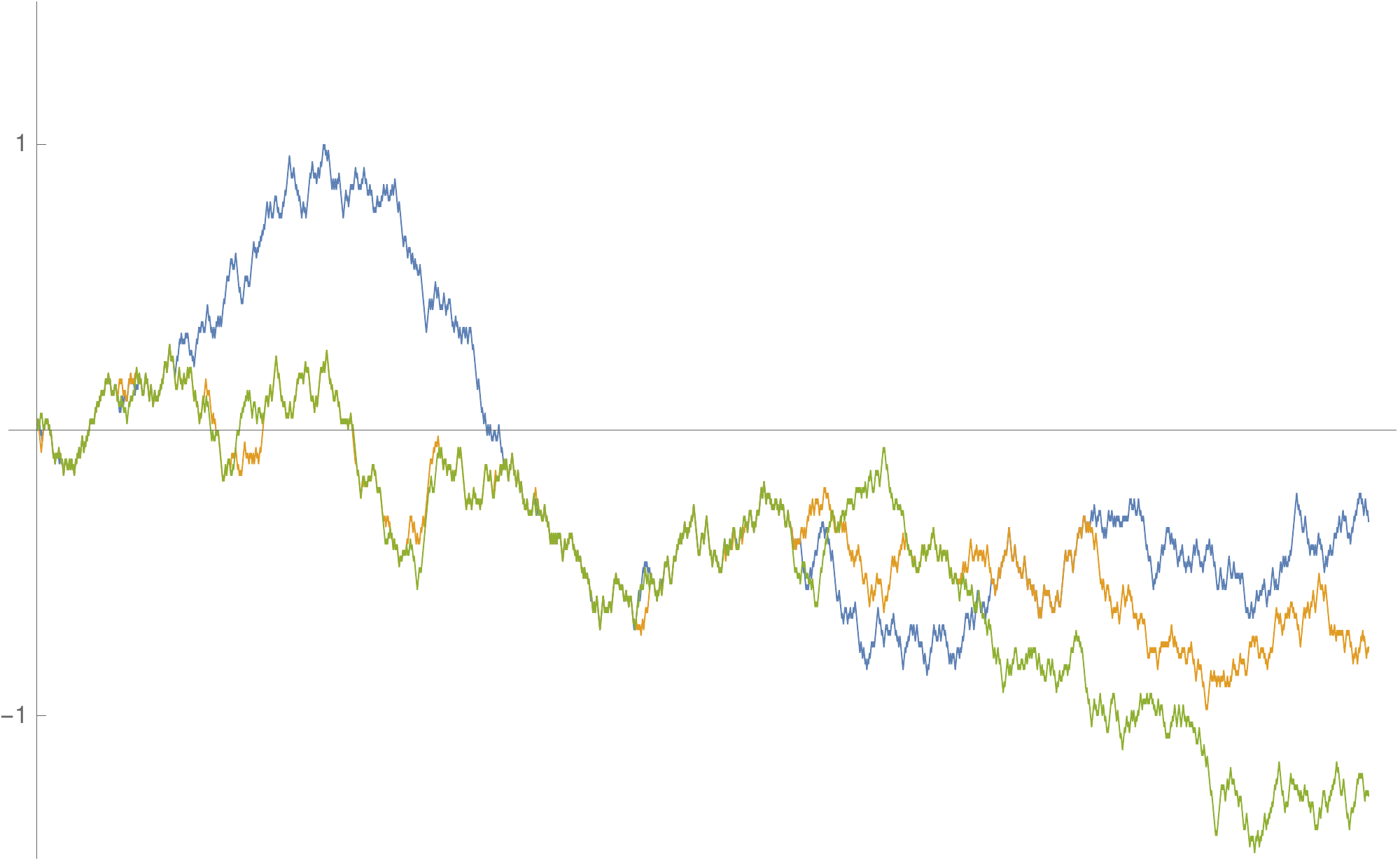}};
\end{scope}
\end{tikzpicture}
\caption{Left panel: Two Brownian motions with sticky interaction. Right 
panel: $3$-point sticky Brownian motions. Not only do the paths stick pairwise, but sometimes all $3$ paths may stick together. Both simulations 
are discretizations of sticky Brownian motions using the beta RWRE with 
$\epsilon=0.02$ (see Section \ref{subsec:integrability}).}
\label{fig:simulations}
\end{figure}
One can use Tanaka's formula to show that equation \eqref{eq:sticky 2 particle eq 2} is equivalent to saying
 \begin{equation}
 |X_1(t)-X_2(t)|-2 \lambda \int_0^t \mathds{1}_{X_1(s)=X_2(s)} ds
 \end{equation}
 is a martingale. 
Howitt and Warren \cite{HowittWarren09a} made this observation and generalized this martingale problem for a family of $n$ particles with pairwise sticky interaction, which we call \emph{$n$-point sticky Brownian motions}. In the most general case, the stickiness behaviour cannot be characterized uniquely by a single parameter $\lambda$. One needs to define for each $k,l\geq 1$ the ``rate'' at which a group of $k+l$ particles at the same position will split into two groups of respectively $k$ and $l$ coinciding particles. Following the notations in \cite{HowittWarren09a, stochasticflowsinthebrownianwebandnet, brownianwebandnet} this rate is denoted 
 $$ \binom{k+l}{k} \theta(k,l).$$

Furthermore, we impose that the law of  $n$-point sticky Brownian motions are consistent in the sense that any subsets of $k$ particles for $k\leq n$ follow the law of the $k$-point sticky Brownian motions. This implies the relation $\theta(k+1,l) + \theta(k,l+1) = \theta(k,l)$. Under this relation, the family of nonnegative real numbers $\theta(k,l)$ can be equivalently (see \cite[Lemma A.4]{stochasticflowsinthebrownianwebandnet}) characterized by a measure $\nu$ on $[0,1]$ such that 
$$ \int_0^1 x^{k-1}(1-x)^{l-1} \nu(dx) = \theta(k,l).$$
 
The following  definition of $n$-point sticky Brownian motions from \cite{brownianwebandnet} is a reformulation of the Howitt-Warren martingale problem \cite{HowittWarren09a}. See Figure \ref{fig:simulations} and Figure \ref{fig:50point} for simulations of $n$-point Brownian motions. 
 \begin{definition}[{\cite[Theorem 5.3]{brownianwebandnet}}] \label{def:martingaleproblem} A stochastic process $\vec{X}(t)=(X_1(t),...,X_n(t))$ started from $\vec X(0)$ will be called $n$-point sticky Brownian motions if it solves the following martingale problem called the Howitt-Warren martingale problem with drift $\beta$ and characteristic measure $\nu$. 
\begin{itemize}
\item{(i)} $\vec{X}$ is a continuous, square integrable martingale.
\item{(ii)} The processes $X_i$ and $X_j$ have covariation process
$$\langle X_i, X_j \rangle(t)=\int_0^t \mathds{1}_{X_i(s)=X_j(s)} ds, \qquad \text{for $t \geq 0$, $i,j=1,...,n$}.$$
\item{(iii)} Consider any $\Delta \subset \{1,...,n\}$. For $\vec{x} \in \mathbb{R}^n$, let
$$f_{\Delta}(\vec{x}):=\max_{i \in \Delta} \{ x_i\} \qquad \text{and} \qquad g_{\Delta}(\vec{x}):=|\{i \in \Delta: x_i=f_{\Delta}(\vec{x})\}|,$$
where $|S|$ denotes the number of elements in a set $S$. Then
$$f_{\Delta}(\vec{X}(t))-\int_0^t \beta_+(g_{\Delta}(\vec{X}(t)) ds$$
is a martingale with respect to the filtration generated by $\vec{X}$, where
$$\beta_+(1):=\beta \qquad and \qquad \beta_+(m):= \beta+2 \int \nu(dy) \sum_{k=0}^{m-2} (1-y)^k = \beta + 2\sum_{k=1}^{m-1} \theta(1,k).$$
\end{itemize}
\end{definition} 
\begin{figure}
\begin{tikzpicture}
\begin{scope}
\node{\includegraphics[width=7.5cm]{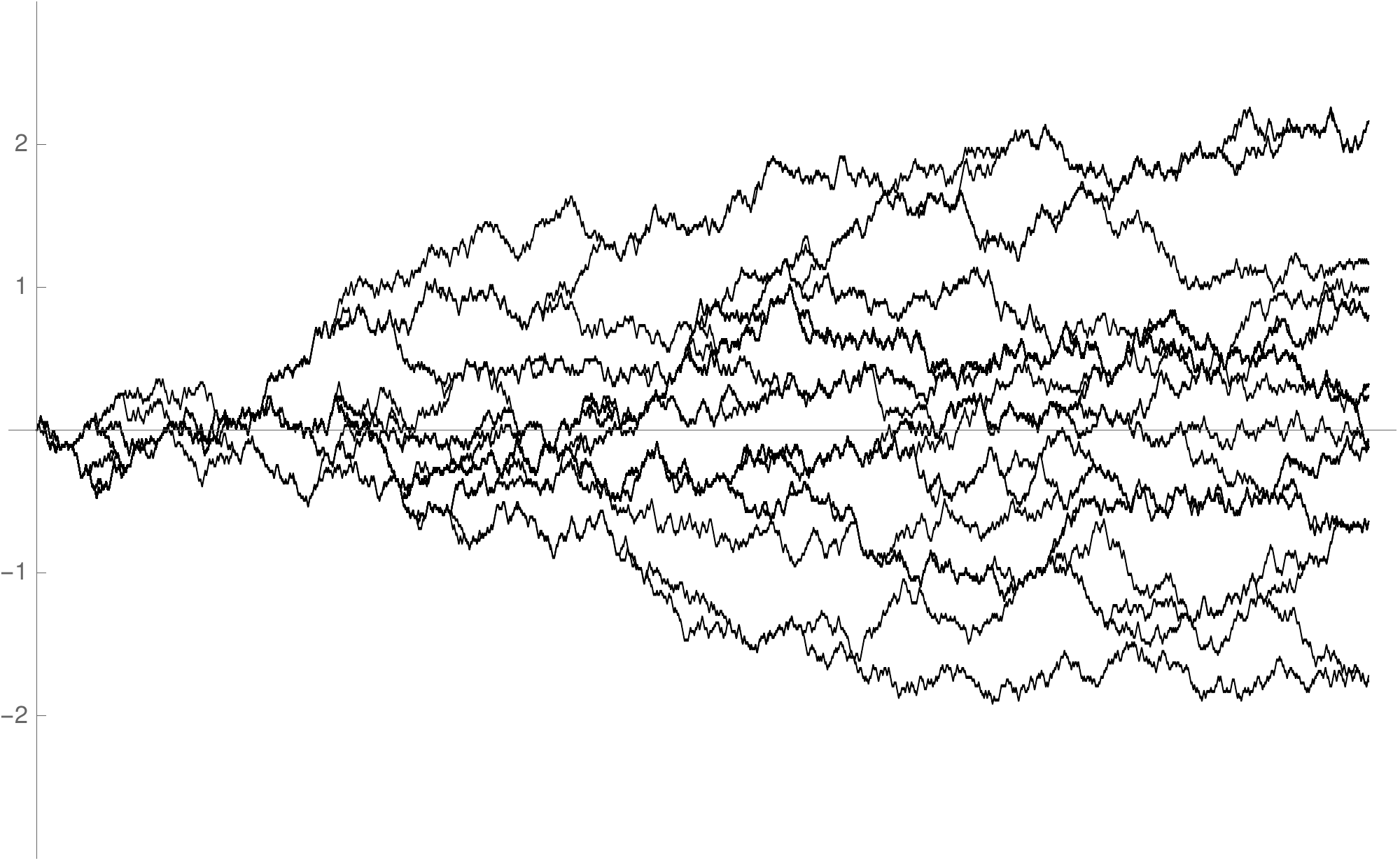}};
\end{scope}
\begin{scope}[xshift=8cm]
\node{\includegraphics[width=7.5cm]{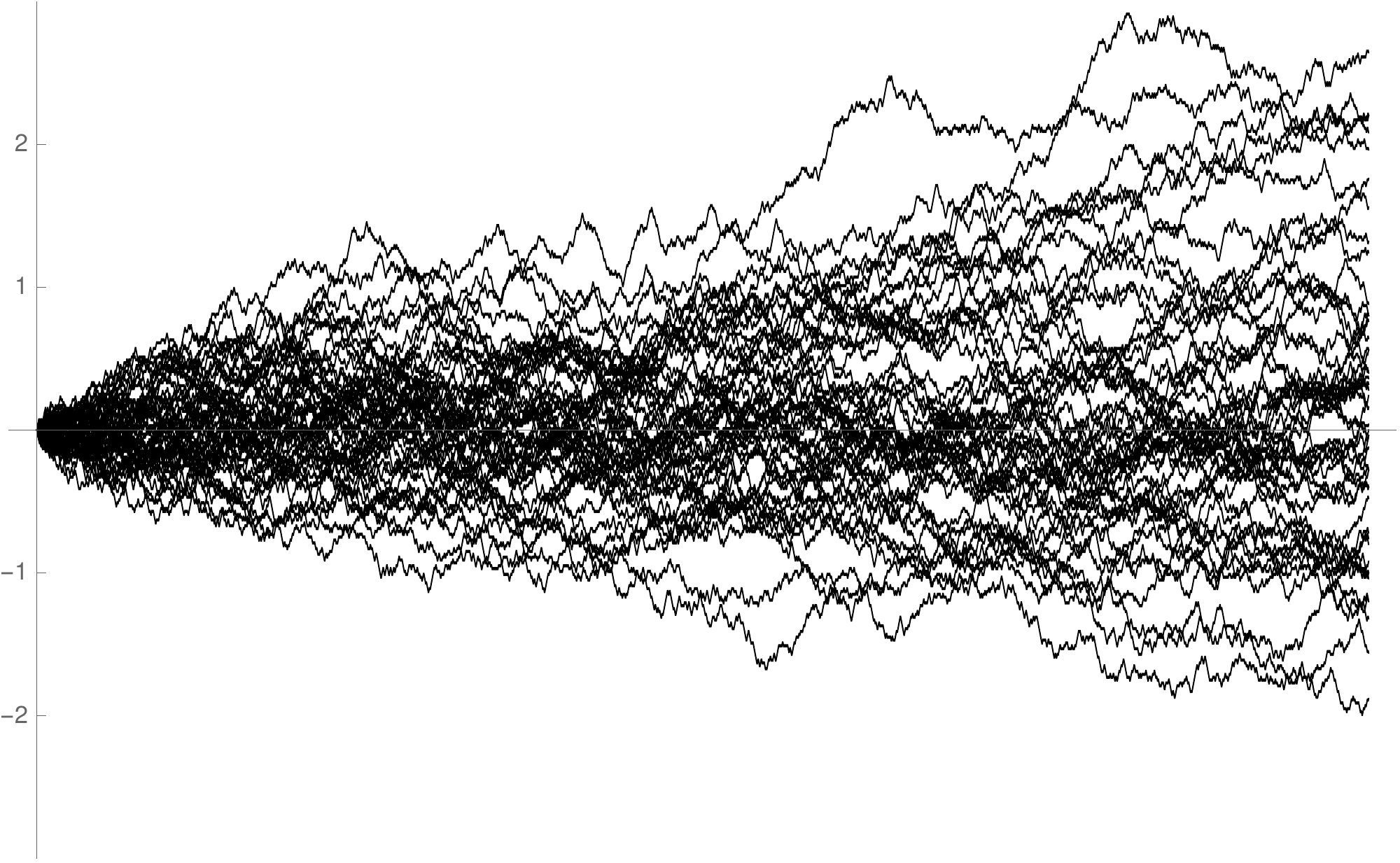}};
\end{scope}
\end{tikzpicture}
\caption{Left panel: $50$ point-sticky Brownian motions using the same discretization as in Fig. \ref{fig:simulations}. Because of the stickiness, the number of trajectories seems much smaller than $50$. Right panel: $50$ independent Brownian motions.}
\label{fig:50point}
\end{figure}
\begin{remark} Definition \ref{def:martingaleproblem} generalizes the definition of $2$-point sticky Brownian motions because each particle marginally evolves as a Brownian motion, and the marginal distribution of any pair of particles is that of a 2 point Brownian motion stickiness parameter $\lambda = \beta_+(2)$. Further, the consistency of the $n$-point motion is clear from property (iii). 
\end{remark} 

We will be interested in a particular exactly solvable case of the Howitt-Warren Martingale problem. 
\begin{definition}
\label{def:simpleHowittWarren}
An $n$-point stochastic process $(B_1(t),...,B_n(t))$ will be called the \emph{$n$-point uniform sticky Brownian motions}  with stickiness $\lambda$ if it solves the Howitt-Warren Martingale problem with drift $\beta=0$ and characteristic measure 
$$\nu(dx)=\mathds{1}_{x \in [0,1]}\frac{\lambda}{2} dx.$$
This choice  corresponds to choosing the fragmentation rates $\theta(k,l) = B(k,l),$
where $B(k,l)=\frac{\Gamma(k)\Gamma(l)}{\Gamma(k+l)}$ denotes the beta function.  
We explain below in Section \ref{subsec:integrability} why this case is exactly solvable.
\end{definition} 

In order to realize the $n$-point sticky Brownian motions as a family of independent random motions in a random environment, we need to introduce the notion of stochastic flows of kernels. Let $\mathcal{B}$ be the Borel $\sigma$-algebra of $\R$. For any $s\leq t$, a \emph{random probability kernel}, denoted  $\msf{K}_{s,t}(x,A)$, for $x\in \R$ and $A\in \mathcal B$, is a measurable function defined on some underlying probability space $\Omega$, such that it defines for each $(x, \omega)\in \R\times \Omega$ a probability  measure on $\R$. In order to interpret this as the random probability to arrive in $A$ at time $t$ starting at $x$ at time $s$, the kernel needs to satisfy the following additional hypotheses.
\begin{definition}[{\cite[Definition 5.1]{brownianwebandnet}}] \label{def:stochastic kernels}
 A family of random probability kernels $(\msf{K}_{s,t})_{s \leq t}$ on $\R$ is called a \emph{stochastic flow of kernels} if the following conditions are satisfied.
\begin{itemize}
\item[(i)]  For any real $s \leq t \leq u$ and $x \in \R$, almost surely $\msf{K}_{s,s}(x, A)=\delta_x(A)$, and 
	$$\int_{\mathbb R} \msf{K}_{s,t}(x,dy) \msf{K}_{t,u}(y,A) dy=\msf{K}_{s,u}(x,A)$$
	 for all $A \in \mathcal{B}$.
\item[(ii)] For any $t_1 \leq t_2 \leq  ... \leq t_k$, the random kernels $(\msf{K}_{t_i,t_{i+1}})_{i=1}^{k-1}$ are independent.
\item[(iii)] For any $s \leq u$ and $t$ real,  $\msf{K}_{s,u}$ and $\msf{K}_{s+t,u+t}$ have the same finite dimensional distributions.
\end{itemize}
\begin{remark}
Additional continuity hypotheses were given in the original definition of a stochastic flow of kernels in \cite{LeJanRaimond04a}, but we will only be interested in Feller processes for which these hypotheses are automatically satisfied. 
\end{remark}

The $n$-point motion of a stochastic flow of kernels is a family of $n$ stochastic processes $X_1,...,X_n$ on $\mathbb{R}$ with transition probabilities given by
\begin{equation}\label{eq:kernel transition probabilities} P(\vec{x},d\vec{y})=\mathbb{E}\left[ \prod_{i=1}^n \msf{K}_{0,t}(x_i,dy_i)\right].
\end{equation}

\end{definition}
Every consistent family of $n$-point motions that is Feller, is the $n$-point motion of some stochastic flow of kernels \cite{LeJanRaimond04a}. Any solution to the Howitt-Warren martingale problem is a consistent family as was noted after Definition \ref{def:martingaleproblem}, and is Feller by \cite{HowittWarren09a}. So any solution to the Howitt-Warren martingale problem is the $n$-point motion of some stochastic flow of kernels. 

\begin{definition} A stochastic flow of kernels whose $n$-point motions solve the Howitt-Warren martingale problem is called a \emph{Howitt-Warren flow}. The stochastic flow corresponding to the special case of the Howitt-Warren martingale problem considered in Definition \ref{def:simpleHowittWarren} (that we called the uniform Howitt-Warren martingale problem), is sometimes called the  \emph{Le Jan-Raimond flow}, after the paper \cite{le2004sticky}, following the terminology used in \cite{brownianwebandnet, stochasticflowsinthebrownianwebandnet}.
\end{definition}

In condition (i) of Definition \ref{def:stochastic kernels}, if we assume that we can move the almost surely so it occurs before choosing $s,t,u$ and $x$, then we can sample all $\K_{s,t}$ and almost surely these kernels define the transition kernels for some continuous space-time markov process. Conditionally on the kernels we can describe the $n$-point motion as independent stochastic processes which evolve according to the transition kernels $\K_{s,t}$. Put simply the $n$-point motion can be seen as continuous space time random motions in a random environment which is given by the set of all transition kernels $\K_{s,t}$. In \cite{stochasticflowsinthebrownianwebandnet} (see also \cite[Section 5]{brownianwebandnet}) it is shown that the change in quantifiers in (i) necessary for this description can be done for Howitt-Warren flows. 
The random environment is explicitly constructed \cite[Section 3]{stochasticflowsinthebrownianwebandnet} (see also \cite[Section 5]{brownianwebandnet}) and consists of a Brownian web \footnote{The Brownian web was introduced in \cite{arratia1980coalescing}, see also \cite{toth1998true}.}  plus a marked Poisson process at special points of the Brownian web \cite{newman2010marking}. The random motions in this environment essentially follow the Brownian web trajectories, except at these special points where they may turn left or right with a random  probability. For Howitt-Warren flows such that  $\int q(1-q)^{-1} \nu(dq) <\infty$ (which is not true for the Le Jan-Raimond flow), the random environment can also be constructed (see \cite[Section 4]{stochasticflowsinthebrownianwebandnet}) using the Brownian net \cite{sun2008brownian, schertzer2009special}.

Note that when starting from a set of particles on the real line and assuming that these particles will branch and coalesce following paths given by either the Brownian net or the Brownian web, the positions of the particles at a later time are given by a Pfaffian point process \cite{garrod2020examples}. This type of evolution of Brownian particles is also related to random matrix theory, in particular the Ginibre evolution \cite{tribe2011pfaffian, tribe2012one,tribe2014ginibre} (the evolution of real eigenvalues in a Ginibre matrix with Brownian coefficients), but these results do not seem to be directly related to the present paper.

\bigskip 

Following \cite{stochasticflowsinthebrownianwebandnet}, we define a measure valued Markov process called the Howitt-Warren process by 
$$ \rho_t(dy)  = \int \rho_0(dx) \msf{K}_{0,t}(x,dy).$$
It describes how a measure on the real line is transported by the Howitt-Warren flow. 
We also define a function valued Markov process called the dual smoothing process by 
\begin{equation}
 \zeta_t(x) = \int \msf{K}_{-t,0}(x, dy)\zeta_0(y). 
 \label{eq:defdualsmoothingprocess}
\end{equation}
This is a continuous analogue of the random average process \cite{balazs2006random}. 
For any fixed $t$, the processes $\rho_t$ and $\zeta_t$ are related via the equality in distribution (called duality in \cite{stochasticflowsinthebrownianwebandnet})
$$  \int \zeta_0(x)\rho_t(dx) = \int \zeta_t(x) \rho_0(dx). $$ 
Note that a different and stronger form of Markov (self-) duality was  investigated in \cite{CarinciGiardinaRedig17}  and applied to characterize the distribution of $2$-point sticky Brownian motions. The result was restricted to $2$-point motions and it is not clear if it translates in terms of stochastic flows of kernels. 

The dual smoothing process was shown to lie in the Edwards-Wilkinson universality class \cite{Yu16}, in the sense that for any fixed $x_0\in \R$, 
$$\mathcal Z_n(t,r):= \frac{1}{n^{1/4}}\zeta_{nt}(nx_0+r\sqrt{n})$$
weakly converges as $n$ goes to infinity -- in the sense of finite dimensional marginals -- to an explicit Gaussian process related to the  stochastic heat equation with additive noise.  This result holds under the assumption that at time $t=0$, $\mathcal Z_n(0,x)$ converges to a smooth profile\footnote{the deterministic part of the initial profile needs to be $C^1$, and \cite{Yu16} assumes further that its derivative is bounded and H\"older $1/2+\epsilon$. } (to which one may add some Brownian noise). An analogous statement in the discrete setting was proved in \cite{balazs2006random}.

In the sequel, we will study the distribution of the dual smoothing process when $\zeta_0(y) = \mathds{1}_{y>0}$ under a different scaling and we will see that the results are very different: instead of lying in the Edwards-Wilkinson universality class, the model lies in the Kardar-Parisi-Zhang universality class.

\subsection{Results}

Our first result is a Fredholm determinant formula for the Laplace transform of the uniform Howitt-Warren stochastic flow of kernels  $\mathsf{K}_{0,t}(0,[x,\infty))$, or Le Jan-Raimond flow. In terms of the dual smoothing process, this corresponds to considering $\zeta_t(-x)$ with the initial condition $\zeta_0(y) = \mathds{1}_{y>0}$.

First recall the definition of the gamma function
$$\Gamma(z)=\int_0^{\infty} x^{z-1} e^{-x} dx,$$
and the polygamma functions 
 $$\psi(\theta)=\partial_{z} \log \Gamma(z)|_{z=\theta}, \qquad \qquad  \psi_i(\theta)=(\partial_{z})^i \psi(z)|_{z=\theta}.$$

\begin{theorem}\label{th:rmre fredholm det for laplace transform} Let $\mathsf{K}_{0,t}(0,[x,\infty))$ denote the kernel of the uniform Howitt-Warren flow with stickiness parameter $\lambda>0$. For $u \in \mathbb{C}\setminus \R_{>0}$, and $x>0$, we have
\begin{equation}
\mathbb E[e^{u \K_{0,t}(0,[x,\infty))}] =\det(I-\K_u)_{L^2(\mathcal{C})},
\label{eq:laplacetransformfredholmdetrerminant}
\end{equation}
(the R.H.S is a Fredholm determinant, see Definition \ref{def:fred det} below), where 
$$K_{u}(v,v')=\frac{1}{2 \pi \i} \int_{1/2-\i \infty}^{1/2+ \i  \infty} \frac{\pi}{\sin(\pi s)} (-u)^s \frac{g(v)}{g(v+s)} \frac{ds}{s+v-v'},$$
and
$$g(v)=\Gamma(v)\exp\left(\lambda x \psi_0(v)+\frac{\lambda^2 t}{2} \psi_1(v)\right).$$
where $\mathcal{C}$ is a positively oriented circle with radius $1/4$ centered at $1/4$. (It is important that this contour passes through zero at the correct angle. The actual radius of the circle $\mathcal{C}$ does not matter.)
\end{theorem}
\begin{remark}
We use two very different notions of kernels, which are both denoted by the letter $K$. We will reserve the font $\K$ for stochastic flows of kernels, and the usual font $K$ for the kernels of $\mathbb L^2$ operators arising in Fredholm determinants.  
\end{remark}

We reach Theorem \ref{th:rmre fredholm det for laplace transform} by taking a limit of a similar Fredholm determinant formula \cite[Theorem 1.13]{RWRE} for the beta RWRE defined in Section \ref{subsec:integrability}. Theorem \ref{th:rmre fredholm det for laplace transform} is proved in Section \ref{sec:sticky brownian limit}.

We perform a rigorous saddle-point analysis of the Laplace transform formula \eqref{eq:laplacetransformfredholmdetrerminant} to obtain a quenched large deviation principle for the uniform Howitt-Warren stochastic flow.
\begin{theorem}\label{th:LDP} Let  $\lambda>0$ and  $x \geq 1.35$. Let $\msf{K}_{s,t}$ be the kernel of a uniform Howitt-Warren flow. Then we have the following convergence in probability
\begin{equation}
\frac{1}{t} \log \K_{0,t}(0,[xt,\infty))\xrightarrow[t \to \infty]{} -\lambda^2 J(x/\lambda), 
\label{eq:LDPquenched}
\end{equation}
where 
\begin{equation}
J(x)=\max_{\theta\in \R_{>0}} \left\lbrace  \frac{1}{2}\psi_2(\theta)+x \psi_1(\theta)\right\rbrace.
\label{eq:defJ}
\end{equation}
\end{theorem}
The condition $x \geq 1.35$ is technical and is addressed in Remark \ref{rem:technicalconditionontheta}. We expect that the limit holds almost surely. This should follow from subadditivity arguments, though we do not pursue this in the present paper (see \cite{rassoul2013quenched} for an almost sure quenched large deviation principle for discrete random walks).  We emphasize that in Theorem \ref{th:LDP}, the rate function $J(x)$ is expressed explicitly using well-known special functions, which is in contrast with what one would obtain using subadditivity arguments. Another large deviation principle was shown in \cite{DaiHongshuai19} for the empirical distribution of a certain class of $n$-point sticky Brownian motions, but this does not seem to be related to the present Theorem \ref{th:LDP}. 

\begin{remark}The annealed\footnote{In the context of random walks in random environment and directed polymers, the (limiting) quenched free energy or rate function is the limit obtained for almost every environment and the annealed analogues correspond to the same quantities for the averaged environment.} analogue of this large deviation principle just describes the tail behavior of a standard Brownian motion. Indeed, 
$$\frac{1}{t} \log \E[\K_{0,t}(0,[xt,\infty))] =  -x^2/2.$$
It can be easily checked that $\lambda^2 J(x/\lambda)> x^2/2$ which, in the context of directed polymers, means that the model exhibits strong disorder.  
Note that the sign of the inequality is consistent with Jensen's inequality (assuming \eqref{eq:LDPquenched} holds in $L^1$). The inequality becomes an equality in the $\lambda\to\infty$ limit, which corresponds to Brownian motions with no stickiness.  
\end{remark}

 When uniform sticky Brownian motions are viewed as random walks in a random environment,
Theorem \ref{th:LDP} gives a large deviation principle whose rate function is deterministic despite the randomness of the environment. The random variable $\log \K_{0,t}$ does depend on the environment, but its fluctuations are small enough that they are not detected by the large deviation principle. We prove that the model is in the KPZ universality class in the sense that the random lower order corrections to the large deviation principle, or equivalently the fluctuations of $\log \K_{0,t}$, are Tracy-Widom GUE distributed on the $t^{1/3}$ scale.
\begin{theorem} \label{th:tracy widom fluctuations} Let $\msf{K}_{s,t}$ be the kernel of a uniform Howitt-Warren flow with stickiness parameter  $\lambda >0$. Let $0 < \theta<1$. We have
$$\lim_{t \to \infty} \mathbb{P}\left(\frac{\log(\K_{0,t}(0,[x(\theta)  t,\infty))+\lambda^2 J(x(\theta)/\lambda)t}{t^{1/3} \sigma(\theta)} <y\right)=F_{\mathrm{GUE}}(y),$$
where $F_{\mathrm{GUE}}(y)$ is the cumulative density function of the Tracy-Widom distribution (defined below in  \eqref{eq:defTracyWidomdistribution}), and 
\begin{equation}\label{eq:x(theta)} 
x(\theta)=-\frac{\lambda}{2} \frac{\psi_3(\theta)}{\psi_2(\theta)}, \qquad 
\sigma(\theta)=\frac{\lambda^{2/3}}{2^{1/3}} \left(\frac{-1}{2} \psi_4(\theta)-\frac{x(\theta)}{\lambda} \psi_3(\theta) \right)^{\frac{1}{3}}.
\end{equation}

\end{theorem}

Theorem \ref{th:tracy widom fluctuations} comes from applying a rigorous steep descent analysis to the Fredholm determinant in Theorem \ref{th:rmre fredholm det for laplace transform}. The proof is given in Section \ref{sec:asymptotic analysis} with some technical challenges deferred to Section \ref{sec:contours} and Appendix \ref{sec:bounds}. The parametrization of functions $J$ and $\sigma$ arising in the limit theorem via the variable $\theta$ may appear unnatural at this point. It will appear more natural in the proof as $\theta$ is the location of  the critical point used in the steep descent analysis. We expect that there should exist another interpretation of the parameter $\theta$. It should naturally parametrize stationary measures associated with the uniform Howitt-Warren flow, and KPZ scaling theory \cite{spohn2012kpz, krug1992amplitude} would predict the expressions for $J(x)$ and $\sigma(\theta)$ given above. This approach would require to degenerate to the continuous limit the results from \cite{balazs2018large} and we leave this for future investigation (the analogue of parameter $\theta$ in the discrete setting is denoted $\lambda(\xi)$ in \cite[Theorem 2.7]{balazs2018large}).

\begin{remark} \label{rem:technicalconditionontheta} Note that $x(\theta)$ is a decreasing function of $\theta$ and the technical hypothesis $\theta<1$ corresponds to approximately  $1.35 \leq x(\theta)$. Similarly $J(x)$ is an increasing function of $x$ and $\theta<1$ corresponds approximately to  $1.02<J(x(\theta))$. We expect Theorem \ref{th:tracy widom fluctuations} to hold for all $\theta>0$, and Theorem \ref{th:LDP} to hold for all $x>0$, however if $\theta \geq 1$ we pick up additional residues while deforming the contours of our Fredholm determinant during the asymptotic analysis which make the necessary  justifications significantly more challenging. 
\end{remark}  

More generally, we believe that the result of Theorem \ref{th:rmre fredholm det for laplace transform} should be universal and hold for more general Howitt-Warren flows under mild assumptions on the characteristic measure $\nu$. This would be analogous to a conjecture that for discrete polymer models the fluctuations of the free energy are Tracy-Widom distributed as long as the weights of the polymer have finite fifth moments \cite[Conjecture 2.6]{alberts2014intermediate}.
Moreover, based on \cite[Theorem 4.3]{rassoul2014quenched}, we expect that the random variable
$$ \log \K_{0,t}(0,[xt, xt+a)), $$
for any $a>0$, satisfies the same limit theorems as $\log \K_{0,t}(0,[xt, +\infty))$ in Theorem \ref{th:LDP} and Theorem \ref{th:tracy widom fluctuations}, with the same constants (the prediction that the constant $\sigma(\theta)$ should remain the same is suggested by the results of \cite{thiery2016exact}).

\bigskip 
Following \cite{RWRE} we can state a corollary of Theorem \ref{th:tracy widom fluctuations}. In general, tail probability estimates provide information about the extremes of independent samples. In the present case, we obtain that the largest among $n$ uniform sticky Brownian motions fluctuates asymptotically for large $n$ according to the Tracy-Widom distribution. We will see that the result is very different from the case of $n$ independent Brownian motions, as can be expected from the simulations in Figure \ref{fig:50point}. 

\begin{corollary} Let $c \in [1.02,\infty)$, let $x_0$ be such that $\lambda^2 J(x_0/\lambda)=c$, let $\theta_0$ be such that $x(\theta_0)=x_0$, and let $\{B_i(t)\}$ be uniform $n$-point sticky Brownian motions with stickiness parameter $\lambda>0$ and scale $n$ as  $n=e^{ct}$, then
\begin{equation}
\lim_{t \to \infty} \mathbb{P}\left(\frac{\max_{i=1,...,n} \{B_i(t)\} - t x_0}{t^{1/3}  \sigma(\theta_0)/(\lambda^2 J'(x_0/\lambda))} \leq y\right)=F_{\rm GUE}(y).
\end{equation}
\label{cor:LDP}
\end{corollary}

The proof of Corollary \ref{cor:LDP} is very similar to the proof of \cite[Corollary 5.8]{RWRE} and uses the fact that after conditioning on the environment we are dealing with independent motions along with our strong control of the random variable $\K_{0,t}(0,[xt,\infty))$ from Theorem \ref{th:tracy widom fluctuations}. The details of the proof can be found at the end of Section \ref{sec:asymptotic analysis}.
\subsection{Integrability for $n$-point uniform sticky Brownian motions}
\label{subsec:integrability}

In 2013 Povolotsky \cite{povolotsky2013integrability} introduced the $q$-Hahn Boson, a three parameter family of Bethe ansatz solvable discrete zero range processes, computed the Bethe ansatz eigenfunctions, and conjectured their completeness. The $q$-Hahn Boson and its eigenfunctions were further studied in \cite{Corwin14} where a Markov duality with the so-called $q$-Hahn TASEP, an interacting particle system closely related to the $q$-Hahn Boson, was used to compute integral formulas for the $q$-moments and the $q$-Laplace transform of the particle positions. The $q$-Hahn Boson eigenfunctions were also further studied in \cite{borodin2017family,BorodinCorwinPetrovSasomoto15a} where the completeness of eigenfunctions was proved and their Plancherel theory was developed. In \cite{RWRE} a model of random walks in a one dimensional random environment, called the beta RWRE, was introduced as the $q \to 1$ limit of the $q$-Hahn TASEP. All features of the integrability of the model survive in the scaling limit. Uniform sticky Brownian motions are a limit of the beta RWRE and we show in the present article that it inherits as well all the integrability of the $q$-Hahn Boson. Note that the $q$-Hahn Boson fits into the more general framework of stochastic higher spin $6$ vertex models \cite{borodin2017family,borodin2018higher,CorwinPetrov16}, so uniform sticky Brownian motions are also a limit of a stochastic vertex model.

\begin{definition} \label{def:betaRWRE}
The beta random walk in random environment (beta RWRE) depends on two parameters $\alpha>0$ and $\beta>0$. Let $\{w_{(x,t)}\}_{x \in \mathbb{Z}, t \in \mathbb{Z}_{\geq 0}} $ be iid beta distributed random variables with parameters $\alpha, \beta$. Recall that a beta random variable $w$ with parameters $\alpha, \beta>0$ is defined by
$$\mathbb{P}(w \in dx)=\mathds{1}_{x \in [0,1]} \frac{x^{\alpha-1} (1-x)^{\beta-1}}{B(\alpha,\beta)}dx,$$
where $B(\alpha,\beta)=\frac{\Gamma(\alpha) \Gamma(\beta)}{\Gamma(\alpha+\beta)}.$
We will call the values of the random variables $w_{(x,t)}$ for all $x \in \mathbb{Z}, t \in \mathbb{Z}_{\geq 0}$ the random environment.

Given a random environment, we begin $k$ independent random walks $(X_1(t),...,X_k(t))$ from position $\vec{x}_0$. Each random walker has jump distribution
$$\msf{P}(X(t+1)=x+1|X(t)=x)=w_{(x,t)} \qquad \msf{P}(X(t+1)=x-1|X(t)=x)=1-w_{(x,t)}.$$
We will use $\vec{X}^{\vec{x}}(t)=(X_1^{x_1}(t),...,X_k^{x_k}(t))$ to refer to the position of $k$ independent random walks started from $(x_1,...,x_k)$ at time $t$. Unless another initial condition is specified, $\vec{X}(t)=(X_1(t),...,X_k(t))$ will refer to the position of $k$ random walkers started from the origin. 

We use the symbol $\mathsf{P}$ with bold font for the quenched probability measure on paths, which is obtained by conditioning on the environment. Similarly we used the same fonts for the quenched probability kernels $\mathsf{K}$ which describe transition probabilities after conditioning on the environment. The usual symbols $\mathbb P$ (resp. $\mathbb E$) will be used
to denote the measure (resp. the expectation) on the environment.
\end{definition}

Note that any single trajectory of the beta RWRE is just a simple random walk and the random environment has no effect. However, if we consider multiple paths on the same environment, they are correlated by the environment. In particular, they do not behave as simple random walks when they meet.

We consider now the continuous limit of the model. If we simply rescale space and time diffusively, trajectories become Brownian motions $\P$-almost-surely \cite{rassoul2005almost}. Moreover, $\vec X(t)$ converges to a family of independent Brownian motions and the effect of the environment has vanished in the limit. In order to keep a dependence on the environment, we need to rescale the weights $w_{(x,t)}$ so that two paths at the same location have a high probability of staying together. This will be the case if $w_{(x,t)}$ is close to either $0$ or $1$ with high probability, which, for a beta distributed random variable, happens when both parameters go to $0$. More precisely, choose a positive parameter $\lambda$ and set $\alpha_{\epsilon}=\beta_{\epsilon}=\lambda \epsilon$. We will be interested in the process $\vec{X}_{\epsilon}(t)=(X_{1,\epsilon}(t),...,X_{k, \epsilon}(t))$, which is obtained as the particle positions at time $t$ of $k$ random walkers in a beta distributed random environment with parameters $\alpha_{\epsilon}, \beta_{\epsilon}$ started from the origin.  

\begin{lemma} \label{lem:beta rwre to sticky Brownian motions}
As $\epsilon \to 0$, the $n$-point beta random walk in random environment $\left(\epsilon \vec{X}_{\epsilon}(\epsilon^{-2}t)\right)_{t\geq 0}$ with parameters $\alpha_{\epsilon}=\beta_{\epsilon}=\lambda \epsilon$ weakly converges to an $n$-point uniform sticky Brownian motions with stickiness parameter $\lambda$ in the space of continuous functions equipped with the topology of uniform convergence on compact sets. 

\end{lemma}
\begin{proof}
We apply \cite[Theorem 5.3]{brownianwebandnet} with drift $\beta=0$, and $\nu(dx)=\frac{\lambda}{2} \mathds{1}_{[0,1]}dx$.
\end{proof}

In fact random walks in a beta distributed random environment were the first random walk in random environment shown to converge to sticky Brownian motions in \cite{LeJanLemaire04}, though this result was shown on a torus. After reformulating sticky Brownian motions as a martingale problem, Howitt and Warren extended this convergence to random walks in any random environment provided the random variables defining the environment have certain scaling limits \cite{HowittWarren09a,HowittWarren09b}. This theorem was reformulated in \cite{stochasticflowsinthebrownianwebandnet,brownianwebandnet} to arrive at the form used above. 

Now we quote a formula for the quenched probability $\mathsf{P}(X(t)>x)$ in the beta random walk in random environment, where $X(t)$ is the path of a single particle that starts from $0$ at time $0$. This quantity is the analogue of $K_{0,t}(0,[x,\infty))$ in the case of the beta random walk in random environment. It satisfies the following formula

\begin{theorem}[{\cite[Theorem 1.13]{RWRE}}] \label{th:betarwre}
For $u \in \mathbb{C} \setminus \R_{>0}$ and $\alpha, \beta>0$, fix $t \in \mathbb{Z}_{\geq 0}$ and $x \in \{-t,...,t\}$ with the same parity. Then
$$\E[e^{u \mathsf{P}(X(t)>x)}]=\det(I-\mathsf{K}^{\mathrm{RW}}_u)_{\mathbb{L}^2(C_0)},$$
where $C_0$ is a small positively oriented contour that contains $0$ and does not contain the points $-\alpha-\beta$ and $-1$, and $\mathsf{K}_u^{\mathrm{RW}}: \mathbb{L}^2(C_0) \to \mathbb{L}^2(C_0)$ is defined in terms of its kernel
$$\mathsf{K}^{\mathrm{RW}}_u(v,v')=\frac{1}{2 \pi \i} \int_{\frac{1}{2}-\i \infty}^{\frac12+\i \infty} \frac{\pi}{\sin(\pi s)} (-u)^s \frac{g^{\mathrm{RW}}(v)}{g^{\mathrm{RW}}(v+s)} \frac{ds}{s+v-v'},$$
where
$$g^{\mathrm{RW}}(v)=\left( \frac{\Gamma(v)}{\Gamma(\alpha+v)} \right)^{(t-x)/2} \left( \frac{\Gamma(\alpha+\beta+v)}{\Gamma(\alpha+v)} \right)^{(t+x)/2} \Gamma(v).$$
\end{theorem}
Theorem \ref{th:betarwre} is the starting point for our study of the uniform sticky Brownian motion in this paper, in particular Theorem \ref{th:rmre fredholm det for laplace transform} is derived as a limit of this formula. 

\begin{remark}
There is a sign mistake in \cite[Theorem 1.13]{RWRE}. It reads $\mathbb E[e^{u \mathsf{P}(X(t) \geq x)}]=\det(I+K_u^{\textrm{RW}})_{\mathbb{L}^2(\mathcal{C}_0)},$ but the right hand side should be $\det(1-K_u^{\textrm{RW}})_{\mathbb{L}^2(\mathcal{C}_0)}$. This is corrected in Theorem \ref{th:betarwre}.
\end{remark}

As we have already mentioned, the crucial tool underlying the exact solvability of the beta RWRE is the Bethe ansatz. We will describe now the sense in which $n$-point uniform sticky Brownian motions are also amenable to Bethe ansatz diagonalization. This could lead to another proof of Theorem \ref{th:rmre fredholm det for laplace transform}, though we do not provide, in this paper, the necessary justifications to make this alternative proof complete.  

Let $\K$ be the kernel of a uniform Howitt-Warren flow, and let $\vec{x}\in\R^k$. We define the function 
$$\Phi_t^{(k)}(x_1, \dots , x_k):= \E \Big[\K_{-t, 0}(x_1, [0, +\infty))\dots \K_{-t, 0}(x_k, (0, +\infty)) \Big].$$ 

Note that since the random variables $\K_{-t, 0}(x, (0, +\infty))$ are bounded between $0$ and $1$, so are the mixed moments $\Phi^{(k)}_t(x_1,...,x_k)$. In particular the knowledge of $\Phi^{(k)}$ uniquely  determines their distribution.  For instance, we have for any $u\in \C$ 
\begin{equation}
\E\left[ e^{u \K_{-t, 0}(x, (0, +\infty))} \right] = \sum_{k=0}^{\infty} \frac{u^k}{k!} \Phi^{(k)}_t(x, \dots, x).
\label{eq:LaplacetranformwithPsi}
\end{equation} 
where there are $k$ occurrences of the variable $x$ in the argument above.
\begin{proposition}\label{prop:moment formula sticky bm}
For $x_1\geq \dots \geq x_k$, and  $t> 0$,
\begin{multline}
\Phi^{(k)}_t(x_1, \dots , x_k) =  \\ \int_{\alpha_1+\I\R} \frac{\mathrm d w_1}{2\I\pi} \dots \int_{\alpha_k+\i \R}  \frac{\mathrm d w_k}{2\i \pi} \prod_{1\leqslant A<B \leqslant k} \frac{w_{B}-w_A}{w_B-w_A-w_Aw_B} \prod_{j=1}^k \exp\left( \frac{t \lambda^2 w_j^2}{2} +\lambda x_jw_j \right)\frac{1}{w_j},
\label{eq:formulaforPsit}
\end{multline}
where for $i<j$, $ 0<\alpha_i < \frac{\alpha_{j}}{1+\alpha_{j}}$. The value at $t=0$ should be understood as 
$$\phi_0(x_1,...,x_k)=\lim_{t \to 0^+} \phi_t(x_1,...,x_k).$$
\end{proposition}
Proposition \ref{prop:moment formula sticky bm} is proved in Section \ref{sec:moment formulas}. We also show in Section \ref{sec:limittoKPZ} that $\Phi^{(k)}_t(\vec x)$ converges, under appropriate scaling, to the moments of the stochastic heat equation with multiplicative noise. This suggests that Howitt-Warren stochastic flows weakly converge in the weak noise limit ($\lambda \to +\infty$ with time and space rescaled) to the solution to the KPZ equation.

 One may observe that (see details in Section \ref{sec:commentsBetheansatz}) the right hand side of \eqref{eq:formulaforPsit} satisfies the following heat equation subject to boundary conditions 
\begin{equation} \label{eq:freeevolutionwithboundary}
\begin{cases} \partial_t u(t, \vec x)=\frac{1}{2} \Delta u(t, \vec x), \quad t\geq 0, \vec x\in \R,\\
(\partial_i \partial_{i+1} +  \lambda (\partial_i-\partial_{i+1})) u(t,\vec x)|_{x_i=x_{i+1}}=0. 
\end{cases}
\end{equation}

Proposition \ref{prop:moment formula sticky bm} shows that \eqref{eq:freeevolutionwithboundary} can be solved using coordinate Bethe ansatz, at least for certain initial conditions. We refer to \cite[Section 3.4.1]{corwin2014macdonald} or \cite{BorodinCorwinPetrovSasomoto15a} for background on coordinate Bethe ansatz in a similar context. In general, Bethe ansatz eigenfunctions corresponding to this problem can be parametrized by $k$ complex numbers $z_1, \dots, z_k$ and written as  
\begin{equation}
\Psi_{\vec z}(\vec x)=\sum_{\sigma \in S_k}  \prod_{1 \leq i<j \leq k} \frac{z_{\sigma(i)}-z_{\sigma(j)}-1}{z_{\sigma(i)}-z_{\sigma(j)}} \prod_{j=1}^k e^{-\frac{\lambda x_j}{z_j}}.
\label{eq:betheansatzeigenfunctions}
\end{equation}

\begin{remark} It is natural (see Section \ref{sec:whitenoisedrift}) to associate to \eqref{eq:freeevolutionwithboundary} the following Schr\"odinger type equation on $\R^k$ with point interactions
\begin{equation}
\partial_t v(t, \vec x)=\frac{1}{2} \Delta v(t,\vec x) + \frac{1}{2 \lambda} \sum_{i \neq j}\delta(x_i-x_j) \partial_{x_i} \partial_{x_j}v(t, \vec x).  
\label{eq:trueevolution}
\end{equation}
We expect the operator $\frac{1}{2}\Delta +\frac{1}{2\lambda}\sum_{i\neq j}  \delta(x_i-x_j) \partial_{x_i} \partial_{x_j}$ to be the generator of the $n$-point uniform sticky Brownian motions, though we do not address in the present paper the details necessary to make rigorous sense of this statement.  Note that similar operators appear in the study of turbulence, in particular in Kraichnan's model of passive scalar \cite{bernard1998slow} and connections to sticky Brownian motions have been noticed in the physics literature \cite{gawkedzki2004sticky}. 
\end{remark}

\begin{remark}
Using $\mathbb E[\xi(s,x)\xi(t,y)]  = \delta(t-s)\delta(y-x)$ for a space-time white noise $\xi$, the Schr\"odinger equation  \eqref{eq:trueevolution} is formally satisfied by the moments of the following stochastic PDE (assuming the existence of such an object, see more details in Section \ref{sec:whitenoisedrift})
	\begin{equation}
	\begin{cases}
	\partial_t q(t,x) = \frac{1}{2} \partial_{xx} q(t,x) + \frac{1}{\sqrt{\lambda}}\xi(t,x) \partial_x q(t,x), \\
	q(0,t)=q_0(x).
	\end{cases}
	\label{eq:Kolmogorov}
	\end{equation}
	If $\xi$ was a smooth and Lipschitz  potential, the Kolmogorov backward equation would provide a representation of the solution as 
	$$q(x,t) = \mathsf E[q_0(X_0)\vert X_{-t}=x],$$
	where $X_t$ is the random diffusion
	\begin{equation}
	dX_t = \frac{1}{\sqrt{\lambda}}\xi(X_t,t)dt + dB_t,
	\label{eq:RMRE}
	\end{equation}
	where the Brownian motion $B$ is independent from $\xi$, and $\mathsf E$ denotes the expectation with respect to $B$, conditionally on the environment $\xi$. 
	For a space-time white noise drift, we have not found any rigorous construction in the literature, and the fact that $\xi$ is not smooth introduces three problems. First, when $\xi$ is a white noise equation \eqref{eq:RMRE} is ill-defined. Second, If $\xi$ were regularized to be smooth in space but white in time equation \eqref{eq:RMRE} would be incorrect (This case is studied in \cite{warren2015sticky}). The final problem is explained in Remark \ref{rem:notstrongsolution}.
	
	 Note that the same diffusion \eqref{eq:RMRE} is considered in the physics paper \cite[Equation (2)]{le2017diffusion} by Le Doussal and Thiery and our results are consistent with some of their predictions (if we identify the solution $q(t,x)$ of \eqref{eq:Kolmogorov} with the dual smoothing process (defined in \eqref{eq:defdualsmoothingprocess}) of the Le Jan-Raimond flow $\zeta_t(-x)$). Moreover, if we interpret $\xi$ as a velocity field, \eqref{eq:Kolmogorov} can be seen as an advection-diffusion equation as in Kraichnan's model \cite{kraichnan1968small}, a model of turbulent flow designed to explain anomalous exponents not predicted by Kolmogorov theory of turbulence, we refer to the review articles  \cite{shraiman2000scalar} for physics background or \cite{kupiainen2010lessons} for a more mathematical exposition. Note that the series of physics works  \cite{chertkov1995normal, gawedzki1995anomalous, bernard1998slow, gawedzki1996university,gawedzki2000phase} on Kraichnan's model were part of the motivation for the work of Le Jan and Raimond \cite{le2002integration, LeJanRaimond04a} on stochastic flows.
	\label{rem:relationtorandomdiffusion}
\end{remark}

\begin{remark} \label{rem:notstrongsolution}

Despite the previous remark, one should not define the solution $q(t,x)$ of the stochastic PDE \eqref{eq:Kolmogorov} as the dual smoothing process $\zeta_t(x)$ of the Le Jan-Raimond flow (defined in \eqref{eq:defdualsmoothingprocess}), even though the moments of both quantities satisfy the same evolution equation (see more details in Section \ref{sec:whitenoisedrift}). Indeed, it was proved by Le Jan and Lemaire \cite{le2002noise, le2004sticky} that the noise generated by the Le Jan-Raimond flow of kernels is black, which implies that, if $\xi$ is a space-time white noise, there cannot be a probability space on which $\zeta_t(x)$ is a strong solution to \eqref{eq:Kolmogorov}. \end{remark} 

\begin{remark}
We expect the Bethe ansatz eigenfunctions $\Psi_{\vec z}(\vec x)$ \eqref{eq:betheansatzeigenfunctions} to be orthogonal with respect to a simple inner product and to form a basis of a large subspace of functions on $\R^k$. These properties would in principle allow to solve  \eqref{eq:freeevolutionwithboundary} for a large class of initial data, although we expect concise integral formulas such as \eqref{eq:moment formula sticky bm} only in a handful cases. Proofs of such statements would likely come from degenerating the Plancherel theory \cite{BorodinCorwinPetrovSasomoto15a,BorodinCorwinPetrovSasomoto15b} for the $q$-Hahn Boson Bethe ansatz eigenfunctions. 
\end{remark}

\subsection{Outline of the proofs}

In Section \ref{sec:asymptotic analysis} we begin with a Fredholm determinant formula for the Laplace transform of the random kernel for a uniform Howitt-Warren flow, then apply a rigorous saddle point analysis to show that the large deviation principle for this random kernel has Tracy-Widom corrections. For readability we will delay some details of the arguments to Section \ref{sec:contours} and Appendix \ref{sec:bounds}. Section \ref{sec:contours} is devoted to constructing a contour which is needed for the saddle point analysis in the previous section. This is one of the main challenges in our saddle point analysis and involves a study of the level set of the real part of a certain meromorphic function. Appendix \ref{sec:bounds} provides the bounds necessary to apply dominated convergence to our Fredholm determinant expansions in order to make the saddle point analysis in Section \ref{sec:asymptotic analysis} rigorous. 

In Section \ref{sec:sticky brownian limit} we derive the Fredholm determinant formula for the Laplace transform of the point to half line probability for uniform sticky Brownian motions used in Section \ref{sec:asymptotic analysis} as the limit of a similar formula for the beta RWRE. The argument is straightforward but requires technical bounds based on known asymptotics for the Gamma and PolyGamma functions. The proof is divided into three steps and the idea of the argument can be understood after reading the first step of the proof. The necessary bounds are provided in the latter two steps. 

Section \ref{sec:moment formulas} is independent from the other sections and provides a proof of the mixed moment formulas for the uniform sticky Brownian motions by taking a limit of similar formulas for the beta RWRE. We also explain the relation between
this moment formula and Bethe ansatz, the KPZ equation and the diffusion
\eqref{eq:RMRE}

Appendix \ref{app:A} gives precise bounds on the Gamma and Polygamma function which are necessary for the construction of the contours in our saddle point analysis. 

\subsection{Acknowledgements}
We are greatly indebted to Rongfeng Sun for telling us about the convergence of random walks in random environment to sticky Brownian motions, and asking if one could study large deviation tails of stochastic flows via similar techniques as in \cite{RWRE}. G.B. and M.R. thank Ivan Corwin for many useful discussions at all stages of this project. We also thank Yves Le Jan and Jon Warren for helpful comments on an initial version of the manuscript. G.B. also thanks Emmanuel Schertzer for enlightening  explanations regarding sticky Brownian motions and stochastic flows and Denis Bernard and Pierre Le Doussal for useful discussions. 

G.B. and M.R.  were partially supported by the NSF grant DMS:1664650.  M. R. was partially supported by the Fernholz Foundation's
``Summer Minerva Fellow'' program, and also received summer support from Ivan Corwin's NSF grant DMS:1811143.

\section{Asymptotic analysis of the Fredholm determinant} \label{sec:asymptotic analysis}

The overall goal of this section is to show that for large time, the fluctuations of the log of the kernel of a uniform Howitt-Warren flow converges to the Tracy-Widom distribution (Theorem \ref{th:tracy widom fluctuations}). We first use a trick from \cite{macdonaldprocesses} to access the large time distribution of $\K_{0,t}(0,[x,\infty))$ from its Laplace transform without using Laplace inversion formula. Then we apply the method of steep descent to the Fredholm determinant from Theorem \ref{th:rmre fredholm det for laplace transform} and prove that, in the appropriate scaling limit, it converges to the cumulative density function of the Tracy-Widom distribution.

We first recall the definition of a Fredholm determinant.
\begin{definition}\label{def:fred det} For any contour $\mathcal{C}$ and any measurable function $K: \mathcal{C} \times \mathcal{C} \to \C$, which we will call a kernel, the Fredholm determinant $\det(1+K)_{L^2(\mathcal{C})}$ is defined by
\begin{equation}\label{eq:fredholm det def}
\det(1+K)_{L^2(\mathcal{C})}=1+\sum_{k=1}^{\infty} \frac{1}{k!} \int_{\mathcal{C}^k} \det(K(x_i,x_j))_{1 \leq i,j \leq k} \prod_{i=1}^k dx_i, \end{equation}
provided the right hand side converges absolutely.

The Tracy-Widom distribution is defined by its cumulative density function 
\begin{equation}
F_{\mathrm{GUE}}(x)=\det(I-K_{\Ai})_{L^2(x,\infty)},
\label{eq:defTracyWidomdistribution}
\end{equation}
 where the Airy kernel $K_{\Ai}$ is defined as
$$K_{\Ai}(x,y)=\frac{1}{2 \pi \i} \int_{e^{-\frac{2 \pi \i}{3}} \infty}^{e^{\frac{2 \pi \i}{3}} \infty} d \omega \int_{e^{-\frac{ \pi \i}{3}} \infty}^{e^{\frac{ \pi \i}{3}} \infty} dz \frac{e^{\frac{z^3}{3}-zx}}{e^{\frac{\omega^3}{3}-\omega y}} \frac{1}{(z-\omega)}.$$
In this integral the contours for $z$ and $\omega$ do not intersect. We may think of the integrating $z$ over the contour $(e^{-\frac{\pi \i}{3}}\infty,1] \cup (1,e^{\frac{\pi \i}{3}} \infty)$ and the integral $w$ over the contour $(e^{-\frac{2 \pi \i}{3}} \infty,0] \cup (0,e^{\frac{2 \pi \i}{3}} \infty)$.
\end{definition}

Instead of inverting the Laplace transform in Theorem \ref{th:rmre fredholm det for laplace transform}, we use a standard trick appearing as Lemma 4.1.39 in \cite{macdonaldprocesses} and take a limit of the Laplace transform to obtain the following formula for the point to half line probability of sticky Brownian motions.

\begin{proposition} \label{prop:point to half line fredholm det} Let $K_u(v,v')$ be as defined in Theorem \ref{th:rmre fredholm det for laplace transform}. For $\lambda>0$, $\theta >0$, $t>0$, and arbitrary constants $x(\theta)$, $J(x(\theta))$, $\sigma(\theta)$ depending on $\theta$, if
$\lim_{t \to \infty} \det(I-K_{u_t(y)})_{L^2(\mathcal{C})}$ is the continuous cumulative density function of a random variable, then
$$\lim_{t \to \infty} \mathbb{P}\left(\frac{\log(\K_{0,t}(0,[x(\theta) t,\infty)))+\lambda^2 J(x(\theta)/\lambda)t}{t^{1/3} \sigma(\theta)} <y\right)=\lim_{t \to \infty} \det(I-K_{u_t(y)})_{L^2(\mathcal{C})},$$
where $u_t(y)=-e^{t\lambda^2 J(x(\theta)/\lambda)-t^{1/3} \sigma(\theta) y}$
\end{proposition}

\begin{proof}[Proof of Proposition \ref{prop:point to half line fredholm det}]
Set $x=x(\theta) t$. Then
$$e^{u_t (y)\K_{0,t}(0,[x,\infty))}=\exp\left(-e^{t^{1/3} \sigma(\theta) \left(\frac{t\lambda^2 J(x(\theta)/\lambda)+\log(\K_{0,t}(0,[x(\theta)t,\infty))}{t^{1/3} \sigma(\theta)} -y\right)}\right).$$
Considering the function $f_t(x)=\exp(-e^{t^{1/3} \sigma(\theta)x})$ and keeping in mind that $\sigma(\theta)>0$, we see that $f_t(x)$ is strictly decreasing in $x$, it approaches $0$ as $x \to \infty$ and it approaches $1$ as $x \to -\infty$. We also see that as $t \to \infty$ this function converges to $\mathds{1}_{x<0}$ uniformly on the interval $\mathbb{R} \setminus [-\delta,\delta]$ for any choice of $\delta>0$. 

If we define the $r$ shift $f_t^r(x)=f_t(x-r)$, then 
$$\E[e^{u_t(r)\K_{0,t}(0,[x,\infty))}]=\E\left[f_t^r \left(\frac{t\lambda^2J(x(\theta)/\lambda)+\log(\K_{0,t}(0,[x(\theta) t,\infty))}{t^{1/3} \sigma(\theta)} \right)\right].$$
By Theorem \ref{th:rmre fredholm det for laplace transform}, $\lim_{t \to \infty} \E[e^{u_t(-y)\K_{0,t}(0,[x,\infty))}]=\lim_{t \to \infty} \det(I-K_{u_t(y)})_{L^2(\mathcal{C})}$, and by assumption, this is the continuous cumulative density function of a random variable. Using \cite[Lemma 4.1.39]{macdonaldprocesses}, completes the proof.
\end{proof}

\subsection{Setup}

Most of this Section \ref{sec:asymptotic analysis} will be devoted to proving the following Proposition \ref{pr:convergence to tracy widom}.  Together with Proposition \ref{prop:point to half line fredholm det} it proves Theorem \ref{th:tracy widom fluctuations}.

\begin{proposition}\label{pr:convergence to tracy widom} For $\lambda>0$, $t>0$, $x>0$, and constants $x(\theta), J(x(\theta)), \sigma(\theta)$ from \eqref{eq:x(theta)}, we have
\begin{equation}
\lim_{t \to \infty} \det(I-K_{u_t(y)})_{L^2(\mathcal{C})}=F_{\rm GUE}(y). \nonumber
\end{equation}
\end{proposition}

First we rewrite $K_{u_t(y)}$ in order to apply the method of steep descent. Performing the change of variables $z=s+v$ gives
$$K_{u_t(y)}(v,v')=\frac{1}{2\pi \mathbf{i}} \int_{1/2+\I\mathbb{R}} \frac{\pi}{\sin(\pi (z-v))} e^{(z-v)(t\lambda^2 J(x(\theta)/\lambda)-t^{1/3} \sigma(\theta)y)} \frac{g(v)}{g(z)} \frac{dz}{z-v'}.$$
Here we have used the fact that the contour for $v$ can be made arbitrarily small so that the contour for $z$ can be deformed from $1/2+v+\i \mathbb{R}$ to $1/2+\i \mathbb{R}$ without crossing poles of $\frac{\pi}{\sin(\pi (z-v))}$.
Recall that
$$g(v)=\exp\left(\frac{\lambda^2 t}{2} \psi_1(v)+ \lambda x \psi(v)\right) \Gamma(v),$$
so replacing $x$ by $xt$ gives
$$K_{u_t(y)}(v,v')=\frac{1}{2\pi \mathbf{i}} \int_{1/2+\I\mathbb{R}} \frac{\pi}{\sin(\pi (z-v))} e^{t(h(z)-h(v))-t^{1/3} \sigma(\theta)y(z-v)} \frac{\Gamma(v)}{\Gamma(z)} \frac{dz}{z-v'},$$
where
$$h(z):=\lambda^2 J(x(\theta)/\lambda)z-\frac{\lambda^2 }{2} \psi_1(z)-\lambda x(\theta) \psi(z)=\lambda^2/2 \left[ (\psi_2(\theta) z-\psi_1(z))-\frac{\psi_3(\theta)}{\psi_2(\theta)} (\psi_1(\theta) z -\psi(z))\right]. $$
The definitions of $x(\theta)$, $\sigma(\theta)$ and $J(x)$ in \eqref{eq:x(theta)}, \eqref{eq:defJ} are tailored precisely so that  $$h'(\theta)=h''(\theta)=0.$$  This will allow us to perform a critical point analysis at $\theta$. Recall \eqref{eq:x(theta)} and note that $\frac{1}{2} \psi_2(\theta)+x \psi_1(\theta)$, is maximized at $x(\theta)/\lambda$, so that we may alternatively define $J(x(\theta)/\lambda)$ by 
$$J(x(\theta)/\lambda)=\frac{1}{2} \psi_2(\theta)+\frac{x(\theta)}{\lambda} \psi_1(\theta).$$
Then
\begin{align*}
h'(z)&=J(x(\theta)/\lambda)-\frac{\lambda^2}{2} \psi_2(z)-\lambda x(\theta) \psi_1(z),\\
h''(z)&=-\frac{\lambda^2}{2} \psi_3(z)-\lambda x(\theta) \psi_2(z),
\end{align*}
and one can immediately check that $h'(\theta)=h''(\theta)=0.$
We also have 
$$h'''(z)=-\frac{\lambda^2}{2} \psi_4(z)-\lambda x(\theta) \psi_3(z)=-\frac{\lambda^2}{2}\left( \psi_4(z)-\frac{\psi_3(z)}{\psi_2(z)} \psi_3(z)\right),$$
which means that $2\sigma(\theta)^3 = h'''(\theta)$.
To control the sign of $h'''(\theta)$, we need the following lemma.
\begin{lemma} For any $z>0$,
$$\psi_m(z)^2<\psi_{m+1}(z)\psi_{m-1}(z).$$
\label{lem:polygamma inequality}
\end{lemma}

\begin{proof}

We adapt the proof of \cite[Lemma 5.3]{RWRE}. The integral representation for polygamma functions gives 
\begin{align*}
\psi_m(z)^2&=\int_0^{\infty} \int_0^{\infty} \frac{e^{-zt-zu}}{(1-e^{-t})(1-e^{-u})} u^mt^mdudt,\\
\psi_{m-1}(z)\psi_{m+1}(z)&=\int_0^{\infty} \int_0^{\infty} \frac{e^{-zt-zu}}{(1-e^{-t})(1-e^{-u})} u^{m-1}t^{m+1}dudt.
\end{align*}
Symmetrizing the second formula in $u$ and $t$ gives
$$\psi_{m-1}(z)\psi_{m+1}(z)=\int_0^{\infty} \int_0^{\infty} \frac{e^{-zt-zu}}{(1-e^{-t})(1-e^{-u})} u^{m-1}t^{m-1}\frac{u^2+t^2}{2}dudt.$$
comparing the integrands and using $ab \leq \frac{a^2+b^2}{2}$ gives the result. 
\end{proof}

\begin{lemma}
For all $\theta>0$, $h'''(\theta)>0$. 
\end{lemma}

\begin{proof}
We have $\psi_2(\theta)<0$ for all $\theta>0$,  this reduces the positivity of $h'''(\theta)$ to the fact that $\psi_4(z) \psi_2(z)>\psi_3(z)^2$, which follows from Lemma \ref{lem:polygamma inequality}.
\end{proof}

\subsection{Outline of the steep descent argument}

Before going further we provide a brief outline of the steep descent argument that the rest of this section will make precise. In this outline we will only describe pointwise convergence of the integrand of $K_{u_t}$ to that of $K_{\Ai}$ without justifying convergence for the Kernel itself or for the Fredholm determinant. We will also ignore the contours of the Fredholm determinant $\det(I-K_{u_t})_{L^2(\mathcal{C})}$ and of the integral which defines the kernels. Consider
$$K_{u_t(y)}(v,v')=\frac{1}{2\pi \mathbf{i}} \int_{\mathcal{D}} \frac{\pi}{\sin(\pi (z-v))} e^{t(h(z)-h(v))-t^{1/3} \sigma(\theta)y(z-v)} \frac{\Gamma(v)}{\Gamma(z)} \frac{dz}{z-v'},$$
and assume that we can deform the contours $\mathcal{C}$ and $\mathcal{D}$ to $\mathscr{C}$ and $\mathscr{D}$ respectively so they pass through $\theta$ at appropriate angles. Perform the change of variables $v=\theta+\sigma(\theta)^{-1}t^{-1/3}\tilde{v}, v'=\theta+\sigma(\theta)^{-1}t^{-1/3}\tilde{v}', z=\theta+\sigma(\theta)^{-1}t^{-1/3} \tilde{z}.$ We know that $h$ has a double critical point at $\theta$, as $h'(\theta)=h''(\theta)=0$ so we Taylor expand and use the large $t$ approximations
$$h(\theta +t^{-1/3} \o{z}) \to h(\theta)+\frac{\tilde{z}^3}{3}, \qquad  \frac{t^{-1/3}\pi}{\sin(\pi (z-v))} \to \frac{1}{\tilde{z}-\tilde{v}}, \qquad \frac{\Gamma(v)}{\Gamma(z)} \to 1.$$
So our kernel becomes
$$K_{(y)}(\tilde{v},\tilde{v}')= \frac{1}{2 \pi \mathbf{i}} \int_{e^{-\frac{ \pi \i}{3}} \infty}^{e^{\frac{\pi \i}{3}} \infty} \frac{e^{\tilde{z}^3/3-y\tilde{z}}}{e^{\tilde{v}^3/3-y \tilde{v}}} \frac{d \tilde{z}}{(\tilde{z}-\tilde{v})(\tilde{z}-\tilde{v}')}.$$
The Fredholm determinant of this kernel is then reformulated as the Fredholm determinant of the Airy kernel on $L^2(\R)$ using the identity  $\det(1+AB)=\det(1+BA)$  in Lemma \ref{lem:tracywidom}.

This completes the brief formal critical point analysis. The main technical challenge is finding contours $\mathscr{C}$ and $\mathscr{D}$ such that the integrals along these contours have (asymptotically as $t \to \infty$) all of their mass near $\theta$ (see Section \ref{subsec:contours} and Section \ref{sec:contours}). This is made more difficult in our case because $h$ is a function with infinitely many poles and it is difficult to explicitly enumerate its critical points. Once such contours are found, a careful argument is necessary to produce the bounds needed to apply dominated convergence to the integral over $\mathscr{D}$ and to the Fredholm determinant expansion (see Section \ref{subsec:localizing} and Appendix \ref{sec:bounds}).

\subsection{Steep descent contours} \label{subsec:contours}

In order to perform our asymptotic analysis on $\det(I-K_{u_t(y)})_{L^2(C)}$, we need to find contours, such that the real part of $h$ (and therefore the norm of the integrand of $K_{u_t(y)}(v,v')$) can be bounded above. In this section we find such contours for the $z$ variable. The contour for the $v,v'$ variables is more elaborate and will be constructed in Section \ref{sec:contours}. 

Without loss of generality we may restrict our attention to $\lambda=1$ in most of the remainder of the paper due to the fact that $h(z)/\lambda^2$ does not depend on $\lambda$. 

\begin{lemma} \label{lem:steep descent vertical contour}The curve $\mc{D}=\theta+\mathbf{i} \mathbb{R}$ is steep descent at the point $\theta$ with respect to the function $h(z)$. In other words $\partial_y \Re[\theta+\mathbf{i} y]<0$ for $y>0$ and $\partial_y \Re[\theta+\mathbf{i} y]>0$ for $y<0$. 
\end{lemma}

\begin{proof} By definition,
	\begin{align*}
	h(z)&=\lambda^2/2 \left[ (\psi_2(\theta) z-\psi_1(z))-\frac{\psi_3(\theta)}{\psi_2(\theta)} (\psi_1(\theta) z -\psi(z))\right] \\
	&= \lambda^2/2\left[(\psi_2(\theta)-\frac{\psi_3(\theta)}{\psi_2(\theta)} \psi_1(\theta))z-(\psi_1(z)-\frac{\psi_3(\theta)}{\psi_2(\theta)} \psi(z))\right],
	\end{align*}
and
$$h'(z)=\lambda^2/2 \left[ (\psi_2(\theta)-\frac{\psi_3(\theta)}{\psi_2(\theta)} \psi_1(\theta))-(\psi_2(z)-\frac{\psi_3(\theta)}{\psi_2(\theta)}\psi_1(z))\right]. $$
Note that $\partial_y\Re[h(\theta+\mathbf{i} y)]=-\Im[h'(\theta+\mathbf{i} y)].$
$(\psi_2(\theta)-\frac{\psi_3(\theta)}{\psi_2(\theta)} \psi_1(\theta))$ is a positive real by Lemma \ref{lem:polygamma inequality}, and $-\psi_2(\theta)$ is positive, so we have
\begin{align*}
A:=-\psi_2(\theta) \Im[h'(\theta+\mathbf{i} y)]= \Im[\psi_2(\theta)\psi_2(\theta+\mathbf{i} y)-\psi_3(\theta) \psi_1(\theta+\mathbf{i} y)] &>0 \qquad \text{for $y>0$},\\
 &<0 \qquad \text{for $y<0$}.
\end{align*}
These two statements are equivalent because the function is odd in $y$. Below we assume $y>0$. 
For $n \geq 1$, we will use the Polygamma series expansion \eqref{eq:polygamma series expansion}.
 First we note that
 \begin{align*}
\Im[\psi_2(\theta+i y)]&=-2 \sum_{k=0}^{\infty} \frac{-3 (t+k)^2 y+y^3}{((t+k)^2+y^2)^3},\\
\Im[\psi_1(\theta+i y)]&=\sum_{k=0}^{\infty} \frac{-2(t+k)y}{((t+k)^2+y^2)^2}.
 \end{align*}
Using the series expansion,
\begin{align*}
A &=4 \sum_{m,n=0}^{\infty} \frac{1}{(n+\theta)^3} \frac{-3(m+\theta)^2 y+y^3}{((m+\theta)^2+y^2)^3}- 6 \sum_{m,n=0}^{\infty}\frac{1}{(n+\theta)^4} \frac{-2 (m+\theta) y}{((m+\theta)^2+y^2)^2}\\
&=\sum_{m,n=0}^{\infty} \frac{1}{(n+\theta)^3} \frac{-12(m+\theta)^2 y+4y^3}{((m+\theta)^2+y^2)^3}+\frac{1}{(n+\theta)^4} \frac{12 (m+\theta) y}{((m+\theta)^2+y^2)^2} \\
&\geq \sum_{m,n=0}^{\infty} \frac{1}{(n+\theta)^3} \frac{-12(m+\theta)^2 y}{((m+\theta)^2+y^2)^3}+\frac{1}{(n+\theta)^4} \frac{12 (m+\theta) y}{((m+\theta)^2+y^2)^2}=B.
\end{align*}

We will show that $B>0$. 
set
$$T_{n,m}=\frac{1}{(n+\theta)^3} \frac{-12(m+\theta)^2 y}{((m+\theta)^2+y^2)^3}+\frac{1}{(n+\theta)^4} \frac{12 (m+\theta) y}{((m+\theta)^2+y^2)^2},$$
so $B=\sum_{n,m=0}^{\infty} T_{n,m}.$
We will prove the following claims for arbitrary $y>0$ and $\theta>0$:
\begin{enumerate}
\item For $0 \leq n \leq m$, $T_{n,m} > 0$. 
\item For $0 \leq n \leq m$, then either $\frac{T_{n,m}}{T_{m,n}}$ is positive, or $\left| \frac{T_{n,m}}{T_{m,n}} \right| \geq 1.$
\end{enumerate}
Together these claims imply that if $n \leq m$, then $T_{n,m}+T_{m,n}>0$, thus $B$ is positive.

In the following two arguments we assume $0 \leq n \leq m$. 
\begin{itemize}[leftmargin=*]
\item{Proof of claim (1):} Let
$$a=\frac{1}{(n+\theta)^4} \frac{12 (m+\theta) y}{((m+\theta)^2+y^2)^2}, \;\; b=\frac{1}{(n+\theta)^3}\frac{-12(m+\theta)^2 y}{((m+\theta)^2+y^2)^3},$$
so that $T_{n,m}=a+b$. $a$ is positive and $b$ is negative, so we need only show that $\left|\frac{a}{b} \right|>1$. We have
$$\left| \frac{a}{b} \right|=\frac{(m+\theta)^2+y^2}{(m+\theta)(n+\theta)} > \frac{(m+\theta)^2}{(m+\theta)(n+\theta)}  = \frac{m+\theta}{n+\theta} \geq 1.$$
which is true because we made the hypothesis that $n\leqslant m$. 

\item{Proof of claim (2):} Setting $m=n+k$ for $k \geq 0$ and simplifying gives
\begin{equation} \frac{T_{n,m}}{T_{m,n}}=\frac{(n+k+\theta)^5 ((n+\theta)^2+y^2)^3(k(n+k+ \theta)+y^2))}{(n+\theta)^5((n+k+\theta)^2+y^2)^3(-k(n+\theta)+y^2)}. \label{eq: simplify Tn,m}\end{equation}
Note that
\begin{equation}\frac{(n+k+\theta)^2}{(n+k)^2}\geq \frac{(n+k+\theta)^2+y^2}{(n+k)^2+y^2}. \label{eq: basic Tn,m inequality}\end{equation}
In the case that $-k(n+\theta)+y^2 > 0$, $\frac{T_{n,m}}{T_{m,n}}$ is positive so there is nothing to show. If $-k(n+\theta)+y^2 \leq 0$, then we have
\begin{equation}\left|-\frac{(k(n+k+ \theta)+y^2)}{(k(n+\theta)-y^2)}\right|=\frac{(n+k+\theta)+y^2/k}{(n+\theta)-y^2/k} \geq \frac{(n+k+\theta)}{(n+\theta)}. \label{eq: Tn,m inequality}\end{equation}
Then \eqref{eq: simplify Tn,m} and \eqref{eq: Tn,m inequality} give
$$\left|\frac{T_{n,m}}{T_{m,n}}\right| \geq \frac{(n+k+\theta)^6}{(n+\theta)^6} \frac{((n+\theta)^2+y^2)^3}{((n+k+\theta)^2+y^2)^3} \geq 1.$$
where the last inequality follows from \eqref{eq: basic Tn,m inequality}. This completes the proof.
\end{itemize}
\end{proof}
Lemma \ref{lem:steep descent vertical contour} will allow us to show that as $t \to \infty$, the kernel $K_{u_t(y)}(v,v')$, which is defined as an integral over $\theta+\i \R$ is the same as the limit as $t \to \infty$ of the same integral restricted to $[\theta- \i \epsilon, \theta+\i \epsilon]$. This is formalized in Lemma \ref{lem:localize kernel}.

\begin{figure}
\begin{tikzpicture}
\node{\includegraphics[width=7.5cm]{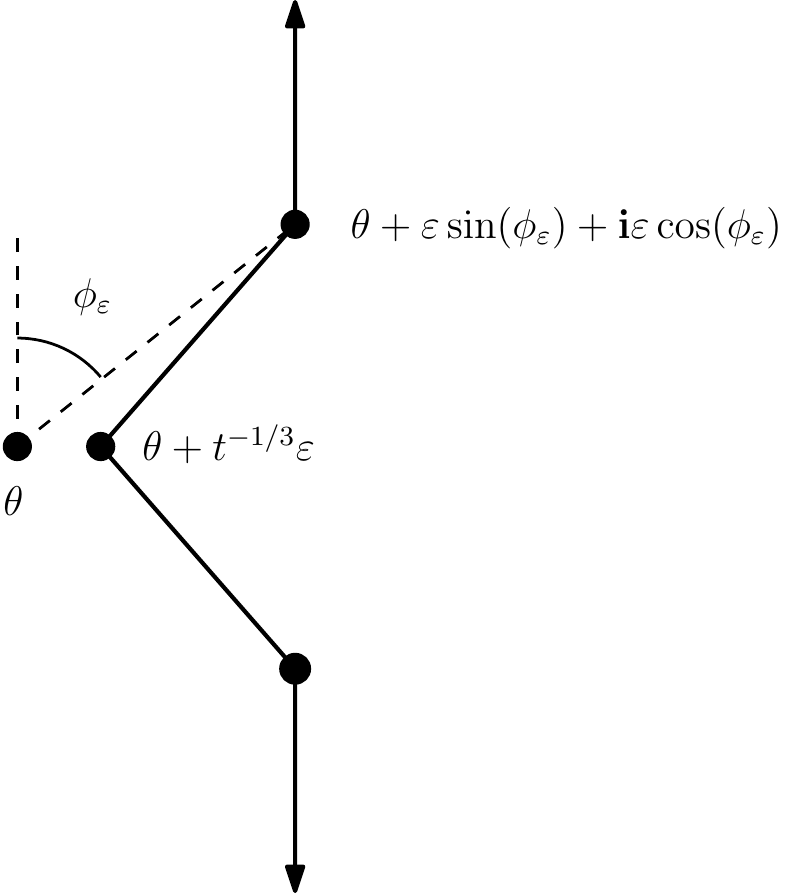}};
\end{tikzpicture}
\caption{The contour $D_{\epsilon}(\phi_{\epsilon})$ is shown in bold, and is oriented in the $+\mathbf{i}$ direction. The dotted line connecting $\theta$ and $\theta+\epsilon \sin(\phi_{\epsilon})+\mathbf{i} \epsilon \cos (\phi_{\epsilon})$ has length $\epsilon$.}
\label{fig:contourangle}
\end{figure}

We will actually use a slight deformation of the contour $\mc{D}$.
\begin{definition} In the following $\epsilon$ is positive, and $\phi_{\epsilon}$ is a small positive angle. Let $\mathcal{D}_{\epsilon}(\phi_{\epsilon})$ be the the union of the diagonal line segment $[\theta+t^{-1/3}\epsilon, \theta + \epsilon e^{\mathbf{i} (\pi-\phi_{\epsilon})})$, and the vertical line $[\theta+\epsilon \sin(\phi_{\epsilon})+\i \cos(\phi_{\epsilon}), \theta+\epsilon \sin(\phi_{\epsilon})+\mathbf{i} \infty)$ along with both their reflections over the real axis, directed from $-\mathbf{i} \infty$ to $\mathbf{i} \infty$. See Figure \ref{fig:contourangle}.
\end{definition}

\begin{lemma} \label{lem:contour vertical deformation} For sufficiently small $\epsilon$ and $\phi_{\epsilon}$, there is an $\eta>0$ such that for any $z \in \mathcal{D}_{\epsilon,t}(\phi_{\epsilon}) \setminus \mathcal{D}_{\epsilon,t}^{\epsilon} (\phi_{\epsilon})$, $\Re[h(z)-h(\theta)]<-\eta$.
\end{lemma}

\begin{proof}[Proof of Lemma \ref{lem:contour vertical deformation}]
Because $h(\o{z})=\o{h(z)}$, it is enough to prove the result in the upper half plane. The idea of this argument is that because $h$ is a holomorphic function in a neighborhood of the contour $\mathcal{D}$, Taylor expanding and choosing $\epsilon$ small allows us to bound the difference between $h'$ on $\mathcal{D}$ and $h'$ on $\mathcal{D}_{\epsilon,t}(\phi_{\epsilon})$ in a large bounded set. We control the difference outside this large ball around $0$ using a rigorous version of Stirling's approximation to control $h'(z)$ for $|\Im[z]|$ very large. 

First we control $h'(z)$ for large $|\Im[z]|$. As $y \to +\infty$, the only term of $h(\theta+\i y)$ that does not go to $0$ is the term containing $\psi(\theta+\i y)$. Lemma \ref{lem:stirling weak} allows us to approximate $\psi(\theta+\i y)$ and gives $h(\theta+\i y) \sim -c \psi(\theta+\i y) \sim -c \log(\theta+\i y)$, where $c=\frac{-\psi_3(\theta)}{\psi_2(\theta)}$ is positive. Thus as $\Im[z] \to + \infty$, $h(z) \to -\infty$ uniformly for $Re[z]$ in a compact set. Thus there is a large $M$ such that for $\Im[z]>M$, $z \in \mathcal{D}_{\epsilon,t}(\phi_{\epsilon}) \setminus \mathcal{D}_{\epsilon,t}^{\epsilon}(\phi_{\epsilon})$, $\Re[h(z)-h(\theta)]<-\eta$. 

Now we will control $h'$ on a bounded set. By Lemma \ref{lem:steep descent vertical contour} $\partial_y \Re[h(\theta+\mathbf{i} y)]<0$
for $y>0$. Thus for some large $M$, on the compact set $y\in [\cos(\phi_{\epsilon}) \epsilon, M]$, $\partial_y \Re[h(\theta+\mathbf{i} y)]$ has some negative minimum. $h''(z)$ is analytic in the compact rectangle with corners $\theta+\mathbf{i}\epsilon \cos(\phi_{\epsilon})$, $\theta+\epsilon e^{\mathbf{i} \phi}$, $\theta+\mathbf{i} M$, $\theta+\epsilon \sin(\phi_{\epsilon}) +\mathbf{i}M$. Thus $|h''(z)|$ is bounded above by some $R$ in this rectangle. Note that $R$ depends only on $\epsilon \cos(\phi_{\epsilon})$ and $M$, and $R$ is increasing in $\cos(\phi_{\epsilon})$. We can choose $\epsilon$ and $\phi_{\epsilon}$ so that $\epsilon \cos(\phi_{\epsilon})$ remains fixed, and $\epsilon \sin(\phi_{\epsilon})$ becomes arbitrarily small. Choosing so that $\epsilon \sin(\phi_{\epsilon}) R <\eta$ guarantees that $\partial_yRe[h(\theta+\sin(\phi_{\epsilon}) \epsilon +\mathbf{i} y)]>0$ for $y \in [\cos(\phi_{\epsilon}) \epsilon, M]$. Because $R$ is increasing $\phi_{\epsilon}$ any smaller choice of $\phi_{\epsilon}>0$ also works. 

Similarly by analyticity of $h$, we can uniformly bound $h'(z)$ on the line segment $[\theta+\mathbf{i} \epsilon \cos(\phi_{\epsilon}), \theta+\epsilon \sin(\phi_{\epsilon})+\mathbf{i} \cos(\phi_{\epsilon})]$, and by Lemma \ref{lem:steep descent vertical contour} we know that $\Re[h(\theta+\mathbf{i} \epsilon \cos(\phi_{\epsilon}))-h(\theta)]<0$. Thus for small enough $\epsilon \sin(\phi_{\epsilon})$, $\Re[h(\theta+\epsilon \sin(\phi_{\epsilon})+\mathbf{i} \epsilon \cos(\phi_{\epsilon}))-h(\theta)]<-\eta$. Again for a particular choice of $\epsilon, \phi_{\epsilon}$, any smaller $\phi_{\epsilon}$ also works. 
\end{proof}

Note that the kernel $K_{u_t(y)}$ is equal to 
\begin{equation}\label{eq:kernelcorrectcontour}
K_{u_t(y)}=\frac{1}{2\pi \mathbf{i}} \int_{\mathcal{D}_{\e}(\phi_{\e})} \frac{\pi}{\sin(\pi (z-v))} e^{t(h(z)-h(v))-t^{1/3} \sigma(\theta)y(z-v)} \frac{\Gamma(v)}{\Gamma(z)} \frac{dz}{z-v'}
\end{equation}
by Cauchy's theorem and the decay of the integrand as $\Im[z] \to \pm \infty$.

\begin{proposition} \label{prop:contour through 0}
There exists a closed contour $\mathcal{C}$ passing through $\theta$ and $0$ , such that for any $\epsilon>0$, there exists $\eta>0$, such that for all $v \in \mathcal{C} \setminus B_{\epsilon}(\theta)$, 
$$\Re[h(\theta)-h(v)]<\eta.$$
\end{proposition}
The proof of Proposition \ref{prop:contour through 0} requires a detailed understanding of the level set $\Re[h(z)]=h(\theta)$. We will defer this proof to Section \ref{sec:contours}.

In the limit $\lim_{t \to \infty} \det(I-K_{u_t(y)})_{L^2(\mc{C})}$, Proposition \ref{prop:contour through 0} will allow us to restrict all contour integrals over $\mc{C}$ in the Fredholm determinant expansion to integrals over $\mc{C} \cap B_{\epsilon}(\theta)$. 

\subsection{Localizing the integrals} \label{subsec:localizing}

We perform the change of variables $v=\theta+t^{-1/3}\o{v}, v'=\theta+t^{-1/3}\o{v}', z=\theta+t^{-1/3} \o{z}.$ For every complex contour $\mathcal{M}$ we will define $\o{\mathcal{M}}=\{z: \theta+t^{-1/3} z \in \mathcal{M} \}$. We will also define the kernel $\o{K}_{u_t}$ by 
$$\o{K}_{u_t}(\o{v},\o{v}')=t^{-1/3}K_{u_t}(\theta+t^{-1/3}v, \theta+t^{-1/3} v'),$$
so that
$$\det(I-\o{K}_{u_t})_{L^2(\o{\mathcal{M}})}=\det(I-K_{u_t})_{L^2(\mathcal{M})}.$$
For any contour $\mathcal{M}$ we define $\mathcal{M}^{\epsilon}$ to be $\mathcal{M} \cap B_\epsilon(\theta)$. Let $K_{u_t(y)}^{\e}(v,v')$ be defined as the right hand side of \eqref{eq:kernelcorrectcontour} with the contour of integration $\mathcal{D}_{\e}(\phi_{\e})$ replaced by the cut off contour $\mathcal{D}_{\e}(\phi_{\e})^{\e}.$

In this section we will use our control of the norm of the integrand of $K_{u_t(y)}(v,v')$ to show that 
$$\lim_{t \to \infty} \det(I-K_{u_t(y)}(v,v'))_{L^2(\mathcal{C})}=\lim_{t \to \infty} \det(I-K^{\e}_{u_t(y)}(v,v'))_{L^2(\mathcal{C}^{\epsilon})}.$$ In this and the next section we will need several bounds in order to apply dominated convergence to the kernel $K_{u_t}(v,v')$ and the Fredholm determinant expansion $\det(I-K_{u_t})_{L^2(\mc{C})}.$ We give these bounds now, but defer most of their proofs to Appendix \ref{sec:bounds}.

\begin{lemma} \label{lem: loop bound}
For $\epsilon$ sufficiently small, $t$ sufficiently large, and $v,v' \in \mathcal{C}\setminus \mathcal{C}^{\epsilon}$, there are constants $R_2,\eta>0$ depending on $\epsilon$ such that
\begin{equation} \label{eq: loop bound 1 lemma} |K_{u_t}(v,v')| \leq R_2 e^{-t \eta/4}.
\end{equation}
For $\epsilon$ sufficiently small, and $t$ sufficiently large, $v \in \mathcal{C}\setminus \mathcal{C}^{\epsilon}$, $v' \in \mathcal{C}$, for the same constants $R_2$ and $\eta$, we have
\begin{equation} \label{eq: loop bound 2 lemma}
 |\o{K}_{u_t}(v,v')| \leq R_2 e^{-t \eta/4}.
\end{equation}
\end{lemma}
This property of the contour $\mathcal{C}$ stated in Proposition \ref{prop:contour through 0} is the main tool necessary to prove Lemma \ref{lem: loop bound}. We defer the proof of Lemma \ref{lem: loop bound} to Appendix \ref{sec:bounds}.
\begin{lemma} \label{lem: non uniform bound on integrand of kernel in epsilon ball}
For $t>1$, and for all sufficiently small $\epsilon>0$, there exists a constant $C_1>0$ such that for $v \in \mathcal{C}^{\epsilon}$ and $z \in \mathcal{D}_{\epsilon,t}^{\epsilon}(\phi_{\epsilon})$,  the integrand of $\o{K_{u_t}^{\epsilon}} (\o{v},\o{v}')$ is bounded above by a positive function of $\o{z}, \o{v}, \o{v}'$ which does not depend on $t$ and whose integral over $\mathcal{D}_{\epsilon,t}^{\epsilon} (\phi_{\epsilon})$ is finite. We also have
$$\o{K_{u_t}^{\epsilon}} (\o{v},\o{v}') \leq C_1 e^{-t \frac{h'''(\theta)}{24} \o{v}^3}.$$
\end{lemma}
\begin{lemma} \label{lem: full kernel bound}
For all sufficiently small $ \epsilon$, for $v,v' \in \mathcal{C}^{\epsilon}$, 
$$\lim_{t \to \infty} (K_{u_t}(v,v') -K_{u_t}^{\epsilon}(v,v'))\to 0.$$
\end{lemma}
\noindent The property of the contour $\mathcal{D}_{\e,t}^{\e}(\phi_{\e})$ stated in Lemma \ref{lem:contour vertical deformation} is the main tool in the proofs of Lemma \ref{lem: non uniform bound on integrand of kernel in epsilon ball} and Lemma \ref{lem: full kernel bound}. We will defer the proofs to Appendix \ref{sec:bounds}.
\begin{lemma} \label{lem:fredholm det expansion bound}For sufficiently small $\epsilon$ and $t>1$, there exists a function $\o{H}_m(\o{v},\o{v}')$ not depending on $t$ such that for all $v \in \mathcal{C}^{\epsilon}$, $\o{H}_m(\o{v},\o{v}') \geq \left| \det(\o{K_{u_t}^{\epsilon}}(\o{v}_i,\o{v}_j)_{i,j=1}^{m}) \right|$ and $\o{H}_m(\o{v},\o{v}') \geq \left| \det(\o{K}_{u_t}(\o{v}_i,\o{v}_j)_{i,j=1}^{m}) \right|$, and
$$1+\sum_{m=1}^{\infty} \frac{1}{m!} \int_{(\o{\mathcal{C}^{\epsilon}})^m} \o{H}_m(\o{v},\o{v}') \leq 1+\sum_{m=1}^{\infty} \frac{1}{m!} \int_{(\mathcal{C}_0)^m} \o{H}_m(\o{v},\o{v}')< \infty.$$
\end{lemma}
The Proof of Lemma \ref{lem:fredholm det expansion bound} uses Lemma \ref{lem: non uniform bound on integrand of kernel in epsilon ball} and Lemma \ref{lem: full kernel bound}. We defer the proof to Appendix \ref{sec:bounds}.

\begin{lemma} \label{lem:localize contour} For any $t>0$ and $\epsilon$ sufficiently small,
$$\lim_{t \to \infty} \det(I-K_{u_t})_{L^2(\mathcal{C})}=\lim_{t \to \infty}\det(I-K_{u_t})_{L^2(\mathcal{C}^{\epsilon})}.$$
\end{lemma}

\begin{proof}

\begin{align}\det(I-K_{u_t})_{L^2(\mathcal{C})}-\det(I-K_{u_t})_{L^2(\mathcal{C}^{\epsilon})}
&= \sum_{m=1}^{\infty} \frac{1}{m!} \int_{\mathcal{C}^m \setminus (\mathcal{C}^{\epsilon})^m} \det\left(K_{u_t}(v_i,v_j)\right)_{i,j=1}^m \prod_{i=1}^m dv_i \nonumber\\
&\leq \sum_{m=1}^{\infty} \frac{1}{m!} \int_{\mathcal{C}^m \setminus (\mathcal{C}^{\epsilon})^m} \left| \det\left(K_{u_t}(v_i,v_j)\right)_{i,j=1}^m \right| \prod_{i=1}^m dv_i . \label{eq:fredholm det difference}\end{align}

By Lemma \ref{lem: loop bound}, for $v_i \in \mathcal{C} \setminus \mathcal{C}^{\epsilon}$, 
$$\o{K}_{u_t}(\o{v}_i,\o{v}_j) \leq R_2 e^{-t \eta/4}.$$
By similar reasoning we can allow $v_j \in \mathcal{C} \setminus \mathcal{C}^{\epsilon}$ without changing the bounds provided by Lemma \ref{lem: full kernel bound} and \ref{lem: non uniform bound on integrand of kernel in epsilon ball}. Thus for $v_i \in \mathcal{C}^{\epsilon}$, $v_j \in \mathcal{C}$, we have
$$\o{K}_{u_t}(\o{v}_i,\o{v}_j) \leq C_1 e^{-t \frac{h'''(\theta)}{24} \o{v}^3}+\eta \leq C_1+\epsilon. $$
Set $R_3=\max[R_2, C_1+\epsilon]$. Then for all $v_i,v_j \in \mathcal{C}$, 
$$\o{K}_{u_t}(\o{v}_i,\o{v}_j) \leq R_3.$$
Using Hadamard's bound with respect to the rows of $|\det(\o{K}_{u_t}(\o{v}_i,\o{v}_j))_{i,j=1}^m|$ with $\o{v}_1 \in \o{\mathcal{C}} \setminus \o{\mathcal{C}^{\epsilon}}$, and $\o{v}_j \in \o{\mathcal{C}}$ for all $j>1$ we obtain
\begin{equation} \label{eq:hadamard bound} |\det(\o{K}_{u_t}(\o{v}_i,\o{v}_j))_{i,j=1}^m| \leq m^{m/2} R_3^{m} e^{-t \eta/4}.\end{equation}

Indeed, because $|\det\left(K_{u_t}(v_i,v_j)_{i,j=1}^m\right)|$ is positive, and unchanged by permuting the $v_1,...,v_m$, we have
\begin{align} \int_{\mathcal{C}^m \setminus (\mathcal{C}^{\epsilon})^m} \left| \det\left(K_{u_t}(v_i,v_j)\right)_{i,j=1}^m \right| \prod_{i=1}^m dv_i 
 & \leq \int_{\mathcal{C}\setminus \mathcal{C}^{\epsilon}} \left( \int_{\mathcal{C}^{m-1}} \left| \det\left(K_{u_t}(v_i,v_j)\right)_{i,j=1}^m \right| \prod_{i=1}^{m-1} dv_i \right) dv_1 \nonumber\\
& \leq \int_{\o{\mathcal{C}}\setminus \o{\mathcal{C}^{\epsilon}}} \left( \int_{\o{\mathcal{C}}^{m-1}} \left| \det\left(\o{K}_{u_t}(\o{v}_i,\o{v}_j)\right)_{i,j=1}^m \right| \prod_{i=1}^{m-1} d\o{v}_i \right) d\o{v}_1 \nonumber\\
&\leq \int_{\o{\mathcal{C}}\setminus \o{\mathcal{C}^{\epsilon}}} \left( \int_{\o{\mathcal{C}}^{m-1}} m^{m/2} R_3^{m} e^{-t \eta/4} \prod_{i=1}^{m-1} d\o{v}_i \right) d\o{v}_1 \nonumber \\
&\leq m^{m/2} (t^{1/3}L R_3)^m e^{-t \eta/4}. \label{eq:sum hadamard bound} \end{align}
In the first inequality we are strictly increasing the set on which we are integrating. In the second inequality we have changed variables from $v_i$ to $\o{v}_i$. In the third inequality we have used \eqref{eq:hadamard bound}. And in the last equality we have used that $\mathcal{C}$ has a finite length $L$, so $\o{\mathcal{C}}$ has length $t^{1/3}L$. 

Thus
\begin{multline}\sum_{m=1}^{\infty} \frac{1}{m!} \int_{\mathcal{C}^m \setminus (\mathcal{C}^{\epsilon})^m} \left| \det\left(K_{u_t}(v_i,v_j)\right)_{i,j=1}^m \right| \prod_{i=1}^m dv_i
 \leq \sum_{m=1}^{\infty}m^{m/2} (t^{1/3}L R_3)^m e^{-t \eta/4}  \\ 
 \leq e^{-t \eta/4}\sum_{m=1}^{\infty}m^{1+m/2} (t^{1/3}L R_3)^m
 \leq e^{-t \eta/4}  (16 t^{1/3}L R_3)^4 e^{2 t^{2/3} (LR_3)^2} \to 0. \label{eq:fredholm det to 0}
 \end{multline}
In the first inequality we used \eqref{eq:sum hadamard bound}. In the second inequality we multiplied each term of the sum by $m$. In the third inequality, we use \cite[Lemma 4.4]{BarraquandRychnovsky18} with $C=(t^{1/3} L R_3)$. Together \eqref{eq:fredholm det difference} and  \eqref{eq:fredholm det to 0} complete the proof.
\end{proof}
\begin{lemma} \label{lem:localize kernel} For $t>0$ and $\epsilon$ sufficiently small,
$$\lim_{t \to \infty} \det(I-K_{u_t})_{L^2(\mathcal{C}^{\epsilon})} =\lim_{t \to \infty} \det(I-K_{u_t}^{\epsilon})_{L^2(\mathcal{C}^{\epsilon})}.$$
\end{lemma}

\begin{proof}
First use Lemma \ref{lem: full kernel bound} to obtain $\lim_{t \to \infty} K_{u_t}^{\epsilon}(v,v')=\lim_{t \to \infty} K_{u_t}(v,v')$, then Lemma \ref{lem:fredholm det expansion bound} allows us to apply dominated convergence to the Fredholm determinant expansion.
\end{proof}

\subsection{Convergence to Tracy-Widom GUE distribution} 

Now we conclude the proof of Theorem \ref{th:tracy widom fluctuations} by identifying the limit of the Fredholm determinant over localized contours from the previous section with the Fredholm determinant expansion of $F_{\mathrm{\rm GUE}}(x)$. 

\begin{proposition}\label{pr: asymptotics} For $t>0$ and $\epsilon$ sufficiently small,
$$\lim_{t \to \infty} \det(I-K_{u_t}^{\epsilon})_{L^2(\mathcal{C}^{\epsilon})}=\det(I-K_{(y)})_{L^2(\mathcal{C}_0)}$$
where
$$K_{(y)}(u,u')=\frac{1}{2 \pi \mathbf{i}} \int_{\mathcal{D}_0} \frac{e^{s^3/3-ys}}{e^{u^3/3-yu}} \frac{ds}{(s-u)(s-u')},$$
and the contours are defined as 
$$\mathcal{D}_0=(e^{- \pi \mathbf{i}/3} \infty, 1) \cup [1, e^{\pi \mathbf{i}/3}\infty),\,\,\,\mathcal{C}_0=(e^{-2 \pi \mathbf{i}/3} \infty, 0) \cup [0, e^{2 \pi \mathbf{i}/3}\infty).$$
\end{proposition}

\begin{proof}
First recall that $\det(I-K_{u_t}^{\epsilon})_{L^2(\mathcal{C}^{\epsilon})}=\det(I-\o{K_{u_t}^{\epsilon}})_{L^2(\o{\mathcal{C}^{\epsilon}})}$. We have the following pointwise limits in $\o{v}, \o{v}', \o{z}$:

\begin{align}\frac{t^{-1/3} \pi}{\sin(\pi(t^{-1/3}(\o{z}-\o{v}))} &\xrightarrow{t \to \infty} \frac{1}{\o{z}-\o{v}}, \\
e^{t(h(z)-h(v))} &\xrightarrow{t \to \infty} e^{\frac{h'''(\theta)}{6} \o{z}^3-\frac{h'''(\theta)}{6} \o{v}^3},\\
\frac{\Gamma(\theta+t^{-1/3} \o{v})}{\Gamma(\theta+t^{-1/3} \o{z})} &\xrightarrow{t \to \infty} 0.
\end{align}
Thus
$$\lim_{t \to \infty}\frac{t^{-1/3} \pi \Gamma(\theta+t^{-1/3} \o{v}) e^{t(h(z)-h(v))-\sigma(\theta)y(\o{z}-\o{v})}}{\sin(\pi(t^{-1/3}(\o{z}-\o{v})) \Gamma(\theta+t^{-1/3} \o{z})(\o{z}-\o{v}')} \\
= \frac{e^{\frac{h'''(\theta)}{6} \o{z}^3-\sigma(\theta) y \o{z}}}{e^{\frac{h'''(\theta)}{6} \o{v}^3-\sigma(\theta)y \o{v}}} \frac{d\o{z}}{(\o{z}-\o{v})(\o{z}-\o{v}')}.$$
The left hand side is the integrand of $\o{K_{u_t}^{\epsilon}}(\o{v},\o{v}')$. Lemma \ref{lem: non uniform bound on integrand of kernel in epsilon ball} allows us to use dominated convergence to get
\begin{equation} \label{eq:kernel converges} \lim_{t \to \infty} \o{K}_{u_t}(\o{v},\o{v}') = K'_{(y)}(\o{v},\o{v}'), \end{equation}
where 
$$K'_{(y)}(\o{v},\o{v}')=\int_{\mathcal{D}_0(\phi_{\epsilon})} \frac{e^{\frac{h'''(\theta)}{6} \o{z}^3-\sigma(\theta) y \o{z}}}{e^{\frac{h'''(\theta)}{6} \o{v}^3-\sigma(\theta)y \o{v}}} \frac{d\o{z}}{(\o{z}-\o{v})(\o{z}-\o{v}')},$$
$$\mathcal{D}_0(\phi_{\epsilon})=(e^{(-\frac{\pi}{2}+\phi_{\epsilon})\mathbf{i}} \infty,\epsilon) \cup [\epsilon, e^{(\frac{\pi}{2}-\phi_{\epsilon})\mathbf{i}} \infty).$$

The real part of $\frac{h'''(\theta)}{6} \o{z}^3$ is negative when $z=e^{i \phi}$ with $\phi \in [\frac{\pi}{2} -\phi_{\epsilon}, \pi/3] \cup [-(\frac{\pi}{2} -\phi_{\epsilon}, -\pi/3]$, so we can deform the contour $\mathcal{D}_0(\phi_{\epsilon})$ to the contour $\mathcal{D}_0$ without changing the value of $K'_{(y)}$. After performing this change of contour and the change of variables $s=\sigma(\theta) \o{z}, u=\sigma(\theta) \o {v}, u'=\sigma(\theta) \o{v}'$, where $\sigma(\theta)=(h'''(\theta)/2)^{1/3}$, we have
\begin{equation} \label{eq:kernel equality} K_{(y)}'(\o{v},\o{v}')=\sigma(\theta)K_{(y)}(u,u').\end{equation}

Note that 
\begin{equation}\label{eq:indicator function fredholm determinant} \det(I-\o{K_{u_t}^{\epsilon}})_{L^2(\o{\mathcal{C}^{\epsilon}})}=\det(I-\mathds{1}_{\o{v}\leq t^{1/3} \epsilon}(\o{v}) K_{u_t}^{\epsilon}(\o{v},\o{v}')  \mathds{1}_{\o{v'} \leq t^{1/3} \epsilon}(\o{v}'))_{L^2(\mathcal{C}_0)}. \end{equation}

By Lemma \ref{lem:fredholm det expansion bound} we can apply dominated convergence to the Fredholm determinant expansion on the right hand side of \eqref{eq:indicator function fredholm determinant}. Along with \eqref{eq:kernel converges} and \eqref{eq:kernel equality} we have

$$\lim_{t \to \infty} \det(I-K_{u_t}^{\epsilon})_{L^2(\mathcal{C}^{\epsilon})}=\det(I-K_{(y)})_{L^2(\mathcal{C}_0)}.$$
\end{proof}
\begin{lemma} \label{lem:tracywidom} For all $y \in \mathbb{R}$,
$$\det(I-K_{(y)})_{L^2(\mathcal{C}_0)}=\det(I-K_{\Ai})_{L^2(y,+\infty)}.$$ where $K_{\Ai}$ is defined in  Definition \ref{def:fred det}.
\end{lemma}
\begin{proof}
We apply \cite[Lemma 8.6]{BorodinCorwinFerrari14}. 
\end{proof}
This reformulation is common in asymptotic analyses of Fredholm determinants. We are now able to conclude. 
\begin{proof}[Proof of Proposition \ref{pr:convergence to tracy widom}]
Together Lemma \ref{lem:localize contour}, Lemma \ref{lem:localize kernel}, Proposition \ref{pr: asymptotics}, and Lemma \ref{lem:tracywidom} yield
\begin{multline} \lim_{t \to \infty} \det(I-K_{u_t})_{L^2(\mathcal{C})}=\lim_{t \to \infty}\det(I-K_{u_t})_{L^2(\mathcal{C}^{\epsilon})} =   \lim_{t \to \infty} \det(I-K_{u_t}^{\epsilon})_{L^2(\mathcal{C}^{\epsilon})} =\\ \det(I-K_{(y)})_{L^2(\mathcal{C}_0)}=\det(I-K_{\Ai})_{L^2(y,+\infty)} .\end{multline}

\end{proof}

The proof of Corollary \ref{cor:LDP} is almost identical to the argument used to obtain \cite[Corollary 5.8]{RWRE} from \cite[Theorem 1.15]{RWRE}. We include it here for completeness.

\begin{proof}[Proof of Corollary \ref{cor:LDP}] 
Observe that we can sample $e^{ct}$ uniform sticky Brownian motions $\{B_i(t)\}$ by first sampling the kernels $\K_{s,t}$ and then sampling $e^{ct}$ iid continuous random walks with these kernels as transition probabilities. For any given kernel, the probability that none of the uniform sticky Brownian motions is greater than $r$ is given by 
\begin{equation} \mathsf{P}\left(\max_{i=1,...,\lfloor e^{ct}\rfloor} B_i(t) \leq  r \right)=\left(1-\K_{0,t}(0,[r,\infty))\right)^{\lfloor e^{ct} \rfloor}=\exp \left( \lfloor e^{ct} \rfloor\log(1-\K_{0,t}(0,[r,\infty)) \right). \label{eq:independentwalkers}
\end{equation}
We set $r=xt=tx_0+\frac{t^{1/3} \sigma(\theta_0) y}{\lambda^2J'(x_0/\lambda)} $, and let $\theta_r$ be defined so that $x(\theta_r)=x$.. Because these motions are independent after conditioning on the environment,
\begin{equation}
\P \left(\max_{i=1,...,\lfloor e^{ct} \rfloor} B_i(t) \leq r \right)=\P \left( \frac{\max_{i=1,...,\lfloor e^{ct} \rfloor} B_i(t)-t x_0}{t^{1/3} \sigma(\theta_0)/(\lambda^2J'(x_0/\lambda))} \leq y \right)
\label{eq:rewriteprobability1}
\end{equation}
Use Theorem \ref{th:tracy widom fluctuations} to approximate 
\begin{equation} \log(\K_{0,t}(0,[r,\infty)))=-t \lambda^2 J(x/\lambda)+t^{1/3} \sigma(\theta_r) \chi_t, \label{eq:kernelexpansion1}
\end{equation}
with $\chi_t$ converging weakly a GUE Tracy-Widom distributed random variable as $t \to \infty$. 
Now Taylor expand 
$$\lambda^2J(x/\lambda)=\lambda^2J(x_0/\lambda)+\sigma(\theta_0) t^{-2/3} y+O(t^{-\frac{4}{3}}).$$
We can take the derivative of $x(\theta)$ and apply Lemma \ref{lem:polygamma inequality} to see that $x$ is a decreasing continuous surjective function of $\theta$ from $\R_{>0} \to \R_{>0}$. Thus we can define the inverse map $\theta(x)$ on $\mathbb{R}_{>0}$, and Taylor expand 
$$ \sigma(\theta_r)=\sigma(\theta_0)+\frac{ t^{-2/3}\sigma'(\theta_0)\theta'(x_0) \sigma(\theta_0) y}{\lambda^2 J'(x_0/\lambda)}+O(t^{-\frac{4}{3}}).$$
We can now expand the right hand side of \eqref{eq:kernelexpansion1} as
 
 \begin{equation}
 -t \lambda^2 J(x_0/\lambda)+t^{1/3} \sigma(\theta) \chi_t= -t\lambda^2J(x_0/\lambda)+\sigma(\theta_0) t^{1/3} (\chi_t-y)+O(t^{-1/3})+O(t^{-1/3} \chi_t)
  \label{eq:taylorexpandexponent1}
 \end{equation}
 
 Choosing $x_0$ so that $\lambda^2J(x_0/\lambda)=c$ gives
  \begin{align} \P \left( \max_{i=1,...,\lfloor e^{ct}\rfloor} B_i(t) \leq  r \right)
 &=\E\exp \left( \lfloor e^{ct} \rfloor \log(1-\K_{0,t}(0,[r,\infty)) \right)\\
 &= \E \exp \left( - \lfloor e^{ct} \rfloor \K_{0,t}(0,[r,\infty)) + O(e^{ct} \K_{0,t}(0,[r,\infty))^2) \right)\\
 &=\E \exp \left(-e^{t^{1/3} \sigma(\theta_0) (\chi_t-y)+O(t^{-1/3}(1+\chi_t))+O(K_{0,t}(0,[r,\infty))+O(e^{ct} \K_{0,t}(0,[r,\infty))^2)}
 \right)
 \end{align}
 The second equality is obtained by Taylor expanding the logarithm around $1$. The third equality is obtained by combining \eqref{eq:kernelexpansion1} and \eqref{eq:taylorexpandexponent1}.
 
Now we control the error terms. The random variable $\chi_t$ converges in distribution, so by Slutsky's theorem, $t^{-1/3}(1+\chi_t) \to 0$ in probability. Recall that $\lambda^2 J(x_0/\lambda)=c$ to obtain
\begin{align*}
e^{ct}K_{0,t}(0,[r,\infty))^2&=e^{ct+2 \log K_{0,t}(0,[r,\infty))}=e^{- c t+O(t^{1/3} \chi_t)} =e^{- c t+t^{2/3} O(t^{-1/3} \chi_t)},\\
K_{0,t}(0,[r,\infty))&=e^{\log K_{0,t}(0,[r,\infty))}=e^{-c t+O(t^{1/3} \chi_t)}=e^{-c_t+t^{2/3}O(t^{-1/3}\chi_t)}.
\end{align*}
Since $O(t^{-1/3} \chi_t) \to 0$ in probability, so do both $O(e^{ct}K_{0,t}(0,[r,\infty))^2)$ and $O(K_{0,t}(0,[r,\infty))).$ Combining this with \eqref{eq:rewriteprobability1} and the fact that $\exp(-e^{t^{1/3}x}) \xrightarrow[t \to \infty]{} \mathbbm{1}_{x<0}$, and using bounded convergence completes the proof. 
 \end{proof}

\section{Construction of steep descent contours}
\label{sec:contours}

This section is devoted to constructing the contour $\mathcal{C}$ whose existence is stated in Proposition \ref{prop:contour through 0}, and which is used in the asymptotic analysis of Section 2. The goal is first to study the level set $\Re[-h(z)]=h(\theta)$, show that it contains well behaved paths from $\theta$ to $0$ in the complex plane, and second to take the slight deformation $\Re[-h(z)]=h(\theta)-\epsilon$ and add small segments to a path in this set to arrive at a contour from $\theta$ to $0$ on which we can bound $\Re[-h(z)]$. The first step is the main difficulty.

Arguments of this type are often performed in cases where the function corresponding to our $h$ is a rational function or the log of a rational function and thus has a finite explicit set of critical points and poles \cite{BorodinCorwinGorin16,fromrepresentationtheorytomacdonaldprocesses}. We will see that the infinite set of poles of $h'$, and the fact that we do not explicitly know all zeros of $h'$ both lead to challenges that we overcome through careful use of conservation of the number of paths in the level set of $\Re[h]$ and $\Re[h']$ which enter and leave a any compact set $K$. 

Before studying the level sets, we will need some bounds. Rather than requiring very careful bounds on $\Re[h(z)]$, we instead only find the sign of the derivative of $\Re[h(z)]$ along the real and imaginary axis.

\begin{lemma} \label{lem: imaginary psi2(iy) negative}
For all $y>0$, $\Im[\psi_2(iy)]<0$. 
\end{lemma}

\begin{proof}

We split the proof into 2 cases. For case 1 assume $y > \frac{1}{\sqrt{5}}.$ Applying Lemma \ref{lem:stirling polygamma}
\begin{align*}
\Im[\psi_2(\mathbf{i}y)]&=-2 \Im \left[\frac{1}{2 (\mathbf{i} y)^2}+\frac{1}{2(\mathbf{i}y^3)}+\frac{3}{6(\mathbf{i} y)^4}+R_m^3(\mathbf{i} y) \right]\\
&=-2 \left(\frac{1}{2 y^3}+\Im[R_m^3(\mathbf{i} y)] \right) \leq \frac{-1}{y^3}+\frac{1}{5y^5}.
\end{align*}
Since $y > \frac{1}{\sqrt{5}}$, we have 
$\Im[\psi_2(\mathbf{i} y)] <0.$ as desired. 

For case 2 assume $y \leq \frac{1}{\sqrt{5}}$. Using the zeroth order Laurent expansion of $\psi_2$ around $0$ gives
$$\Im[\psi_2(\mathbf{i} y)] = \Im\left[\frac{-2 \mathbf{i}}{y^3}-2 \zeta(3)+R_2^0(\mathbf{i} y) \right]=\frac{-2}{y^3}+\Im\left[R_2^0(\mathbf{i} y)\right] \leq \frac{-2}{y^3}+3! \zeta(4) y.$$
Where $\zeta(\cdot)$ is the Riemann zeta function. $\frac{-2}{y^3}+3! \zeta(4) y$ has the same sign as $\frac{-2}{y^4}+3! \zeta(4)$, and when $y \leq \frac{1}{\sqrt{5}}$, 
$$\frac{-2}{y^4}+3! \zeta(4) \leq -50 +6 \zeta(4) <0,$$
where we have used $\zeta(4)<2$. Thus we have 
$\Im[\psi_2(\mathbf{i} y)] \leq 0.$ as well in case 2. 
\end{proof}

\begin{lemma} \label{lem: barrier}
We have $\Im[h'(\mathbf{i} y)]<0$ for $y>0$, and $\Im[h'(\mathbf{i} y)]>0$ for $y<0$
\end{lemma}
\begin{proof}
Because $h'(\overline{z})=\overline{h'(z)}$, the two statements in this lemma are equivalent; we will prove the first. Because $\psi_2(\theta)<0$, this is equivalent to showing
$$A=A(\theta,y)=\Im[\psi_2(\theta) \psi_2(\mathbf{i} y)-\psi_3(\theta) \psi_1(\mathbf{i} y)]>0.$$

For $\theta>0$, $\psi_3(\theta)$ is positive and $\Im[\psi_1 (\mathbf{i} y)]$ is negative, so the second term is positive. $\psi_2(\theta)$ is positive, and by Lemma \ref{lem: imaginary psi2(iy) negative} $\Im[\psi_2(\mathbf{i} y)]$ is negative, so the first term is positive
\end{proof}

Let
\begin{equation} \label{eq: p} p_{\theta}(a)=\psi_2(a)-\frac{\psi_3(\theta)}{\psi_2(\theta)} \psi_1(a). \end{equation}
So that $h'(a)=p(\theta)-p(a)$. We will often omit the $\theta$ subscript and simply write $p(a)$.

\begin{lemma} \label{lem: derivative p greater than 0}The function $p$ satisfies $p'(a) >0$ for all $a<\theta$, and $p'(a)<0$ for all $a>\theta$. 
\end{lemma}
\begin{proof}
By Lemma \ref{lem:polygamma inequality} we have $\psi_4(\theta) \psi_2(\theta)-\psi_3(\theta)^2 > 0$ for all $\theta>0$. After dividing by $\psi_2^2(\theta)$, we get
$$\partial_{\theta} \frac{\psi_3(\theta)}{\psi_2(\theta)}=\frac{1}{\psi_2(\theta)} \left( \psi_4(\theta)-\frac{\psi_3(\theta)^2}{\psi_2(\theta)} \right) > 0,$$
so $\frac{\psi_3(\theta)}{\psi_2(\theta)}$ is increasing in $\theta$. This implies
$$\frac{\psi_3(a)}{\psi_2(a)}-\frac{\psi_3(\theta)}{\psi_2(\theta)} \quad \text{is} \quad
\begin{cases} < 0 & \text{ for $a<\theta$,}\\
> 0 & \text{ for $a>\theta$.}
\end{cases}$$ 
Multiplying by the negative term $\psi_2(a)$ gives
$$f'(a)=\psi_3(a)-\frac{\psi_3(\theta)}{\psi_2(\theta)} \psi_2(a) \quad \text{is} \quad
\begin{cases} > 0 & \text{ for $a<\theta$,}\\
< 0 & \text{ for $a>\theta$.}
\end{cases}$$
\end{proof}

\begin{lemma}
The function $a \mapsto \Re[-h(a)]$ is increasing for $a<t$ and decreasing in $a$ for $a>t$.
\end{lemma}
\begin{proof}
 $h'(a)$ and $h(a)$ are real for $a \in \mathbb{R}$, so $\partial_a \Re[-h(a)]=-h'(a)$. From \eqref{eq: p}, we see that
$-h'(a)= p(a)-p(\theta)$. Together with Claim \ref{lem: derivative p greater than 0} this gives that $h'(a)$ is negative for $a<t$ and positive for $a>t$.
%$$-h'(a) \begin{cases} >0 & \text{ for $a<t$,}\\
%<0 &\text{ for $a>t$}.\\
%\end{cases}$$
This completes the proof.
\end{proof}

\subsection{Contour curves and Contour paths}

\begin{figure}
\begin{tikzpicture}
\node{\includegraphics[width=7.5cm]{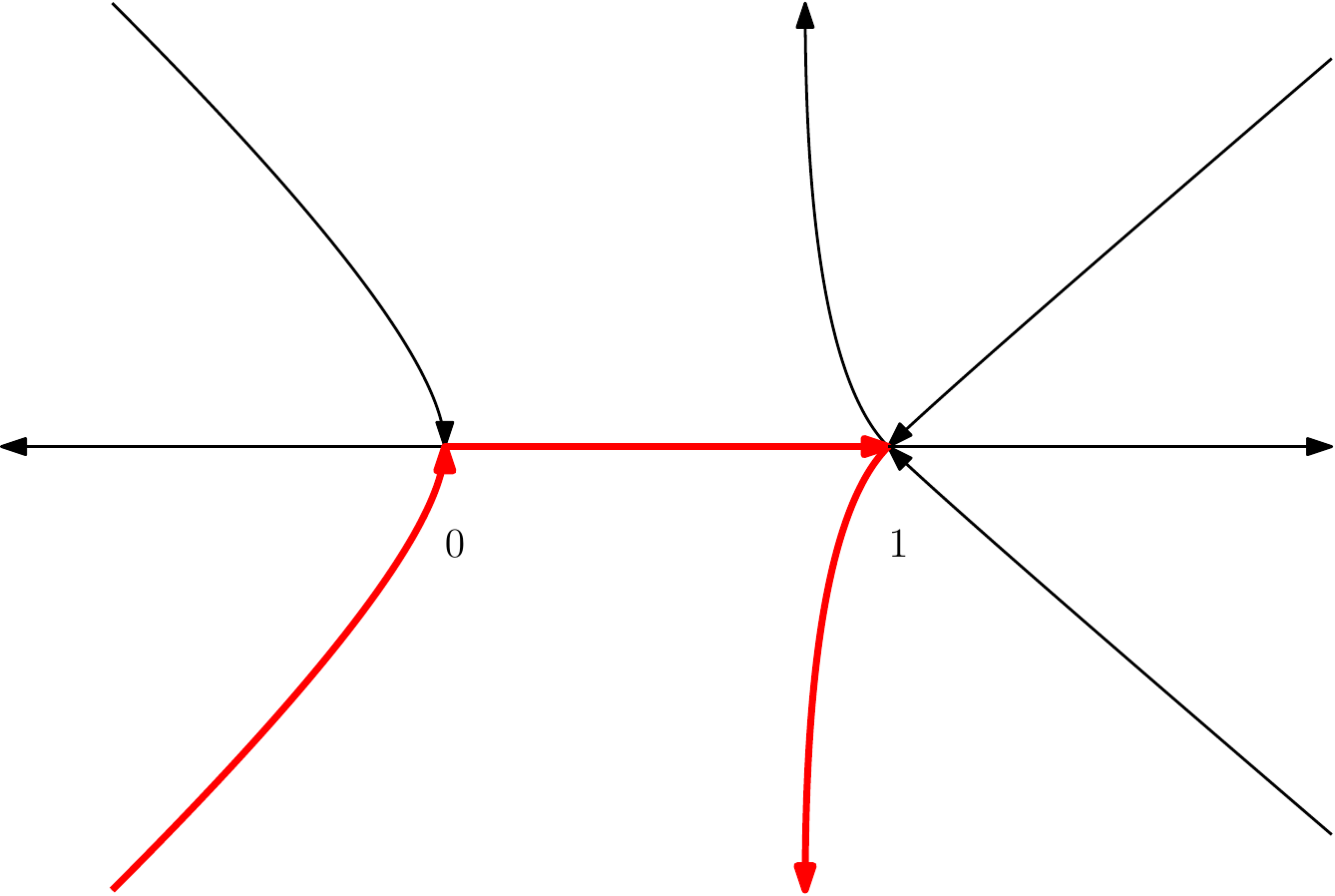}};
\end{tikzpicture}
\caption{An image of the levelset $\Im[f(z)]=0$ for $f(z)=\frac{z^2}{2}-\frac{2 z^3}{3}+\frac{z^4}{4}.$ Because $f'(z)=z(z-1)^2$ we see a critical point at $0$ and a double critical point at $1$. On each contour curve we have drawn an arrow indicating the direction in which $\Re[f(z)]$ is increasing. In thick red we show the contour path that starts at the point $\infty$, and exits its starting point along the contour curve in the lower half plane that connects $\infty$ to $0$. As indicated in the Definition \ref{def:contourpath} at each critical point the next contour curve in the path is immediately counterclockwise to the previous contour curve.}
\label{fig:contourpath}
\end{figure}

Now using the sign of the derivatives of $\Re[h(z)]$ along the real and imaginary axis, we begin a more careful study of the level sets of $\Re[h(z)]$.  First we introduce a helpful way to think about the level set of the real or imaginary part of an arbitrary meromorphic function by defining contour curves and contour paths. 

Let $f$ be a meromorphic function on the complex plane. Let $\gamma=\{z \in \mathbb{C} : \Im[f(z)]=0\}$. Then $\gamma$ can be decomposed as a (potentially infinite) collection of differentiable curves which meet only at critical points and poles of $f$. 

\begin{definition} We will call a maximal connected subset of the level set $\gamma$ that does not contain a critical point or a pole a \emph{contour curve} of $\Im[f(z)]=0$. Contour curves will be differentiable paths with a critical point, a pole, or the point $\infty$ at either end. We assign an orientation to each contour curve so that $\Re[f(z)]$ is an increasing function as we traverse the curve in the positive direction. We will say that a contour curve exits one of its endpoints and enters the other based on this orientation
\end{definition} \label{def:contourpath}
We also define a notion of \emph{contour path}, which connects a pole or the point infinity to another pole or to the point $\infty$ and on which $\Re[f(z)]$ goes from $-\infty$ to $\infty$. To do this we need to make an arbitrary choice of what to do at critical points.
\begin{definition} A contour path of $\Im[f(z)]=0$ is a subset of $\gamma$, which is also a path consisting of a union of contour curves and critical points constructed by the following procedure. Choose a pole or the point $\infty$ to be the starting point of the contourpath. Select one contour curve which exits the starting point. If this curve hits a critical point, then select the critical point and the contour curve leaving the critical point immediately counterclockwise to the previous contour curve. Repeat this step until you reach a pole or until you travel along a contour curve which is unbound in which case we say you reach the point $\infty$ (If a pole or $\infty$ is never reached then repeat this step infinitely many times). The contour path is the union of all contour curves and critical points selected by this procedure. See Figure \ref{fig:contourpath}.

Note that every contour path is a piecewise-differentiable path with endpoints either at a pole or at the point $\infty$. Note also that if two contour-paths do not contain exactly the same set of contour-curves, then they have no contour-curves in common. This is true because each outgoing contour curve of a critical point has only one incoming curve immediately clockwise from it, and each incoming contour curve has only one outgoing curve immediately counter-clockwise from it.

We choose an orientation on the level set $\Im[f(z)]=0$ so that all contour curves and contour paths are directed so that $\Re[f(z)]$ is an increasing function in the chosen direction. Such an orientation exists because we chose each contour path to exit a critical point along a contour curve neighboring the contour curve at which they entered the critical point.
\end{definition}

With these definitions in place we would intuitively like to say that for any bounded set that does not contain a pole of $f$, the number of directed contour paths of $\gamma$ entering the set is equal to the number of directed contour paths leaving the set. We give a more precise definition of "entering a set" then we state this conservation rigorously in Lemma \ref{lem: contourline conservation}.

\begin{definition} We say the contour path $\gamma^i(t)$ (parametrized at unit speed in the positive direction) enters a set $K$ at the point $a$ if there is a $t_a$ and $\epsilon>0$ such that for $t \in (t_a -\epsilon, t_a]$,   $\gamma_i(t) \notin \text{Int}(K)$, and for $t \in (t_a, t_a+\epsilon)$, $\gamma_i(t) \in \text{Int}(K)$. We say a contour path $\gamma^i(t)$ exits $K$ at the point $b$ if there is a $t_b$ and $\epsilon>0$ such that for $t \in (t_b-\epsilon,t_b]$, $\gamma_i(t) \in K$, and for $t \in (t_b, t_b+\epsilon)$,  $\gamma_i(t) \notin K$. Let $[\gamma,K]_{\text{in}}$ be the multiset all of points at which a contour path in $\gamma$ enters $K$ (a point occurs $n$ times in $[\gamma,K]_{\text{in}}$ if $n$ contour paths enter at that point). Let $[\gamma,K]_{\text{out}}$ be the multiset of all points at which a contour path in $\gamma$ exits $K$, similarly counted with multiplicity.
\end{definition}

\begin{lemma} \label{lem: contourline conservation} Let $f$ be a meromorphic function, and let $K$ be a connected compact set, so that no pole of $f$ lies in $K$. If $[\gamma,K]_{\text{in}}$ consists of $n$ points $a_1,...,a_n$, then $[\gamma,K]_{\text{out}}$ consists of $n$ points $b_1	,...,b_n$, so that there is a contour path in the set $\gamma$ from $a_i$ to $b_i$, and $\Re[f(a_i)] \leq \Re[f(b_i)]$ for all $i$. Note the $a_i$s are not distinct if a critical point is on the boundary of $K$, and similarly for the $b_i$'s.
\end{lemma}

\begin{proof} If $K$ contains infinitely many critical points of $f$, then the derivative of $f$ is $0$, in which case the lemma is trivial.

Assume $K \cap \gamma$ has either no critical points, or $1$ critical point of order $r$. At each critical point of order $r$ there are $r$ incoming contour curves of $\gamma$ and $r$ outgoing contour curves of $\gamma$. 

Enumerate all contour paths $\gamma_i$ entering $K$ and pair them so that $\gamma_i$ enters $K$ at the point $a_i$. We define the parametrization of $\gamma_i$ by $|\gamma_i'(t)|=\max[\frac{1}{(\partial_z \Re[f])(\gamma_i(t))},1]$ so that $\partial_t \Re[f(\gamma_i(t))] \geq 1$. $\Re[f]$ is bounded in $K$, so the path $\gamma_i(t)$ eventually leaves $K$. Set $t_i=\inf \{t| \gamma_i(t) \notin K\}$ and set $b_i=\gamma_i(t_i)$, then $b_i$ is the point at which $\gamma_i(t)$ exits $K$. Thus there are at least $n$ exit points $b_1,...,b_n$, and we have traversed $\gamma_i(t)$ in the positive direction to get from $a_i$ to $b_i$, so $\Re[f(a_i)] \leq \Re[f(b_i)].$  To show that there are only $n$ points at which $\gamma$ exits $K$ we can follow the paths in reverse direction (i.e. apply the same argument to $-f$).
To prove the lemma for $m$ critical points in $\gamma \cap K$, we proceed by induction dividing $K$ into one set containing $m-1$ critical points, for which the lemma holds, and one containing $1$ critical point, for which the above argument yields the lemma, then delete all entry and exit points along the shared boundary between the two sets.

\end{proof}

Now we are in a position to see why, for a rational function $g$ with a finite explicit set of critical points and poles, we can find a contour curve in $\{z:\Re[g(z)]=0\}$ from $\theta$ to $\infty$. Up to homotopy there is a finite number of contour curve configurations so that each critical point or pole has the correct number of incident contour curves (twice its order). This means if our sets of critical points and poles are small we can rule out a few possible configurations by controlling $\Re[g(z)]$ or its derivatives until the only remaining configurations have the desired curve. 

We will follow the same general plan for our function $h$, however we will have to address the fact that we are dealing with a nonexplicit set of critical points and an infinite set of poles. The more difficult problem of critical points is addressed in Lemma \ref{lem: no critical point for h} by examining level sets of $h'(z)$ using our conservation property for contour paths and our control of the sign of $\Re[h(z)]$ along the real and imaginary axis. 

\begin{lemma} \label{lem: no critical point for h} The only critical point of $-h$ with nonnegative real part is at $\theta$.
\end{lemma}

\begin{proof}
Recall $-h'(z)=p(z)-p(\theta)$, and $p(\theta)>0$. Thus if $a$ is a critical point of $-h$, then $\Im[p(a)]=0$. We will examine the level set $\Im[p(z)]=0$ in the right half plane. $p(z)$ differs from $h'(z)$ by a real number, so by Lemma \ref{lem: barrier} the level set $\Im[p(z)]=0$ does not intersect the imaginary axis. As $z \to \infty$, in the right half plane, $p(z) \to 0$. $\Re[p(z)]$ is increasing along contour paths of $p$, so no contour path of $p$ can travel from $\infty$ to $\infty$. Thus  every contour path for $\Im[p(z)]=0$ must start or end at a pole, and the only pole of $p(z)$ in the right half plane is at $0$. This pole has highest order term $1/z^3$ near $0$, so there are at most $3$ contour paths of $\Im[p(z)]$ in the right half plane. One contour path begins at $-\infty$ and travels along the real line (directed away from $0$], the other two contour paths are directed toward $0$ with one above the real line and one below the real line. 

The point $\theta$ is a zero of $p$ and a critical point of $p$ with negative second derivative (because $h'''(0)>0)$), so $p$ is equivalent to $-(z-\theta)^2$ near $\theta$. Thus $p$ has contour curves entering $\theta$ along the positive and negative real line, and has contour curves leaving parallel to the positive and negative imaginary axis. By Lemma \ref{lem: derivative p greater than 0}, there exists a contour curve directed from $0$ to $\theta$ along the real axis, and a contour curve directed from $\infty$ to $\theta$ along the real axis. 

We have $p(\overline{z})=\overline{p(z)}$, so it is enough to consider the level set of $\Im[p(z)]=0$ restricted to the upper right quarter plane. In the upper right quarter plane, $p(z) \to 0$ as $|z| \to \infty$ uniformly in $|z|$. Let $\overline{D}_{\theta}$ be a disk centered at $0$ intersected with the upper right quarter plane, with the disk chosen large enough that $\Re[p(z)]< p(\theta)>0$ for all $z \notin D_{\theta}$ in the upper right quarter plane. Let $D_{\theta}$ be the set $\overline{D}_{\theta}$ with an arbitrarily small circle around $0$ removed, so $D_{\theta}$ contains no poles. 

 By Lemma \ref{lem: contourline conservation}, the contour path entering $D_{\theta}$ at $\theta$ must exit $D_{\theta}$ at a point $b$ such that $\Re[p(b)]>\Re[p(\theta)]$. By our choice of $D_{\theta}$ this contour path cannot exit $D_{\theta}$ toward $\infty$, by Lemma \ref{lem: barrier} it cannot exit along the imaginary axis, and by Claim \ref{lem: derivative p greater than 0} it cannot exit along the real axis, because the real axis is contained in the level set of $\Im[p(z)]=0$, and $\theta$ is the only critical point along the real axis. Thus the contour path entering at $\theta$ must exit toward the pole at $0$, so there is a contour path $\alpha(t)$ of $\Im[p(z)]=0$ from $\theta$ to $0$. 

Furthermore the contour path $\alpha(t)$ connecting $\theta$ to $0$ contains no critical points of $h'$ other than $\theta$. We prove this by contradiction. Assume $\alpha(t)$ has a critical point besides $\theta$, then it has finitely many critical points because $\alpha(t)$ is contained in the compact set $D_{\theta}$. Let $z_c$ be the critical point for which $\Re[\alpha(z)]$ is smallest. Let $\overline{A}$ be the compact set enclosed between $\alpha(t)$ and the line segment $[0,t]$. Let $A$ be $\overline{A}$ with an arbitrarily small circle around $0$ removed. One contour line exits $A$ at $z_c$, so by Lemma \ref{lem: contourline conservation} there must be a contour line entering $A$ at a point $z_b$ with $\Re[\alpha(z_b)]<\Re[\alpha(z_c)]$. Because $c$ minimizes $\Re[\alpha(z)]$ over all critical points of $\alpha(t)$, and no critical point occurs along the real axis, we arrive at a contradiction.

We have classified the contour curves of $\Im[p(z)]=0$ in the right half plane as: one contour curve with real part of $p(z)$ increasing from $\theta$ to $0$ along the real line, one contour curve with real part of $p(z)$ decreasing from $\theta$ to $\infty$ along the real line, one contour curve with real part of $p(z)$ increasing from $\theta$ to $0$ above the real line, one contour curve with real part of $p(z)$ increasing from $\theta$ to $0$ below the real line. Any critical point of $-h$ must have $\Im[p(z)]=0$ and $\Re[p(z)]=p(\theta)$. Thus any critical point must be on one of the four contour lines described above or the critical point $\theta$, but every point $z$ on these contour curves has been specified to have $\Re[p(z)]$ either strictly greater than, or strictly less than $p(\theta)$. So $\theta$ is the only critical point of $-h$. 

\end{proof}

Now we can address the simpler problem that $h$ has an infinite number of poles using the conservation of contour paths and the the sign of the derivative of $\Re[h(z)]$ along the imaginary axis. We do so in Lemma \ref{lem: contour curve exists} and prove the existence of a contour curve in $\{z:\Re[h(z)]=h(\theta)\}$ with the desired properties.

\begin{lemma} \label{lem: contour curve exists}
The contour curve $\gamma_1$ for $\Re[h(z)]=h(\theta)$ which exits $\theta$ at angle $\frac{5 \pi}{6}$ enters $0$ at angle $\pi/4$, and the contour curve $\gamma_2$ for $\Re[h(z)]=h(\theta)$ which exits $\theta$ at angle $\frac{\pi}{2}$ crosses the positive imaginary axis. 
\end{lemma}

\begin{proof}
Lemma \ref{lem:stirling polygamma} shows that $\lim_{y \to \infty}\Im[\mathbf{i} h(x+\mathbf{i} y)]=\lim_{y \to \infty} \Re[-\psi(x+ \mathbf{i} y)]=- \infty$, and that this convergence is uniform with respect to $x$ for $x \in [0,\theta]$. Let $C$ be large enough that for all $y>C$, $\Im[\mathbf{i} h(x+\mathbf{i} y)]<\Im[\mathbf{i} h(x+\mathbf{i} y)]<h(\theta)$, and consider the rectangle $\overline{S}=[0,\theta] \times [0,\mathbf{i} C]$ in the complex plane. Let $S$ be $\overline{S}$ with an arbitrarily small open circle around $0$ removed.
Neither $\gamma_1$ nor $\gamma_2$ can cross the line $[\mathbf{i} C, t+\mathbf{i} C]$ because for $z \in [\mathbf{i} C, t+\mathbf{i} C]$, $\Re[h(z)]< h(\theta)$.
Multiplying $h$ by $\mathbf{i}$ and applying Lemma \ref{lem: contourline conservation} tells us that the contour curve $\gamma_1$ enters $S$ at $\theta$, and must exit $S$ at a point $b$ with $\Im[h(b)]>\Im[h(\theta)]=0$. It cannot exit $S$ along $[0,\theta]$, because $\Im[h(t)]=0$ for all $t \in \mathbb{R}$.

Examining the critical point at $\theta$ shows that if we follow $\gamma_1$ away from $\theta$, then $\Re[-h(z)]$ is positive for $z$ immediately to the left of $\gamma_1$ and negative immediately to the right. Thus if this contour curve were to cross the imaginary axis, $\Re[-h]$ would be decreasing in a neighborhood of the intersection. This contradicts Lemma \ref{lem: barrier} so $\gamma_1$ cannot cross the imaginary axis. 

The contour curve $\gamma_2$, is left of the line $\theta+\mathbf{i} \mathbb{R}$. Examining the critical point at $\theta$ shows that if we follow this new contour away from $\theta$, then $\Re[h(z)]$ is positive for $z$ immediately to its right, and negative for $z$ immediately to its left. If this contour were to cross the line $\theta+\mathbf{i} \mathbb{R}$ then $\Re[h(z)]$ would be increasing on this line in a neighborhood of the intersection. This contradicts Lemma \ref{lem:steep descent vertical contour}, so $\gamma_2$ cannot cross the line $\theta+\mathbf{i} \mathbb{R}$.
By Lemma \ref{lem: no critical point for h}, $\theta$ is the only critical point of $h$ in the right half plane, thus the contour line $\gamma_1$ cannot cross $\gamma_2$ to exit $S$ on the right. Thus the only possible place for $\gamma_1$ to exit $S$ is to the pole at $0$. 
$\gamma_2$ cannot cross $\gamma_1$ to reach $(0,t]$, we have already shown that it does not cross $[\mathbf{i} C, t+\mathbf{i} C]$ or $\theta+\mathbf{i} \mathbb{R}$, and no other contour lines leave $0$ into the upper right half plane, so $\gamma_2$ must cross the positive imaginary axis.
\end{proof}

Now we are prepared to prove Proposition \ref{prop:contour through 0} by deforming the contour curve found in Lemma \ref{lem: contour curve exists} so that it lies in the level set $\{z:\Re[h(z)]=h(\theta)-\epsilon\}$.

\begin{proof}[Proof of Proposition \ref{prop:contour through 0}]
Because $h(\o{z})=\o{h(z)}$, it is enough to prove the lemma in the upper half plane. As $z \to \infty$, only one term of $h$ becomes infinite, so $h(z) \sim -c \psi(z) \sim -c \log(z)$ by Lemma \ref{lem:stirling weak}, and as $\Im[z] \to + \infty$, $h(z) \to -\infty$ uniformly for $\Re[z]$ in a compact set. Thus there exists a large $M$ such that for $\Im[z]>M$, $z \in \mathcal{D}_{\epsilon,t}(\phi_{\epsilon}) \setminus \mathcal{D}_{\epsilon,t}^{\epsilon}(\phi_{\epsilon})$, $\Re[h(z)-h(\theta)]<-\eta$. 

 By Lemma \ref{lem:steep descent vertical contour} $\partial_y \Re[h(\theta+\mathbf{i} y)]<0$
for $y>0$. Thus for some large $M$, on the compact set $y\in [\cos(\phi_{\epsilon}) \epsilon, M]$, $\partial_y \Re[h(\theta+\mathbf{i} y)]$ has some negative minimum. The function $h''(z)$ is analytic in the compact rectangle with corners $\theta+\mathbf{i}\epsilon \cos(\phi_{\epsilon})$, $\theta+\epsilon e^{\mathbf{i} \phi}$, $\theta+\mathbf{i} M$, $\theta+\epsilon \sin(\phi_{\epsilon}) +\mathbf{i}M$. Thus $|h''(z)|$ is bounded above by some $R$ in this rectangle. Note that $R$ depends only on $\epsilon \cos(\phi_{\epsilon})$ and $M$, and $R$ is increasing in $\cos(\phi_{\epsilon})$. We can choose $\epsilon$ and $\phi$ so that $\epsilon \cos(\phi_{\epsilon})$ remains fixed, and $\epsilon \sin(\phi_{\epsilon})$ becomes arbitrarily small. Choosing so that $\epsilon \sin(\phi_{\epsilon}) R <\eta$ guarantees that $\partial_y\Re[h(\theta+\sin(\phi_{\epsilon}) \epsilon +\mathbf{i} y)]>0$ for $y \in [\cos(\phi_{\epsilon}) \epsilon, M]$. Because $R$ is increasing $\phi_{\epsilon}$ any smaller choice of $\phi_{\epsilon}>0$ also works. 

Similarly by analyticity of $h$, we can uniformly bound $h'(z)$ on the line segment $[\theta+\mathbf{i} \epsilon \cos(\phi_{\epsilon}), \theta+\epsilon \sin(\phi_{\epsilon})+\mathbf{i} \cos(\phi_{\epsilon})]$, and by Lemma \ref{lem:steep descent vertical contour} we know that $\Re[h(\theta+\mathbf{i} \epsilon \cos(\phi_{\epsilon}))-h(\theta)]<0$. Thus for small enough $\epsilon \sin(\phi_{\epsilon})$, $\Re[h(\theta+\epsilon \sin(\phi_{\epsilon})+\mathbf{i} \epsilon \cos(\phi_{\epsilon}))-h(\theta)]<-\eta$. Again for a particular choice of $\epsilon, \phi_{\epsilon}$, any smaller $\phi_{\epsilon}$ also works. 
\end{proof}

\section{Proof of the Fredholm determinant formula}\label{sec:sticky brownian limit}

In this section we will degenerate the Fredholm determinant formula in Theorem \ref{th:betarwre} for the Laplace transform of the quenched point to half line probability of a beta RWRE to arrive at the Fredholm determinant formula for the Laplace transform of $\K_{0,t}(0,[x,\infty))$ given in Theorem \ref{th:rmre fredholm det for laplace transform}.

\begin{proof}[Proof of Theorem \ref{th:rmre fredholm det for laplace transform}]

Let $X_{\epsilon}(t)$ be $X(t)$ be as in Definition \ref{def:betaRWRE} with parameters $\alpha=\beta=\epsilon \lambda$. By Lemma \ref{lem:beta rwre to sticky Brownian motions}, we have $\K_{0,t}(0,[x,\infty))=\lim_{\epsilon \to 0} \mathsf{P}(\epsilon X_{\epsilon}(\epsilon^{-2}t) \geq x).$ Note that in the expression for $K_{u}^{RW}(v,v')$, the only place where any of $x,t,\alpha,\beta$ appear is in the definition of $g^{RW}$. 

By Theorem \ref{th:betarwre} we can write
\begin{equation} \label{eq:fredholm det laplace transform rwre}
\mathbb E\left[e^{u\mathsf{P}(\epsilon X_{\epsilon}(\epsilon^{-2}t) \geq x)}\right] = \det(I - K_{u,\e}^{RW})_{L^2(C_0)}.
\end{equation}
with
$$K_{u,\epsilon}^{RW}(v,v')=\frac{1}{2 \pi \i} \int_{1/2-\i \infty}^{1/2+ \i  \infty} \frac{\pi}{\sin(\pi s)} (-u)^s \frac{g_{\epsilon}^{RW}(v)}{g_{\epsilon}^{RW}(v+s)} \frac{ds}{s+v-v'},$$
$$g^{RW}_{\epsilon}(v)=\left(\frac{\Gamma(v)}{\Gamma(\epsilon a+v)}\right)^{(\epsilon^{-2}t-\epsilon^{-1}x)/2}\left(\frac{\Gamma( \epsilon(a+b) +v)}{\Gamma(\epsilon a+v)}\right)^{(\epsilon^{-2}t+\epsilon^{-1}x)/2}\Gamma(v),$$
and $\mathcal{C}_0$ is a positively oriented circle around $1/2$ with radius $1/2$. 

We will take the limit of \eqref{eq:fredholm det laplace transform rwre} as $\epsilon \to 0$. The expression $e^{\mathsf{P}(\epsilon X_{\epsilon}(\epsilon^{-2}t) \geq x)}$ is bounded above by $e$, so in the left hand side we can pass the limit through the expectation to get 
$$\lim_{\e \to 0}\mathbb E[e^{\mathsf{P}(\epsilon X_{\epsilon}(\epsilon^{-2}t) \geq x)}] = \mathbb E[e^{u\K_{0,t}(0,[x,\infty))}].$$
Thus to complete the proof, we only need to show that
\begin{equation} \label{eq:fredholmdetwanttoshow}
\lim_{\e \to 0}\det(I - K_{u,\e}^{RW})_{L^2(C_0)}=\det(I-K_u)_{L^2(\mathcal{C})}.
\end{equation}
We prove \eqref{eq:fredholmdetwanttoshow} in three steps; step 1 gives the reason why this convergence should hold, while steps 2 and 3 provide the bounds necessary to make the argument rigorous.

\textbf{Step 1:}
First for fixed $v,v',s$ we show the integrand of $K_{u,\epsilon}^{RW}(v,v')$ converges to the integrand of $K_u(v,v')$ as $\e \to 0$.

By Taylor expanding in $\epsilon$, and setting $a=b=\lambda$, we have

$$g^{RW}_{\epsilon}(v)=\Big(1+\epsilon^2 \lambda^2\psi_1(v)+O(\epsilon^3)\Big)^{\epsilon^{-2}t/2}\Big(1-\epsilon \lambda \psi(v)+O(\epsilon^2)\Big)^{-\epsilon^{-1}x/2}\Big(1+\epsilon \lambda \psi(v)+O(\epsilon^2)\Big)^{\epsilon^{-1}x/2} \Gamma(v).$$
Taking the limit as $\epsilon \to 0$ gives
\begin{equation} \label{eq:pointwise convergence to RMRE}\lim_{\epsilon \to 0} g^{RW}_{\epsilon}(v)=g(v),
\qquad \text{for $v \in \mathbb{C} \setminus \Z_{\leq 0}$.}
\end{equation}
The limit \eqref{eq:pointwise convergence to RMRE} shows pointwise convergence of the integrand of $K_{u,\epsilon}^{RW}(v,v')$ to the integrand of $K_u(v,v')$.

Additionally if $K$ is a compact set which is separated from all poles of the Gamma function, then the convergence in \eqref{eq:pointwise convergence to RMRE} is uniform for $v \in K$. This follows from the fact that the Lagrange form of the remainder in the taylor expansions is bounded for $v \in K$, because $\Gamma''(v)$ is bounded for $v \in K$. So we have shown that integrand of $K_{u,\epsilon}^{RW}(v,v')$ converges to the integrand of $K_u(v,v')$. uniformly for $v$ in a compact set $K$ that does not contain poles of the Gamma function.

\textbf{Step 2:} Now we prove that for fixed $v,v'$, the kernel $K_{u,\epsilon}^{RW}(v,v') \to K_u(v,v')$ as $\e \to 0$. We do this by proving bounds on the integrand of $K_{u,\e}^{RW}(v,v')$ in order to apply dominated convergence to the pointwise convergence of the integrand in step 1.

For $s \in 3/4+\mathbf{i} \mathbb{R}$, $v \in B_{1/8}(0)$, we have the following bounds

\begin{equation} \label{eq:kernel RWRE bound 1}\left| \frac{\pi}{\sin(\pi s) \Gamma(s+v)} \right| \leq \frac{2 \pi}{e^{\pi |\Im[s]|}}e^{\frac{\pi}{2} |\Im[s+v]|-C+(\frac{1}{4}+\Re[v])\log|\Im[s]|},\end{equation}
\begin{equation} \label{eq:kernel RWRE bound 2}
|(-u)^s| \leq |u|^{3/4}
\end{equation}
Equation \eqref{eq:kernel RWRE bound 1} follows from Lemma \ref{lem:gamma bounds weak} and Lemma \ref{lem:sine bounds}. Equation \eqref{eq:kernel RWRE bound 2} follows from the fact that $u \in \mathbb{R}$.
Note that for $s=3/4+\mathbf{i} y$, $|y|>M$,
\begin{multline} \label{eq:kernel RWRE bound 3} \left|\left(\frac{\Gamma(v+s+\lambda \epsilon)^2}{\Gamma(v+s) \Gamma(v+s+2\lambda \epsilon)}\right)^{\epsilon^{-2} t/2}\right| =\exp( \log(\Gamma(v+s+\lambda \epsilon)-\log(\Gamma(v+s)\\+\log(\Gamma(v+s+\lambda \epsilon)-\log(\Gamma(v+s+2 \lambda \epsilon)) \leq 1. \end{multline}
The last inequality can be seen by applying Stirling's approximation in a precise way. For details see Lemma \ref{lem:gamma log convex}. Similarly for $|y|>M$,
\begin{equation} \label{eq:kernel RWRE bound 4}\left|\left( \frac{\Gamma(v)}{\Gamma(v+s+2 \lambda \epsilon)} \right)^{\epsilon^{-1} x/2} \right|=\exp(\log(\Gamma(v+s))-\Gamma(v+s+2 \lambda \epsilon))<1.\end{equation}
The last inequality follows from an approximation of the Gamma function which is similar to Stirling's approximation. See Lemma \ref{lem:gamma log increasing} for details.
For the final $s$ dependent term of the integrand, there is a constant $C>0$ such that
\begin{equation} \label{eq:kernel RWRE bound 5} \left|\frac{\Gamma(v+s)}{g^{RW}(v+s)}\right| =\left|\left(\frac{\Gamma(v+s+\lambda \epsilon)^2}{\Gamma(v+s) \Gamma(v+s+2\lambda \epsilon)}\right)^{\epsilon^{-2} t/2} \left( \frac{\Gamma(v+s)}{\Gamma(v+s+2 \lambda \epsilon)} \right)^{\epsilon^{-1} x/2}\right| \leq C \end{equation}
When $y>M$, \eqref{eq:kernel RWRE bound 5} follows from $x,t \geq 0$ along with \eqref{eq:kernel RWRE bound 3} and \eqref{eq:kernel RWRE bound 4}. When $y \leq M$, \eqref{eq:kernel RWRE bound 5} follows from uniform convergence for $s \in [3/4-\mathbf{i} M, 3/4+\mathbf{i} M]$ of $\frac{\Gamma(v+s)}{g^{RW}(v+s)}$ to $e^{-\lambda x \psi_0(v)-\frac{\lambda^2 t}{2} \psi_1(v)}$. 
By \eqref{eq:kernel RWRE bound 1}, \eqref{eq:kernel RWRE bound 2}, and \eqref{eq:kernel RWRE bound 5} we see that for $s \in 3/4+\mathbf{i} \mathbb{R}$ the integrand of $K_{u,\epsilon}^{RW}(v,v')$ is bounded, and has exponential decay coming from \eqref{eq:kernel RWRE bound 2} as $\Im[s] \to \infty$. Thus we can apply dominated convergence to show that 

\begin{equation} \label{eq:124}\lim_{\epsilon \to 0} K_{u,\epsilon}^{RW}(v,v') = K_u(v,v').\end{equation}

\textbf{Step 3:} Now we complete the proof of \eqref{eq:fredholmdetwanttoshow} by bounding the full Fredholm determinant expansion of $\det(I - K_{u,\e}^{RW})_{L^2(C_0)}$  in order to apply dominated convergence to the pointwise convergence of kernels proved in step 2.

 Let $\mathcal{A}_{\epsilon}$ be a rectangle with corners at $\frac{1}{8}+\mathbf{i} \f{1}{8}$, $\frac{1}{8}-\mathbf{i} \frac{1}{8}$, $-\lambda \epsilon +\mathbf{i} \frac{1}{8}$, $-\lambda \epsilon-\mathbf{i} \frac{1}{8}$, oriented in the counterclockwise direction. The convergence in \eqref{eq:pointwise convergence to RMRE} is uniform on $A_{\epsilon} \setminus B_{\delta}(0)$, so for sufficiently small $\epsilon>0$, and $v \in A_{\epsilon} \setminus B_{\delta}(0)$ there is a constant $C$ such that
$$g_{\epsilon}^{RW}(v) \leq C.$$ 

Now setting $v=\i y + \lambda \e$, we need to control
\begin{equation}g_{\epsilon}^{RW}(-\lambda \epsilon+\mathbf{i} y) = \Gamma(\i y - \lambda \e) \left( \frac{\Gamma(\mathbf{i} y +\lambda \epsilon) \Gamma(\mathbf{i} y -\epsilon)}{\Gamma(\mathbf{i} y)^2}\right)^{\epsilon^{-2} t/2} \left(\frac{\Gamma(\i y +\lambda \e)}{\Gamma(\i y -\lambda \e)} \right)^{\e^{-1} x/2}, \label{eq:123}\end{equation}
for $\epsilon, y \leq \delta$. Let $R(z)=\Gamma[z]-1/z$ and note that $R(z)$ is holomorphic in a neighborhood of $0$. By Taylor's theorem,
\begin{align} \label{eq:113} R(\i y+1+\epsilon)&= R(\i y+1)+R'(\i y+1) \epsilon +\text{Rem}(\i y+1,\e) \epsilon^2\\
\label{eq:114}R(\i y+1-\epsilon)&=R(\i y+1)-R'(\i y+1) \epsilon+\text{Rem}(\i y+1,-\epsilon) \e^2,
\end{align}
where $R(\i y+1)$, $R'(\i y+1)$, $\text{Rem}(\i y+1,\e)$, and $\text{Rem}(\i y+1,-\e)$ are bounded uniformly for $y \in (-\delta,\delta)$, $\epsilon \in (0,\delta)$.  

\begin{align}\label{eq:116} &\left( \frac{\Gamma(\mathbf{i} y +\lambda \epsilon) \Gamma(\mathbf{i} y -\epsilon)}{\Gamma(\mathbf{i} y)^2}\right)
=\frac{\frac{1}{\i y + \e}+R(\i y +\e +1)}{\frac{1}{\i y}+R(\i y+1)} \frac{\frac{1}{\i y-\e}+R(\i y-\e+1)}{\frac{1}{\i y}+\R(\i y+1)}\\
&=\left(\frac{(\i y)^2}{(\i y+\e)(\i y-\e)} \right) \left(\frac{\left(1+(\i y+\e) \left( R(\i y+1)+R'(\i y+1) \epsilon +\text{Rem}(\i y+1,\e) \epsilon^2\right) \right)}{(1+\i y R(\i y+1))^2} \right)\nonumber\\
&\times \left(1+(\i y-\e)\left(R(\i y+1)-R'(\i y+1) \epsilon+\text{Rem}(\i y+1,-\epsilon) \e^2 \right) \right)\nonumber\\
&=\left( \frac{1}{1+\frac{\e^2}{y^2}} \right) \left( \f{(1+\i y R(\i y +1))^2+\e^2 \text{Rem}_1(\i y+1, \e)}{(1+\i y R(\i y +1))^2} \right), \nonumber \end{align}
where $\text{Rem}_1(\i y+1, \e)$ is bounded uniformly for $y \in (-\delta, \delta), \epsilon \in (0, \delta)$. The first equality follows from the definition of $R$ and the second follows from \eqref{eq:113} and \eqref{eq:114}. The third equality follows expanding
\begin{align} \hspace{-1in} \label{eq:115} &\left(1+(\i y+\e) \left( R(\i y+1)+R'(\i y+1) \epsilon +\text{Rem}(\i y+1,\e) \epsilon^2\right) \right)\\
&\times \left(1+(\i y-\e)\left(R(\i y+1)-R'(\i y+1) \epsilon+\text{Rem}(\i y+1,-\epsilon) \e^2 \right)\right),\end{align}
and noting that the coefficient of $\e^0$ is $(1+\i y R(\i y +1))^2$, the coefficient of $\e^1$ is $0$. The fact that $\text{Rem}_1(\i y+1, \e)$ is bounded comes from the fact that every coefficient of $\e^k$ in the two terms of \eqref{eq:115} is bounded uniformly in $y,\epsilon$. 

Define
\begin{equation}\text{Rem}_2(\i y+1, \e)=\frac{\text{Rem}_1(\i y+1, \e)}{(1+\i y R(\i y +1))^2}.
\end{equation}
We have that for $x\in (0, \delta)$, $y \in (-\delta,\delta)$, 
\begin{align} \left|\left( \frac{\Gamma(\mathbf{i} y +\lambda \epsilon) \Gamma(\mathbf{i} y -\epsilon)}{\Gamma(\mathbf{i} y)^2}\right)^{-1}\right|
&=\left|\left(1+\frac{\e^2}{y^2}\right)\left(\frac{1}{1+\epsilon^2 \text{Rem}_2(\i y+1, \e)}\right)\right|\nonumber \\
&\geq \left|\left(1+\frac{\e^2}{y^2}\right)(1-\epsilon^2 \text{Rem}_2(\i y+1, \e))\right| \nonumber\\
&\geq \left|1+\e^2\left(\frac{1}{y^2}- \text{Rem}_2(\i y+1, \e) -\e^2 \frac{\text{Rem}_2(\i y+1, \e)}{y^2}\right)\right|\nonumber\\
&\geq \left(1+\e^2 \frac{3}{4y^2}\right). \label{eq:117}
\end{align}
 The first equality follows from \eqref{eq:116}. The first inequality follows from the fact that for any $0<x<1$, $\left|\frac{1}{1+x}\right| \geq \left| 1-x \right|,$. The final inequality may require us to choose a still smaller $\delta>0$ and follows from the fact that $\text{Rem}_2(\i y+1, \e)$ is bounded.

By Laurent expanding the Gamma function around $0$, we can see that
\begin{equation} \label{eq:118}\Gamma(-\lambda \epsilon + \i y) \leq \frac{1}{\sqrt{y^2+\e^{2}}}+C \leq \frac{1}{y}+C,\end{equation}
for $0 < \epsilon<\delta$ and $y \in (-\delta,\delta)$. We also have

\begin{align} \left(\left( \frac{\Gamma(\mathbf{i} y +\lambda \epsilon) \Gamma(\mathbf{i} y -\epsilon)}{\Gamma(\mathbf{i} y)^2}\right)^{\epsilon^{-2} t/2}\right)^{-1}
&\geq \left(1+\e^2 \frac{3}{4y^2}\right)^{\f{\e^{-2} t}{2}}
 \geq 1+\frac{3t}{8 y^2} \nonumber \\
&\geq \left(\f{1}{y}+C\right)
\geq  \Gamma(\i y-\e \lambda), \label{eq:119} \end{align}
for $y$ sufficiently small. The first inequality follows from \eqref{eq:117}, the equality is Newton's generalized binomial theorem, the second inequality uses Bernoulli's inequality. The third inequality is true if $y$ large in particular $1/y>1+8C/3t$. The fourth inequality follows \eqref{eq:118}.

Equation \eqref{eq:119} implies
\begin{equation}\Gamma(\i y - \lambda \e) \left( \frac{\Gamma(\mathbf{i} y +\lambda \epsilon) \Gamma(\mathbf{i} y -\epsilon)}{\Gamma(\mathbf{i} y)^2}\right)^{\epsilon^{-2} t/2} \leq 1.\label{eq:120}\end{equation}
By Taylor's theorem, there exists a function $\text{Rem}_3(\i y, \epsilon)$ which is bounded for $\epsilon \in (0,\delta)$, $y \in (-\delta,\delta)$ satisfying
$$\Gamma(\i y+ \e)=\frac{1}{\i y +\e} +\text{Rem}_3(\i y, \epsilon) \e\,\,\text{ and }\,\,\Gamma(\i y- \e)=\frac{1}{\i y -\e} +\text{Rem}_3(\i y, -\epsilon) \e.$$
Thus
$$\frac{\Gamma(\i y + \e)}{\Gamma(\i y -\e)}=\frac{\frac{1}{\i y +\e} +\text{Rem}_3(\i y, \epsilon) \e}{\frac{1}{\i y -\e} +\text{Rem}_3(\i y, -\epsilon) \e}=\left(\frac{\i y-\e}{\i y +\e}\right) \left(\frac{1+(\i y +\e)\text{Rem}_3(\i y, \epsilon) \e}{1+(\i y -\e)\text{Rem}_3(\i y, -\epsilon) \e}\right) = \left(\frac{\i y-\e}{\i y +\e}\right) (1+C_{\epsilon,y} \e),$$
where for any $\eta$ we can choose $\delta$ small enough that $C_{\e ,y} \leq \eta$. Thus for all $\e \in (0,\delta), y \in (-\delta,\delta)$,
\begin{equation} \label{eq:121}
(1-\eta \e) \leq \left| \frac{\Gamma(\i y + \e)}{\Gamma(\i y -\e)} \right| \leq (1+\eta \e).
\end{equation}
This implies
\begin{equation} \label{eq:122}
\left|\left(\frac{\Gamma(\i y +\lambda \e)}{\Gamma(\i y -\lambda \e)} \right)^{\e^{-1} x/2}\right| \leq (1+\eta \e)^{e^{-1} x/2} \leq e^{\eta x/2}.
\end{equation}

Together \eqref{eq:123}, \eqref{eq:120}, and \eqref{eq:122} imply that for $\delta$ small, $\e \in (0, \delta), y \in (-\delta,\delta)$,
$$|g_{\epsilon}^{RW}(-\lambda \epsilon+\mathbf{i} y)| \leq e^{\eta x/2}.$$
Thus $K_{u,\e}^{RW}(v,v')$ is bounded by some $C$ on the contour $\mathcal{A_{\e}}$. Hadamard's bound implies that 
$$\frac{\det[K_{u,\epsilon}(x_i,x_j)]_{i,j=1}^{k}}{k!} \leq \frac{C^k k^{k/2}}{k!},$$ 
where the right hand side decays at rate $ \frac{C^k}{e^{k \log(k)/2} }$ by Stirling's formula. Together with \eqref{eq:124}, and the fact that the contours $A_{\epsilon}$ are finite volume, this allows us to apply dominated convergence to the Fredholm determinant expansion $\det(I-K_{u,\e}^{RW}(v,v'))_{L^2(\mathcal{A}_{\e})},$ to get
\begin{equation} \label{eq:125}
\lim_{\e \to 0}\det(I-K_{u,\e}^{RW})_{L^2(\mathcal{A}_{\e})}
=\det(I-K_{u})_{L^2(\mathcal{A}_{\e})}.
\end{equation}
We can deform the contour $\mathcal{A}_{\e}$ to $\mathcal{C}$ without crossing any poles of $K_{u}(v,v')$ and we can deform $\mathcal{A}_{\e}$ to $\mathcal{C}_0$ without crossing any poles of $K_{u,\e}^{RW}(v,v')$, so by \eqref{eq:125}
$$\lim_{\e \to 0}\det(I-K_{u,\e}^{RW})_{L^2(\mathcal{C}_0}))=\det(I-K_u)_{L^2(\mathcal{C})}.$$
\end{proof}

\section{Moment formulas and Bethe ansatz} 
\label{sec:moment formulas}

\subsection{Proof of the moment formula Proposition \ref{prop:moment formula sticky bm}}In this section we find moment formulas for kernels of uniform Howitt-Warren flows, by taking the diffusive limit of \cite[Proposition 3.4]{RWRE}. In order to state precisely how we use results from \cite{RWRE}, we first explain the connection between the beta RWRE and another model called the beta polymer, which was also introduced in \cite{RWRE}.

\begin{definition}[beta polymer] The beta polymer is a probability measure on oriented lattice paths constructed as follows. We consider paths in $\Z^2$ with allowed edges of the form $(i,j)\to (i+1,j)$ and $(i,j)\to(i+1,j+1)$. In other terms, we allow paths to make either right or up-right steps. The measure depends on two parameters  $\nu>\mu>0$. Let $\{B_{(i,j)}\}_{i,j\Z^2}$ be a family of iid random variables distributed according to the beta distribution with parameters $\nu, \nu-\mu$. To each horizontal edge $e=(i-1,j) \to (i,j)$ we assign the Boltzmann weight $w_e=B_{i,j}$, and to each diagonal edge $e=(i-1,j-1) \to (i,j)$ we associate the Boltzmann weight $w_e=(1-B_{i,j})$. 
	
For fixed points $S,T\in \Z^2$, the beta polymer is a measure on paths $\pi: S\to T$ such that the probability of a path $\pi$ is proportional to $\prod_{e \in \pi} w_e$. In this paper, we are mostly interested in paths between the half-line $D:= \lbrace (0,i): i>0\rbrace$ and any point of coordinates $(t,n)$ for $t\geqslant 0$. The associated partition function is defined as 
$$\msf{Z}(t,n)=\sum_{i=1}^{n}\sum_{\pi:(0,i) \to (t,n)} \prod_{e \in \pi} w_e.$$
By the definition of our Boltzmann weights, for $t\geq 0, n\in \Z$, the partition function $\msf{Z}(t,n)$ is characterized by the following recurrence relation.  
$$\begin{cases}
\msf{Z}(t,n)=B_{t,n} \msf{Z}(t-1,n)+(1-B_{t,n}) \msf{Z}(t-1,n-1), & \text{ if $t>0$}\\
\msf{Z}(0,n)=\mathds{1}_{n>0}. 
\end{cases}$$
Note that this half line to point partition function is the same as the partition function $Z(t,n)$ defined in \cite[Definition 1.2]{RWRE}.
\end{definition} 

 Now we rephrase the relation between the beta RWRE and the beta polymer from  \cite[Proposition 1.6]{RWRE}.

\begin{proposition}\label{prop:beta RWRE to beta polymer} Consider the beta RWRE with parameters $\alpha, \beta >0$ (see Definition \ref{def:betaRWRE})  and the beta polymer with parameters $\mu=\alpha$, $\nu=\alpha+\beta$. For $t\geq 0$ and $n_1, \dots, n_k\in \Z$, we have the following equality in distribution,
$$(\msf{Z}(t,n_1),...,\msf{Z}(t,n_k))=(\msf{P}(X_1^{x_1}(t)\geq 0),...,\msf{P}(X_k^{x_k}(t)\geq 0)) \qquad \text{for $x_i=2n_i-2-t$},$$
and
$$\E \left[\prod_{i=1}^k \msf{Z}(t,n_i)\right]=\E\left[\prod_{i=1}^k \msf{P}(X_i^{x_i}(t)\geq 0)\right],$$
where these expectations are taken over the random environments of the beta polymer and the beta RWRE respectively.
\end{proposition}

\begin{proof}
First note that although the beta RWRE was defined for positive time, we can apply a spatial shift to our variables so that it is defined for all $t>-L$ for any $L \in \mathbb{Z}$. We will use this interpretation when we describe a particle trajectory in the beta RWRE starting from a point with a negative time coordinate. Consider the change of coordinates $x=2n-2-t$ and rewrite $\msf{Z}(t,n)$ in terms of $(t,x)$. This corresponds to transforming  horizontal edges into  diagonal down-right edges. Then, reversing the time direction allows us to identify paths from $D$ to $(t,n)$ in the beta polymer with space-time  trajectories in the beta RWRE from the point $x$ at time $-t$  to the half line $[0, +\infty)$ at time $0$, such that the Boltzmann weight of the beta polymer path is equal in distribution to the probability of the beta RWRE trajectory. This equality in distribution also holds jointly for arbitrary collections of paths. Finally,  shifting all paths forward in time by $t$ steps in the beta RWRE does not change their law, thus we have the desired equality in distribution.
%$$(\msf{Z}(t,n_1),...,\msf{Z}(t,n_k))=(\msf{P}(X_1^{x_1}(t)\geq 0),...,\msf{P}(X_k^{x_k}(t)\geq 0)) \qquad \text{for $x_i=2n_i-2-t$}.$$ 
\end{proof}
Now we can prove the  mixed moments formula (Proposition \ref{prop:moment formula sticky bm}).
\begin{proof}[Proof of Proposition \ref{prop:moment formula sticky bm}]

We begin with the moment formula \cite[Proposition 3.4]{RWRE}, Using Proposition \ref{prop:beta RWRE to beta polymer} to rewrite $\msf{Z}(t,n)$ in terms of $\msf{P}(X^{x}(t)\geq 0)$ gives, for $x_1\geq \dots \geq x_k$, 
\begin{multline}\label{eq:beta rwre moment formula 2} \E[\msf{P}(X_1^{x_1}(t)\geq 0)...\msf{P}(X_k^{x_k}(t)\geq 0)]= \\  \frac{1}{(2 \pi \i)^k} \int_{\gamma_1} ... \int_{\gamma_k} \prod_{A<B} \frac{z_A-z_B}{z_A-z_B-1} \prod_{j=1}^k \left(\frac{ \nu+z_j}{z_j} \right)^{\frac{t+x}{2}-1} \left( \frac{ \mu +z_j}{\nu+z_j} \right)^t \frac{dz_j}{z_j+\nu}.
\end{multline}
Where $\gamma_k$ is a small contour around $0$ and $\gamma_i$ contains $1+\gamma_j$ for $i<j$, and all contours exclude $-\nu$. To choose the $\gamma_i$ precisely, fix a small $a_k>0$ and define the contour $\gamma_k=\gamma_k^{\epsilon}$ to be a short vertical line segment $\{-\lambda \epsilon+\i y: y \in [-a_k,a_k]\}$ union a half circle a $\{-\lambda \epsilon + a_ke^{\i \theta}: \theta \in [-\pi/2, \pi/2]\}$. Let construct $\gamma_{i}=\gamma_i^{\epsilon}$ in the same way with $a_{i}$ replacing $a_k$ and choose each $a_i$ large enough that $1+\gamma_{i+1}^{\epsilon}$ is contained in $\gamma_i^{\epsilon}$.
Recalling Lemma \ref{lem:beta rwre to sticky Brownian motions} and taking $\e \to 0$ in \eqref{eq:beta rwre moment formula 2} gives
\begin{multline*}
 \E[\K_{-t,0}(x_1,[0,+\infty))...\K_{-t,0}(x_k,[0,+\infty))]= \nonumber \\
 \lim_{\epsilon \to 0} \frac{1}{(2 \pi \i)^k}\int_{\gamma^{\epsilon}_1} ... \int_{\gamma^{\epsilon}_k} \prod_{A<B} \frac{z_A-z_B}{z_A-z_B-1} \prod_{j=1}^k \left(\frac{ 2 \lambda \e+z_j}{z_j} \right)^{\frac{\epsilon^{-2}t+\epsilon^{-1}x}{2}} \left( \frac{ \lambda \e +z_j}{2 \lambda \e+z_j} \right)^{\epsilon^{-2}t} \frac{dz_j}{z_j}.
 \end{multline*}
 We simplify the product 
 \begin{equation}\label{eq:pre limit moment formula}\prod_{j=1}^k \left(\frac{ 2 \lambda \e+z_j}{z_j} \right)^{\frac{\epsilon^{-2}t+\epsilon^{-1}x}{2}} \left( \frac{ \lambda \e +z_j}{2 \lambda \e+z_j} \right)^{\epsilon^{-2}t} \frac{dz_j}{z_j}= \prod_{j=1}^k \left( 1+\frac{2 \lambda \e}{z_j} \right)^{\f{\e^{-1}x}{2}} \left( \frac{(\lambda \e +z_j)^2}{z_j(2 \lambda \e +z_j)}\right)^{\f{\e^{-2} t}{2}} \frac{dz_j}{z_j}.
 \end{equation}
 Taking the pointwise limit of the integrand suggests that 
 \begin{multline} \label{eq:moment formula sticky bm}
 \mathbb E[\K_{-t,0}(x_1,[0,+\infty))...\K_{-t,0}(x_k,[0,+\infty))]= \\  \int_{\gamma^0_1}\frac{dz_1}{ 2 \pi \i} ... \int_{\gamma^0_k} \frac{d z_k}{ 2 \pi \i} \prod_{A<B} \frac{z_A-z_B}{z_A-z_B-1} \prod_{j=1}^k \exp \left( \frac{\lambda^2 t}{2 z_j^2} +\frac{\lambda x_j}{z_j} \right) \frac{1}{z_j},
 \end{multline}
where now the contours $\gamma^0_k$,... $\gamma^0_1$ all pass through $0$ in the vertical direction and $\gamma^0_i$ contains $1+\gamma^0_j$ for all $i<j$.  We will justify this limit by applying dominated convergence at the end of the proof.

The condition $\alpha_i < \frac{\alpha_{j}}{1+\alpha_{j}}$ for all $i<j$ implies that if $\bar{\gamma}_i$ is the circle centered at $\alpha_i^{-1}/2$ with radius $\alpha_i^{-1}/2$ oriented in the counterclockwise direction, then $1+\bar{\gamma}_j$ is contained in $\bar{\gamma}_i$ for all $i<j$. Thus Cauchy's theorem allows us to deform the integration contours $\gamma_i$ to $\bar{\gamma}_i$ in \eqref{eq:moment formula sticky bm} without collecting any residues.

We perform a change of variables $w_j=1/z_j$ on \eqref{eq:moment formula sticky bm} and use the fact that the pointwise inverse in the complex plane of a circle with center $\alpha^{-1}/2$ and radius $\alpha^{-1}/2$ is the line $\alpha+\i \R$. We obtain 
\begin{multline}
\mathbb E[\K_{-t,0}(x_1,[0,+\infty))...\K_{-t,0}(x_k,[0,+\infty))]= \\ 
\int_{\alpha_1+\I \R} \frac{\mathrm d w_1}{2\I\pi} \dots \int_{\alpha_k+\I \R} \frac{\mathrm d w_k}{2\I\pi} \prod_{1\leqslant A<B \leqslant k} \frac{w_{B}-w_A}{w_B-w_A-w_Aw_B} \prod_{j=1}^k \exp\left( \frac{t \lambda^2 w_j^2}{2} +\lambda x_jw_j \right)\frac{1}{w_j}.
\label{eq:momentformulawithwvariables}
\end{multline}

Now use dominated convergence to justify the the $\epsilon \to 0$ limit which gave \eqref{eq:moment formula sticky bm}. The contours $\gamma_i(\epsilon)$ depend on $\epsilon$ and in order to apply dominated convergence, we perform the change of variables $z_i=\bar{z}_i-\lambda \epsilon$ in $\eqref{eq:pre limit moment formula}$ so that our contours of integration change from $\gamma_i[\epsilon]$ to $\gamma_i[0]$, and set $\gamma_i'=\gamma_i[0]$. 

Now that all our contours of integration do not depend on $\epsilon$, all we need to do is bound the integrand along these contours. The argument which allows us to apply dominated convergence to get \eqref{eq:moment formula sticky bm} is a simplified form of the argument which allows us to apply dominated convergence in the proof of Theorem \ref{th:rmre fredholm det for laplace transform}. Taylor expanding shows that
$$\left( 1+\e a \right)^{\epsilon^{-1}} \xrightarrow{\e \to 0} e^a, \qquad \text{uniformly in $a$ for $|a|<R.$}$$
Thus, uniformly for $z_j$ outside a neighborhood of $0$, 
\begin{equation} \label{eq:dom con moment formula 1} \left(1+\frac{2 \lambda \e}{z_j} \right)^{\frac{\e^{-1} x}{2}} \xrightarrow[\e \to 0]{} e^{\frac{\lambda x}{z_j}},
\end{equation}
and
\begin{equation} \label{eq:dom con moment formula 2}\left(\frac{(\lambda \e +z_j)^2}{z_j (2 \lambda \e +z_j)}\right)^{\frac{\e^{-2} t}{2}}=\left(1+\frac{\e^2 \lambda^2}{z_j^2+2 \lambda z \e}\right)^{\frac{\e^{-2} t}{2}} \xrightarrow[\e \to 0]{} e^{\frac{\lambda^2 t}{2 z_j^2}}. 
\end{equation}

Now we bound the integrands along $\gamma_i[\epsilon] \cap B_{\delta}(0)$. Near $0$ we have $z_i=-\lambda \e +\i y$, so
\begin{equation} \label{eq:dom con moment formula 3}
\left\vert \left(1+\frac{2 \lambda \e}{z_j} \right)^{\frac{\e^{-1} x}{2}}\right\vert =\left\vert\frac{ \lambda \e + \i y}{-\lambda \e + \i y}\right\vert^{\frac{\e^{-1} x}{2}} =1,
\end{equation}
and
\begin{equation} \label{eq:dom con moment formula 4}
\left\vert \left(\frac{(\lambda \e +z_j)^2}{z_j (2 \lambda \e +z_j)}\right)^{\frac{\e^{-2} t}{2}}\right\vert=\left| \frac{- y^2}{-y^2-\lambda^2 \e^2} \right|^{\frac{\e^{-2} t}{2}}<1.
\end{equation}

Together \eqref{eq:dom con moment formula 1}, \eqref{eq:dom con moment formula 2}, \eqref{eq:dom con moment formula 3}, and \eqref{eq:dom con moment formula 4}, and the fact that $\gamma_k^{\e}$ has uniformly bounded length,  allow us to apply dominated convergence to \eqref{eq:pre limit moment formula} to obtain \eqref{eq:moment formula sticky bm}. This completes the proof.
\end{proof}

\subsection{Limit to the KPZ equation}
\label{sec:limittoKPZ}
In this Section, we show that the moment formula from Proposition \ref{prop:moment formula sticky bm} converges to the moments of the solution to the multiplicative noise stochastic heat equation with delta initial data, which suggests that Howitt-Warren stochastic flows of kernels converge to the KPZ equation.

Consider $Z(t,x)$ the solution to the multiplicative noise stochastic heat equation 
$$\partial_tZ(t,x) = \frac{1}{2} \partial_{xx}Z(t,x) + \sqrt{\kappa} \xi(t,x) Z(t,x), \;t>0, x\in \R,$$
where $\xi$ is a space time white noise and $\kappa>0$ is a parameter controlling the noise strength. This stochastic PDE has attracted much attention recently because the solution to the KPZ equation 
$$ \partial_th(t,x) = \frac{1}{2} \partial_{xx}h(t,x) + \frac{1}{2}\left( \partial_x h(t,x) \right)^2 +\sqrt{\kappa}\xi(t,x)$$
is given by $h(t,x)=\log Z(t,x)$. It is expected that models in the KPZ class which depend on a tunable parameter controlling noise or asymmetry converge to the KPZ equation in the weak asymmetry/noise scaling limit.  We refer to \cite{corwin2012kardar} for background on these scalings and stochastic PDEs. 

Let 
$$ u_{\kappa}(t,\vec{x}) = \mathbb E \left[ Z(t,x_1) \dots Z(t,x_k \right]. $$
It was shown in \cite[Section 6.2]{macdonaldprocesses} (see also \cite{ghosal2018moments}) that for Dirac delta initial data $u(0,\cdot )= \delta_0(\cdot)$, the function $u_{\kappa}$ can be written for $x_1\leqslant  \dots \leqslant  x_k$ as 
\begin{equation}
u_{\kappa}(t,\vec{x})= \int_{r_1+\i \R} \frac{dz_1}{2\I\pi} \dots \int_{r_k+\i \R} \frac{dz_k}{2\I\pi} \prod_{1 \leq A <B \leq k} \frac{z_A-z_B}{z_A-z_B-\kappa} \prod_{j=1}^{k} e^{x_j z_j +\frac{t}{2} z_j^2},
\label{eq:momentsKPZ}
\end{equation}
where the contours are such that $r_i>r_{i+1}+\kappa $ for all $1\leqslant i\leqslant k$.

Recall the moments of the uniform Howitt-Warren flow
$$\Phi^{(k)}_t(x_1, \dots , x_k)= \E \Big[\K_{-t, 0}(x_1, [0, +\infty))\dots \K_{-t, 0}(x_k, [0, +\infty)) \Big], $$
and recall that they depend on a noise parameter $\lambda$. 

\begin{proposition}
Let $\gamma>0$ and consider the scalings  
\begin{equation}
T=\lambda^2 t,\,  X_i=\lambda^2t\gamma+\lambda x_i. 
\label{eq:scalings to KPZ}
\end{equation}
Let $\K_t(x, \cdot)$ be the kernel of the uniform Howitt-Warren stochastic flow with stickiness parameter $\lambda$.  We have that for fixed $t>0$ and $x_1 \leqslant \dots \leqslant x_k$, 
$$\lim_{\lambda \to \infty}\,\, (\lambda \gamma)^k \exp\left( {\frac{k}{2} t \lambda^2  \gamma^2+\lambda \gamma \sum_{j=1}^k x_{j} }\right) \Phi^{(k)}_T(-\vec{X})=u_{\gamma^2}(t,\vec{x}).$$
\label{prop:convergencetoKPZ}
\end{proposition}
\begin{remark}
Proposition \ref{prop:convergencetoKPZ} suggests that under the scalings \eqref{eq:scalings to KPZ},  
$$Z_{\lambda}(t,x):= \gamma \lambda e^{t \lambda^2 \gamma^2/2 + \lambda \gamma x}\K_{-T}(-X, [0,+\infty))$$
 weakly converges as $\lambda$ goes to $+\infty$ (in the space of continuous time space trajectories) to the solution of the multiplicative noise stochastic heat equation $Z(t,x)$ with Dirac delta initial data and $\kappa=\gamma^2$. Equivalently, $\log Z_{\lambda}(t,x)$ would converge weakly to the solution to the KPZ equation with narrow wedge initial data. The analogous statement for discrete random walks in space-time iid random environment is proved in \cite{corwin2017kardar}. 
\end{remark}
\begin{proof}
Consider \eqref{eq:formulaforPsit} and perform the change of variables 
$w_j = \frac{\gamma}{\lambda} + \frac{z_i}{\lambda^2}$. For large enough $\lambda$, the contour for $z_i$ may be chosen as $r_i+\I\R$ where $r_{i+1}>r_{i}+\gamma^2 $ for all $1\leqslant i
\leqslant k$.  Under the scalings \eqref{eq:scalings to KPZ}, we have (dropping unnecessary indices)
$$ \frac{T}{2} \lambda^2 w^2 - \lambda X w = \frac{t}{2}z^2 - x z -\gamma x\lambda -\frac{t}{2}\gamma^2\lambda^2, $$
and we have the pointwise convergences
 $$ \frac{w_b-w_a}{w_b-w_a-w_aw_b} \xrightarrow[\lambda\to +\infty]{}  \frac{z_b-z_a}{z_b-z_a-\gamma^2}, \,\,\, \frac{1}{\lambda w_i} \xrightarrow[\lambda\to +\infty]{} \frac{1}{\gamma}  . $$
Moreover, it is easy to see that the ratios stay bounded for $z_a,z_b, z_i$ belonging to their fixed vertical contours. Thus, by dominated convergence, 
\begin{multline*} (\lambda\gamma)^k e^{\sum_{j=1}^k \frac{t}{2}\gamma^2\lambda^2 + \gamma x_i \lambda} \Phi^{(k)}_T(-X_1, \dots, -X_k) \xrightarrow[\lambda\to+\infty]{} \\ \int_{r_1+\i \R} \frac{dz_1}{2\I\pi} \dots \int_{r_k+\i \R} \frac{dz_k}{2\I\pi} \prod_{1 \leq A <B \leq k} \frac{z_B-z_A}{z_B-z_A-\gamma^2} \prod_{j=1}^{k} e^{\frac{t}{2} z_j^2 -x_jz_j}.
\end{multline*}
We finally obtain \eqref{eq:momentsKPZ} by the change of variables $z_i = -\tilde z_i$. 
\end{proof}

\subsection{Bethe Ansatz solvability of $n$-point uniform sticky Brownian motions}
\label{sec:commentsBetheansatz}

For $x\in \R^k$ and $t\geq 0$, let  $u(t,\vec x)$ be the right hand side of \eqref{eq:momentformulawithwvariables}. We claim that $u$ satisfies 
\begin{align}\label{eq:heatequation} 
& \partial_t u=\frac{1}{2} \Delta u,\\
&(\partial_i \partial_{i+1} +  \lambda (\partial_{i}-\partial_{i+1})) u|_{x_i=x_{i+1}}=0 .\label{eq:boundarycondition}
\end{align}

Indeed, for any $w\in \C$, the function $\exp\left( \frac{t\lambda^2 w^2}{2} +\lambda x w \right)$ is clearly a solution to \eqref{eq:heatequation}. This equation is linear and hence any superposition of solutions satisfies it, so does $u(t,\vec x)$. 

Regarding the boundary condition \eqref{eq:boundarycondition}, let us apply the operator $\partial_i \partial_{i+1} +  \lambda (\partial_{i}-\partial_{i+1})$ to $u(t,\vec x)$. The operator can be brought inside the integrals in \eqref{eq:momentformulawithwvariables}  and yields a multiplicative factor 
$$\lambda^2w_iw_{i+1} +\lambda( \lambda w_{i} - \lambda w_{i+1}).$$
This factor cancels the denominator of 
$$\frac{w_{B}-w_A}{w_B-w_A-w_Aw_B}$$ 
when $A=i$, $B=i+1$, so that the integral in $w_{i+1}$ does not have a pole anymore at $w_{i+1}=w_i/(1+w_i)$. Thus, by Cauchy's theorem, one can shift the $w_{i+1}$ contour from $\alpha_{i+1}+\I\R$ to  $\alpha_{i}+\I\R$. Now that variables $w_i$ and $w_{i+1}$ are integrated on the same contour, we notice that for $x_{i+1}=x_i$, the integrand is antisymmetric with respect to exchanging $w_i$ and $w_{i+1}$ (because of the factor $w_i-w_{i+1}$), and hence the integral is zero. Thus $u(t,\vec x)$ satisfies \eqref{eq:boundarycondition}. 

More generally, the function 
\begin{equation}
\Psi_{\vec z}(\vec x)=\sum_{\sigma}  \prod_{i<j} \frac{z_{\sigma(i)}-z_{\sigma(j)}-1}{z_{\sigma(i)}-z_{\sigma(j)}} \prod_{j=1}^k e^{\frac{\lambda x_j}{z_j}},
\end{equation}
satisfies \eqref{eq:heatequation}, \eqref{eq:boundarycondition} for any $\vec z\in (\C\setminus \lbrace 0 \rbrace )^k$. 

The function $u(0,\vec x)$ is a linear superposition of $\Psi_{\vec z}(\vec x)$ which additionally satisfies the initial condition for $x_1>\dots >x_k$ that 
$$ u(0,\vec x)  =  \prod_{i=1}^k \mathds{1}_{x_i>0}. $$

Note that the function $\Phi^{(k)}_t(\vec x):= \mathbb E[\K_{-t,0}(x_1,[0,+\infty))...\K_{-t,0}(x_k,[0,+\infty))]$ satisfies the same initial condition. 

The discrete analogue of $\Phi^{(k)}_t(\vec x)$ is $\E[\msf{P}(X^{x_1}(t)>0)...\msf{P}(X^{x_k}(t)>0)]$ (in the sense of Lemma \ref{lem:beta rwre to sticky Brownian motions}). It was shown in \cite[Section 3.1]{RWRE} using simple probabilistic considerations that the latter quantity satisfies  discretizations of \eqref{eq:heatequation}, \eqref{eq:boundarycondition}.

It would be interesting to provide a probabilistic explanation for  why $\Phi^{(k)}_t(\vec x)$ must satisfy \eqref{eq:heatequation}, \eqref{eq:boundarycondition}. Note that $\Phi^{(k)}_t(\vec x)$ is symmetric in the $x_i$'s so that we need to understand it only in the Weyl chamber $\mathbb W_k:= \lbrace x\in \R^k : x_1\geq \dots \geq x_k\rbrace $. Then \eqref{eq:heatequation}, \eqref{eq:boundarycondition} should be regarded as Kolmogorov's backward equation for $k$-point uniform sticky Brownian motions. Inside the open sector  $x_1>\dots >x_k$, it is clear that the generator should be given by the Laplacian (since $k$-point sticky Brownian motions evolve as $k$ independent Brownian motions), hence the heat equation \eqref{eq:heatequation}. However, we have not found in the literature a rigorous definition of the generator for $n$-point uniform sticky Brownian motions and we are unable to deduce the boundary condition \eqref{eq:boundarycondition} directly from the definition of uniform sticky Brownian motions. After the posting of the manuscript on arXiv, we have learned from Jon Warren that it is possible to derive \eqref{eq:heatequation}, \eqref{eq:boundarycondition} directly from the martingale problem characterizing sticky Brownian motions, and this will be explained in the forthcoming paper \cite{warren}.

\subsection{A formal relation to diffusions with white noise drift}
\label{sec:whitenoisedrift}

By analogy with the Lieb-Liniger model (we refer the reader to the book  \cite[Chap. 4]{gaudin2014bethe} for background on the Lieb Liniger model, or  \cite[Section 6]{macdonaldprocesses} for its relation to the KPZ equation), it is natural from the physics point of view to associate to the equation \eqref{eq:heatequation} with boundary condition \eqref{eq:boundarycondition}  the following PDE on $\R^k$
	\begin{equation}
	\partial_t v(t, \vec x)=\frac{1}{2} \Delta v(t,\vec x) + \frac{1}{2 \lambda} \sum_{i \neq j}\delta(x_i-x_j) \partial_{x_i} \partial_{x_j}v(t, \vec x).  
	\label{eq:trueevolution2}
	\end{equation}
In order to see that \eqref{eq:trueevolution2} satisfies the boundary condition \eqref{eq:boundarycondition}, integrate the equation over the variable $y=x_{i+1}-x_i$ in a neighborhood of $0$, and use the fact that $v(t,\vec x)$ is symmetric in the $x_i$'s for symmetric initial condition. Assuming uniqueness of solutions to \eqref{eq:heatequation}+\eqref{eq:boundarycondition} and \eqref{eq:trueevolution2}, their restrictions to the Weyl chamber $\mathbb W_k:= \lbrace x\in \R^k : x_1\geq \dots \geq x_k\rbrace $ must coincide, provided the initial conditions coincide on $\mathbb W_k$.

Consider now the  stochastic PDE
\begin{equation}
\begin{cases}
\partial_t q(t,x) = \frac{1}{2} \partial_{xx} q(t,x) + \frac{1}{\sqrt{\lambda}}\xi(t,x) \partial_x q(t,x), \\
q(0,t)=q_0(x).
\end{cases}
\label{eq:Kolmogorov2}
\end{equation}
It is not clear to us if a solution theory is available when $\xi$ is a space-time white noise, although this is the case we are ultimately interested in. 
However, if $\xi$ is a smooth and Lipschitz  potential, the Kolmogorov backward equation provides a representation of the solution as 
$$q(t,x) = \mathsf E[q_0(X_0)\vert X_{-t}=x],$$
where $X_t$ is the random diffusion \cite{le2017diffusion} (see also \cite[Eq. (2.9)]{bernard1998slow})
\begin{equation}
dX_t = \frac{1}{\sqrt{\lambda}}\xi(t,X_t)dt + dB_t,
\label{eq:RMRE2}
\end{equation}
where the Brownian motion $B$ is independent from $\xi$, and $\mathsf E$ denotes the expectation with respect to $B$, conditionally on the environment $\xi$. For a white noise potential $\xi$ depending only on space, such diffusion can be constructed rigorously \cite{tanaka1997limit}. Note that $q$ does not satisfy \eqref{eq:Kolmogorov2} for $\xi$ white in time and smooth in space due to Ito corrections in the derivation of \eqref{eq:Kolmogorov2} from \eqref{eq:RMRE2}.

Let 
\begin{equation}
\tilde v(t,\vec x):= \mathbb E \left[ q(t,x_1)\dots q(t,x_k)\right],
\label{eq:defvtilde}
\end{equation}
where $q$ solves \eqref{eq:Kolmogorov2}. We claim that $\tilde v(t,\vec x)$ satisfies \eqref{eq:trueevolution2} in the following formal sense. The following arguments are non rigorous, as we will discard many analytic difficulties such as exchanging derivatives with expectation without justification and we implicitly assume existence and uniqueness of solutions of \eqref{eq:Kolmogorov2} when $\xi$ is a space time white noise.

 By definition, a solution to \eqref{eq:Kolmogorov2} satisfies 
$$ q(t,x) = p_t \ast q_0(x) + \frac{1}{\sqrt{\lambda}} \int_{0}^tds \int_{\R}dy p_{t-s}(x-y)\xi(s,y)\partial_y q(s,y),$$
where $\ast$ denotes convolution in space, and $p_t(x) = \frac{1}{\sqrt{2\pi t}}e^{-x^2/2t}$ denotes the heat kernel. Note that when $\xi(s,y)$ is not smooth in space, the integral against $\xi(s,y)\partial_y q(s,y)$ is not well defined even using Ito calculus.
Let us assume for the moment that the covariance of the environment  $\xi$ is given by 
\begin{equation}
 \mathbb E \left[ \xi(t,x)\xi(x,y) \right]  = \delta(t-s)R(x-y),
 \label{eq:covariancenoise}
\end{equation}
where $R$ is a smooth and compactly supported function. Considering the case $k=2$ for simplicity, we may write 
\begin{multline*}
 \tilde v(t,x_1,x_2) =  \\  \frac{1}{\lambda}\mathbb E \left[ \int_{\R}dy_1\int_{\R}dy_2\int_{0}^t ds_1\int_{0}^t ds_2 p_{t-s_1}(x_1-y_1)p_{t-s_2}(x_2-y_2) \xi(s_1,y_1)\xi(s_2,y_2)\partial_{y_1} q(s_1,y_1)\partial_{y_2} q(s_2,y_2)\right]\\
 + \frac{1}{\sqrt{\lambda}} p_t \ast q_0(x_1)\,\,\mathbb E \int_{\R}dy_2\int_{0}^t ds_2p_{t-s_2}(x_2-y_2)\xi(s_2,y_2)\partial_{y_2} q(s_2,y_2) + 1\leftrightarrow 2 + p_t \ast q_0(x_1)p_t \ast q_0(x_2),
\end{multline*}
where $1\leftrightarrow 2$ denotes  the previous term after exchanging indices $1$ and $2$. In the sequel we will discard the terms depending on $p_t\ast q_0$ which play no role in the following computation, because they solve the homogeneous heat equation. Using  \eqref{eq:covariancenoise}, we obtain 
\begin{multline*}
\tilde v(t,x_1,x_2) = \frac{1}{\lambda}
\int_{\R}dy_1\int_{\R}dy_2 R(y_1-y_2) \int_{0}^t ds \  p_{t-s}(x_1-y_1)p_{t-s}(x_2-y_2) \mathbb E\left[ \partial_{y_1} q(s,y_1)\partial_{y_2} q(s,y_2) \right]\\
+ \text{ terms depending on }p_t\ast q_0.
\end{multline*}
Thus, using that $p_t(x)$ solves the heat equation and $p_{t-s}(\cdot) \Rightarrow \delta_0(\cdot)$ as $s\to t$,  we obtain 
$$
\partial_t \tilde v(t,x_1,x_2)  = \frac{1}{\lambda} R(x_1 - x_2) \mathbb E\left[ \partial_{x_1}q(t,x_1) \partial_{x_2}q(t,x_2)\right] + \frac 1 2 (\partial_{x_1}^2 + \partial_{x_2}^2) \mathbb E\left[ q(t,x_1)q(t,x_2) \right]. 
$$
Finally, if $R$ converges to a delta function, the noise $\xi$ becomes a space time white noise, and assuming one can exchange the derivatives $\partial_{x_1}, \partial_{x_2}$ with the expectation, we obtain that $\tilde v(t,x_1,x_2)$ satisfies \eqref{eq:trueevolution2}.  

\begin{remark}
\label{rem:knowninphysics}
The fact that the function $\tilde v(t,\vec x )$ defined in \eqref{eq:defvtilde} satisfies the evolution \eqref{eq:trueevolution2} was essentially known in the physics literature. Indeed, the operator in the right hand side of \eqref{eq:trueevolution2} appears in \cite[Eq. (2.17)]{bernard1998slow} where it was shown to be related to the moments of a stochastic PDE \cite[Eq. (2.2)]{bernard1998slow} which has the same form as \eqref{eq:Kolmogorov2}. 
\end{remark}

\appendix
\section{Approximating Gamma and PolyGamma functions} \label{app:A}

For $n \geq 1$ the polygamma functions have a series representation
 \begin{equation} \label{eq:polygamma series expansion}\psi_n(t)=(-1)^{n+1} n!\sum_{k=0}^{\infty} \frac{1}{(t+k)^{n+1}},\end{equation}

\begin{lemma}\label{lem:stirling polygamma} For all $m\geq 1$, $z \in \mathbb{C} \setminus [0,-\infty)$,
$$\psi_m(z)=(-1)^{m+1} m! \left[ \frac{1}{mz^m}+\frac{1}{2z^{m+1}}+\frac{m+1}{6 z^{m+2}}+\int_0^{\infty} \frac{(m+1)(m+2)(m+3)}{(x+z)^{m+4}} \frac{P_3(x)}{6} dx \right],$$
where $P_3(x)$ is the third order Bernoulli polynomial with period $1$, and
$$\left| \int_0^{\infty} \frac{(m+1)(m+2)(m+3)}{(x+z)^{m+4}} \frac{P_3(x)}{6} dx \right| \leq \left| \frac{(m+1)(m+2)}{120}\frac{1}{ z^{m+3}}\right|.$$
\end{lemma}

\begin{proof}
The first statement is proved by applying the Euler-Maclaurin formula to the series expansion \eqref{eq:polygamma series expansion} of $\psi_m(z)$. The inequality follows from the fact that $\sup_x |P_3(x)| \leq 1/20$. 
\end{proof}

\begin{lemma} For $|z|<1$, $m \geq 0$, 
$$\psi_m(z)=(-1)^{m+1}m! z^{-(m+1)} + \sum_{k=0}^{\infty} (-1)^{k+m+1} \zeta(k+m+1) (k+1)_m z^k.$$
We also have
$$\psi_m(z)=(-1)^{m+1}m! z^{-(m+1)} + \sum_{k=0}^{n} (-1)^{k+m+1} \zeta(k+m+1) (k+1)_m z^k+R^n_m(z),$$
where 
$$|\Re[R^n_m(z)]|, |\Im[R^n_m(z)]| \leq \frac{(n+m+1)!}{(n+1)!} \zeta(n+m+2) |z|^{n+1}.$$
\end{lemma}

\begin{proof}
The first equation is the Laurent expansion of $\psi_m(z)$ around $0$. The bound on the remainder comes from Taylor's theorem.
\end{proof}

\begin{lemma}\label{lem:stirling weak}
for $\text{arg}(z)$ strictly inside $(-\pi, \pi)$, as $|z| \to \infty,$
$$\log \Gamma(z)=\left(z-\frac{1}{2}\right) \log(z)-z+\frac{1}{2} \log\left(2 \pi \right)+O\left(\f{1}{z}\right).$$
and
$$\psi(z)=\log(z)-\f{1}{2z}+O\left(\f{1}{z^2}\right).$$
\end{lemma}
\noindent These are special cases of  \cite[equations 6.1.42 and 6.4.11]{AbramowitzStegun1965}

\begin{lemma} \label{lem:gamma bounds weak} for each $\theta > 0$, there exist constants $C$ and $D$, such that for all $y$, 
$$e^{-\frac{\pi}{2} |y|+C+(\theta-\frac{1}{2}) \log(|y|)} \leq |\Gamma(\theta+\mathbf{i}y)|,$$
and for each $\epsilon, \theta>0$, there exists $M$ such that for all $y>M$, 
$$e^{(-\frac{\pi}{2}-\epsilon) |y|} \leq |\Gamma(\theta+\mathbf{i}y)| \leq e^{(-\frac{\pi}{2}+\epsilon) |y|}.$$
\end{lemma}

\begin{proof}
The first statement follows from applying the Euler-Maclaurin formula to the series expansion of $\log(\Gamma(z))$ and simplifying. The second statement follows from the first order Stirling approximation of $\Gamma(z)$. 
\end{proof}

\begin{lemma} \label{lem:gamma log convex} For any $\epsilon>0$, there exists an $M>0$ such that if $y>M$, $t \geq \frac{1}{2}+\epsilon$, then $\Re[\psi_1(t+\mathbf{i} y)]>0$
\end{lemma}

\begin{proof}
By Lemma \ref{lem:stirling polygamma}, we have
\begin{equation} \label{eq:expand psi1}\Re[\psi_1(t+\mathbf{i} y)] \geq \Re\left[\frac{1}{t+\mathbf{i} y}+\frac{1}{2(t+\mathbf{i} y)^2}+\frac{1}{3(t+\mathbf{i} y)^3}\right]-\left|\frac{1}{20 (t+\mathbf{i} y)^4}\right|.\end{equation}
We expand the first two summands of \eqref{eq:expand psi1}
$$\Re\left[\frac{1}{t+\mathbf{i} y}+\frac{1}{2(t+\mathbf{i} y)^2}\right] = \frac{t}{t^2+y^2}+\frac{t^2-y^2}{2(t^2+y^2)^2} \geq \frac{t-\frac{1}{2}}{t^2+y^2}> \frac{\epsilon}{t^2+y^2}.$$
The third and fourth summands of \eqref{eq:expand psi1} are bounded above by $\frac{1}{3(t^2+y^2)^{3/2}}$ and $\frac{1}{20((t^2+y^2)^2}$ respectively, so we can choose an $M$ large enough that $\Re[\psi_1(t+\mathbf{i} y)]>0$. 
\end{proof}

\begin{lemma} \label{lem:gamma log increasing} There exists $M \in \mathbb{R}$ such that for any $t \in [0,1]$, $|y|>M$, $\Re[\psi(t+\mathbf{i} y)] >0$
\end{lemma}

\begin{proof}
Lemma \ref{lem:stirling weak} implies that as $y \to \infty$ are
$$\psi(t+\mathbf{i} y) \sim 1/2 \log((t-1)^2+y^2)+\mathbf{i} \arctan \left(\frac{y}{t-1}\right)+\frac{1}{2(t+\mathbf{i} y)}.$$
Thus as $|y| \to \infty$, $\Re[\psi(t+\mathbf{i} y)] \to +\infty$. 
\end{proof}

\begin{lemma} \label{lem:sine bounds}
For all $\theta \in \mathbb{R}$ and $|y| \geq 1$,  we have
$$\frac{2 \pi}{e^{\pi |y|}+1} \leq \frac{\pi}{|\sin(\pi(\theta+\mathbf{i} y))|} \leq \frac{2 \pi}{e^{\pi |y|}-1}.$$
For all $\theta,y \in \mathbb{R}$, we have
$$\frac{\pi}{|\sin(\pi(\theta+\mathbf{i} y))} \leq \frac{\pi}{|\sin(\pi \theta)|}.$$
\end{lemma}

\begin{proof}
The inequalities are straightforward to prove using $\sin(z)=\frac{e^{\mathbf{i}z}-e^{-\mathbf{i} z}}{2 \mathbf{i}}.$
\end{proof}

\section{Bounds for dominated convergence}\label{sec:bounds}

In this appendix we will complete the proofs of Lemma \ref{lem: loop bound} and Lemma \ref{lem: non uniform bound on integrand of kernel in epsilon ball}, \ref{lem: full kernel bound}, and \ref{lem:fredholm det expansion bound}.

\begin{proof}[Proof of Lemma \ref{lem: loop bound}]
We first prove \eqref{eq: loop bound 1 lemma}.
For $z\in \mathcal{D}_{\epsilon}(\phi_{\epsilon})$, and $v',v \in \mathcal{C}\setminus \mathcal{C}^{\epsilon}$, the expression $\left| \frac{1}{z-v'} \right|$ is bounded, and
\begin{equation} \label{eq: loop bound 1}\left| \frac{\pi}{\sin(\pi(z-v))} \frac{1}{\Gamma(z)} \right| \leq  \frac{2 \pi e^{{\frac{\pi}{2} |\Im[z]|-C-Re[z-1/2] \log[\Im[z]]}}}{e^{\pi |\Im[z-v]|-1}}. \end{equation}
by Lemma \ref{lem:gamma bounds weak} and Lemma \ref{lem:sine bounds}. Because $\theta<1$, for small enough $\epsilon$, $1/2 \epsilon \leq Re[z-v] \leq 1-\delta$, so that $|\sin(\pi(z-v))|$ is bounded below by a constant $c$ by Lemma \ref{lem:sine bounds}, and $\frac{1}{|\Gamma(z)|}$ is bounded above on $\mathcal{D}_{\epsilon,t}(\phi_{\epsilon})$ by Lemma \ref{lem:gamma bounds weak}. Thus
\begin{equation} \label{eq: loop bound 2} \left| \frac{\pi}{\sin(\pi(z-v))} \frac{1}{\Gamma(z)} \right| \leq C.\end{equation}
for some constant $C$. 

The function $\Gamma(v)$ has a pole at $0$, and $h(v)$ has a pole of order $2$ at $0$. For small enough $\delta$ and $t>1$, when $v \in \mathcal{C} \cap B_{\delta}(0)$. We know  $\Gamma(z)$ is well approximated by $\frac{1}{z}$ near $0$ and $h(\theta)-h(v)$ is well approximated by $\frac{1}{z^2}$ near $0$. For any constant $\eta>0$, we can choose an $\epsilon$, such that for all $y \in (-\epsilon,\epsilon)$, 
$$\left|\frac{1}{\i y} e^{\frac{1}{(\i y)^2}}\right| \leq \left|\frac{1}{\epsilon} e^{-\frac{1}{\epsilon^2}}\right|<\eta.$$
The contour $\mathcal{C}$ crosses $0$ along the imaginary axis, so we can use the above bound with $\eta$ as small as desired to control $e^{t\frac{(h(\theta)-h(v))}{2}}$, and for any $v \in \mathcal{C} \setminus B_{\delta}(0)$, $\Gamma(v)$ is holomorphic and thus bounded, so

\begin{equation} \label{eq: loop bound 3} \left| \Gamma(v) e^{t\frac{(h(\theta)-h(v))}{2}}\right| \leq C', \end{equation}
for some constant $C'$.

For all $z$, we have
\begin{equation} \label{eq: loop bound 5} |e^{t(h(z)-h(\theta))}| \leq e^{\frac{h'''(\theta)}{4} \epsilon^3} \leq C'''.\end{equation}
For all $v \in \mathcal{C}\setminus \mathcal{C}^{\epsilon}$, 
$$|e^{t\frac{(h(\theta)-h(v))}{4}}| \leq e^{-t \eta/4},$$
by Proposition \ref{prop:contour through 0}. Thus there exists $T>0$ such that for all $t>T$, 
\begin{equation} \label{eq: loop bound 6} |e^{t\frac{(h(\theta)-h(v))}{4}}|e^{t^{-1/3} \sigma(\theta) y (z-v)}| \leq |e^{-t \eta/4} e^{-t^{1/3} \sigma(\theta) y}|<1.\end{equation}
The last inequality comes from choosing $t$ sufficiently large. 

Altogether \eqref{eq: loop bound 1}, \eqref{eq: loop bound 2}, \eqref{eq: loop bound 3}, \eqref{eq: loop bound 5}, \eqref{eq: loop bound 6} imply that for all $z \in \mathcal{D}_{\epsilon}(\phi_{\epsilon})$, $v \in \mathcal{C}^{\epsilon}$,

\begin{multline}  \left| \frac{\pi}{\sin(\pi(z-v))} \frac{\Gamma(v)}{\Gamma(z)}\frac{ e^{t(h(z)-h(v))-t^{1/3} \sigma(\theta) y Re[z-v]}}{z-v'} \right| \\
\leq  \frac{2C'''}{\e} e^{-t \eta/4}  \max[C',C''] \min\left[C,\frac{2 \pi e^{{\frac{\pi}{2} |\Im[z]|-C-Re[z-1/2] \log(\Im[z])}}}{e^{\pi |\Im[z-v]|-1}} \right] . \label{eq: loop bound 7} \end{multline}

The left hand side of \eqref{eq: loop bound 7} is the integrand of $K_{u_t}(v,v')$, so we can set $G(z,v,v')$ equal to the right hand side of \eqref{eq: loop bound 7}. Observe that $\min\left[C,\frac{2 \pi e^{{\frac{\pi}{2} |\Im[z]|-C-Re[z-1/2] \log[\Im[z]]}}}{e^{\pi |\Im[z-v]|-1}} \right]$ 
is bounded above by a constant and has exponential decay in $\Im[z]$ for $\Im[z] \to +\infty$, thus we can set
 $$R_1=\int_{\mathcal{D}_{\epsilon,t}(\phi_{\epsilon})} \min\left[C,\frac{2 \pi e^{{\frac{\pi}{2} |\Im[z]|-C-Re[z-1/2] \log[\Im[z]]}}}{e^{\pi |\Im[z-v]|-1}} \right] dz < \infty.$$
 Then
$$|K_{u_t}(v,v')| \leq \int_{\mathcal{D}_{\epsilon,t}(\phi_{\epsilon})} G(z,v,v') dz \leq R_1  \f{2C'''}{\e} e^{-t \eta/4}  \max[C',C''] \leq R_2  e^{-t \eta/4},$$
where $R_2=R_1 \f{2C'''}{\e} \max[C',C''].$

Note that \eqref{eq: loop bound 2 lemma} follows from \eqref{eq: loop bound 1 lemma}, because $K_{u_t}(v,v')$ depends on $v'$ only through the factor $\frac{1}{z-v'}$ in the integrand. Thus we can apply the same argument  where $\frac{1}{z-v'}$ is multiplied by $t^{-1/3}$ to get
\begin{equation} \label{eq: loop bound 8} |K_{u_t}(v,v')| \leq R_2  t^{-1/3}e^{-t \eta/4}.
\end{equation}
Now \eqref{eq: loop bound 2 lemma} follows from \eqref{eq: loop bound 8} and the definition of $\o{K}$. 
\end{proof}

\begin{proof}[Proof of Lemma \ref{lem: non uniform bound on integrand of kernel in epsilon ball}]
For $z \in \mathcal{D}_{\epsilon}^{\epsilon}(\phi_{\epsilon})$ and $v \in \mathcal{C}^{\epsilon}$ the function $\left|\frac{1}{\o{z}-\o{v}'}\right|$ is bounded and

\begin{equation} \label{eq: bounding kernel 2}\left|\frac{ t^{-1/3} \pi}{\sin(\pi t^{-1/3} (\o{z}-\o{v}))} \frac{\Gamma(\theta+t^{-1/3} \o{v})}{\Gamma(\theta +t^{-1/3} \o{z})}\right| \leq c \frac{t^{-1/3}}{t^{-1/3} \epsilon} \leq \frac{c}{\epsilon}, \end{equation}

The second inequality is true because $\Gamma$ is holomorphic in a neighborhood of $\theta$, and $\sin(\theta+\mathbf{i} y) \geq \sin(\theta)$ for all $y \neq 0$. Set $r=\max_{\o{z} \in B_{3 \epsilon}(0)} Re[\o{z}]^3$. we then have
\begin{align} \label{eq: bounding kernel 3} |e^{t[h(z)-h(v)]}| 
\leq e^{(\frac{h'''(\theta)}{6}+\eta) r+t (\frac{h'''(\theta)}{6}+\eta) (Re[\o{z}^3]-r)+t(-\frac{h'''(\theta)}{6}+\eta) \o{v}^3} 
\leq e^{(\frac{h'''(\theta)}{4}) r+t(\frac{ h'''(\theta)}{12}) (Re[\o{z}^3]-r)-t(\frac{h'''(\theta)}{12}) \o{v}^3}, \end{align}
\begin{equation} \label{eq: bounding kernel 4} |e^{-\sigma(\theta) y (\o{z}-\o{v}) }|\leq e^{-\sigma(\theta) y (\epsilon -\o{v})} \end{equation}
The first inequality follows from Taylor expanding $h(z)$ around $\theta$, setting $\eta=\f{h'''(\theta)}{12}$. The second inequality is true because $\Re[z]=\epsilon$. 

\noindent Inequalities \eqref{eq: bounding kernel 3} and \eqref{eq: bounding kernel 4} together yield
\begin{multline} \left|\exp\left(t[h(z)-h(v)]-\sigma(\theta) y (\o{z}-\o{v}) \right)\right| 
\\ \leq \exp\left(\frac{h'''(\theta)}{4} r+t\frac{ h'''(\theta)}{12} (Re[\o{z}^3]-r)-t\frac{h'''(\theta)}{12} \o{v}^3-\sigma(\theta) y (\epsilon -\o{v})\right)\label{eq: bounding kernel 5}
\end{multline}
and 
\begin{multline*}
 \eqref{eq: bounding kernel 5} \leq \exp \left(\frac{h'''(\theta)}{4} r  +t \frac{h'''(\theta)}{12} (Re[\o{z}^3]-r) -t\frac{h'''(\theta)}{24} \o{v}^3-\sigma(\theta) y \left(\epsilon -\frac{\sqrt{24 \sigma(\theta) y}}{h'''(\theta)}\right)\right)\\
 \leq c'\exp \left(t\frac{h'''(\theta)}{12} (Re[\o{z}^3]-r)-t\frac{h'''(\theta)}{24}\o{v}^3\right).
 \end{multline*}
The last inequality is true because $\Re[t( h'''(\theta)/24 )+\o{v}^4 \sigma(\theta) y]$ achieves its maximum at \linebreak $\Re[\o{v}]=\sqrt{24 \sigma(\theta)y}{h'''(\theta)}.$

Let
$$\o{f}(\o{v},\o{v}',\o{z})=\frac{C}{\epsilon^2}e^{t\frac{h'''(\theta)}{12} (Re[\o{z}^3]-r)} e^{-t\frac{h'''(\theta)}{24}},$$
where $C= c c'$. 
Altogether \eqref{eq: bounding kernel 2}, and \eqref{eq: bounding kernel 5} yield
$$\left| \frac{1}{\o{z}-\o{v}'}\frac{ t^{-1/3} \pi}{\sin(\pi t^{-1/3} (\o{z}-\o{v}))} \frac{\Gamma(\theta+t^{-1/3} \o{v})}{\Gamma(\theta +t^{-1/3} \o{z})} e^{t[h(z)-h(v)]-\sigma(\theta) y (\o{z}-\o{v}) } \right| \leq \o{f}(\o{v},\o{v}',\o{z}).$$
where the left hand side is the integrand of $\o{K_{u_t}^{\epsilon}} (\o{v},\o{v}')$, so the integrand is bounded above by $\o{f}(\o{v},\o{v}',\o{z})$. 
Note that $\o{f}$ is decreasing in $t$, so setting $t=1$ gives that the integrand of $\o{K}_{u_t}(\o{v}, \o{v}')$ is less than or equal to $\frac{C}{2 \epsilon}\exp\left(\frac{h'''(\theta)}{12} (Re[\o{z}^3]-r)-\frac{h'''(\theta)}{24}\o{v}^3\right)$. This function is independent of $t$ and has exponential decay in $\cos(3 \phi_{\epsilon}) |z|^3$ so integrating it over $\mathcal{D}_{\epsilon,t}^{\epsilon} (\phi_{\epsilon})$ gives a finite result, so we have proven the first claim.

Set $\ell=\int_{\epsilon}^{e^{\mathbf{i}(\pi- \phi) } \infty}e^{\frac{h'''(\theta)}{12} (Re[\o{z}^3]-r)} d\o{z} <\infty,$ then

$$\o{K_{u_t}^{\epsilon}} (\o{v},\o{v}') \leq \int_{\mathcal{D}_{\epsilon,t}^{\epsilon}(\phi_{\epsilon})} \o{f}(\o{v},\o{v}',\o{z}) d\o{z} \leq  \frac{\ell C }{ \epsilon^2}e^{-\frac{h'''(\theta)}{24}\o{v}^3}=C_1 e^{-\frac{h'''(\theta)}{24}\o{v}^3},$$
where $C_1=\frac{\ell C}{ 2 \epsilon^2}$. This completes the proof.
\end{proof}

\begin{proof}[Proof of Lemma \ref{lem: full kernel bound}]
For $v,v' \in \mathcal{C}^{\epsilon}$, $z \in \mathcal{D}_{\epsilon,t}(\phi_{\epsilon}) \setminus \mathcal{D}_{\epsilon,t}^{\epsilon}(\phi_{\epsilon})$, the function $\left| \frac{1}{z-v} \right|$ is bounded and
\begin{equation} \label{eq: full kernel bound 2} \left| \frac {\pi}{\sin(\pi( z-v))} \frac{ \Gamma(v)}{\Gamma(z)} \right| \leq \frac{ce^{\frac{\pi}{2} |\Im[z]|-C-(\theta-\frac{1}{2}) \log(|\Im[z]|)}} {e^{\pi |\Im[(z-v)]|-1}} \end{equation}
This inequality follows from Lemma \ref{lem:gamma bounds weak}. As long as $\phi_{\epsilon}<\frac{\pi}{6}$, $|\Im[z-v]| > \delta$ for some $\delta$, so
the right hand side of \eqref{eq: full kernel bound 2} is bounded for $\Im[z] \in \mathbb{R}$, and when $\Im[z]$ is large it has exponential decay of order $e^{-\pi/2 \Im[z]}.$ Also
\begin{equation} \label{eq: full kernel bound 3} |e^{-t^{1/3}\sigma(\theta) y( z-v)}| =|e^{-t^{1/3}\sigma(\theta) y( \Re[z-v])}| \leq |e^{-t^{1/3}\sigma(\theta) y( \epsilon \sin(\phi_{\epsilon})-\o{v})}|,\end{equation}
and by Lemma \ref{lem:contour vertical deformation} there exists $\eta>0$, such that
\begin{equation} \label{eq: full kernel bound 4} |e^{t[h(z)-h(v)]}| =|e^{h(z)-h(\theta)}| |e^{h(\theta)-h(v)}| \leq |e^{-t \eta} e^{-\frac{h'''(\theta)}{12} \o{v}^3}|.
\end{equation}
The last inequality follows from Taylor expanding the $v$ variable term, and applying Lemma \ref{lem:contour vertical deformation} to the $z$ variable term.
There exists a constant $T>0$ such that for all $t>T$, $t \eta/2 \geq t^{1/3} \sigma(\theta) y \epsilon \sin(\phi_{\epsilon}).$ This inequality together with \eqref{eq: full kernel bound 3} and \eqref{eq: full kernel bound 4} implies that for all $t>T$, 
\begin{align} \label{eq: full kernel bound 5} |e^{t[h(z)-h(v)-t^{1/3}\sigma(\theta) y( z-v)}| 
&\leq |e^{-t \eta-\frac{h'''(\theta)}{12} \o{v}^3-t^{1/3}\sigma(\theta) y( \epsilon \sin(\phi_{\epsilon})-\o{v})}| 
\leq |e^{-t \eta/2-\frac{h'''(\theta)}{24} \o{v}^3+(-\frac{h'''(\theta)}{24} \o{v}^3-t^{1/3}\sigma(\theta) y \o{v})}\\
&\leq e^{-t \eta/2-t\frac{h'''(\theta)}{24} \o{v}^3-\sigma(\theta) y (\epsilon -\frac{\sqrt{24 \sigma(\theta) y}}{h'''(\theta)})}
\leq c' e^{-t \eta/2-t\frac{h'''(\theta)}{24} \o{v}^3}.\nonumber
\end{align}
The first inequality follows from our choice of $T$, and the second inequality follows from the fact that $\Re[t( h'''(\theta)/24 )-\o{v}^3 \sigma(\theta) y (-\o{v})]$ achieves its maximum at $\Re[v]=\sqrt{24 \sigma(\theta)y}{h'''(\theta)}.$

Set 
$$g(z,\o{v},\o{v}')=\frac{10}{\epsilon} c' e^{-t \eta/2-t\frac{h'''(\theta)}{24} \o{v}^3}     \left(\frac{ce^{\frac{\pi}{2} |\Im[z]|-C-(\theta-\frac{1}{2}) \log(|\Im[z]|)}} {e^{\pi |\Im[z-v]|}}\right).$$
Together \eqref{eq: full kernel bound 2} and \eqref{eq: bounding kernel 5} imply
\begin{equation}\label{eq: full kernel bound 6} \left| \frac{1}{z-v} e^{t[h(z)-h(v)-t^{1/3}\sigma(\theta) y( z-v)} \frac {\pi}{\sin(\pi( z-v))} \frac{ \Gamma(v)}{\Gamma(z)} \right| \leq g(z,\o{v},\o{v}'). \end{equation}
The left hand side of \eqref{eq: full kernel bound 6} is the integrand of $K_{u_t}(v,v')$. 
The expression $\left(\frac{ce^{\frac{\pi}{2} |\Im[z]|-C-(\theta-\frac{1}{2}) \log(|\Im[z]|)}} {e^{\pi |\Im[(z-v)]|}}\right)$ is bounded as $\Im[z]$ varies in $(-\infty,+\infty)$, and has exponential decay in $\Im[z]$ for large $\Im[z]$, so we can set
$$S=\int_{\mathcal{D}_{\epsilon,t}(\phi_{\epsilon})\setminus \mathcal{D}_{\epsilon,t}^{\epsilon}(\phi_{\epsilon})}\left(\frac{ce^{\frac{\pi}{2} |\Im[z]|-C-(\theta-\frac{1}{2}) \log(|\Im[z]|)}} {e^{\pi |\Im[z-v]|}}\right) dz < \infty.$$
Then
\begin{multline} \left\vert K_t(\theta+\o{v},\theta+\o{v}') -K_t^{\epsilon}(\theta+\o{v},\theta+\o{v}')  \right\vert\\
\leq \int_{\mathcal{D}_{\epsilon,t}(\phi_{\epsilon})\setminus \mathcal{D}_{\epsilon,t}^{\epsilon}(\phi_{\epsilon})} g(z,\o{v},\o{v}') dz
\leq \frac{10S c'}{\epsilon} e^{-t \eta/2} e^{t(-h'''(\theta)/24) \o{v}^3} \xrightarrow{t \to \infty} 0. \end{multline}
\end{proof}

\begin{proof}[Proof of Lemma \ref{lem:fredholm det expansion bound}]
We have the following inequalities,
 \begin{equation} \label{eq:111}\o{K_{u_t}^{\epsilon}}(\o{v}_i,\o{v}_j) \leq C_1 e^{-t \frac{h'''(\theta)}{24} \o{v}_i^3},\end{equation}
\begin{align}\o{K}_{u_t}(\o{v}_i,\o{v}_j)
=\o{K_{u_t}^{\epsilon}}(\o{v}_i,\o{v}_j)&+\left(t^{-1/3} K_{u_t}(\theta+t^{-1/3} \o{v}_i, \theta+t^{-1/3} \o{v}_j)-t^{-1/3} K_{u_t}^{\epsilon}(\theta+t^{-1/3} \o{v}_i, \theta+t^{-1/3} \o{v}_j)\right) \nonumber\\
&\leq  C_2 e^{-t \eta/2} e^{t(-h'''(\theta)/24) \o{v}^3}+C_1 e^{-t \frac{h'''(\theta)}{24} \o{v}_i^3} 
\leq C_3 e^{t(-h'''(\theta)/24) \o{v}^3} \label{eq:112},
\end{align}
where $C_3=C_1+C_2e^{- \eta/2}$. Inequality \eqref{eq:111} follows from Lemma \ref{lem: non uniform bound on integrand of kernel in epsilon ball}.  The first inequality of \eqref{eq:112} comes from Lemma \ref{lem: non uniform bound on integrand of kernel in epsilon ball}, Lemma \ref{lem: full kernel bound} and the fact that $t>1$. 
Hadamard's bound implies
$$|\det(\o{K_{u_t}^{\epsilon}}(\o{v}_i,\o{v}_j))_{i,j=1}^{m}| \leq m^{m/2} C_3^{m/2} \prod_{i=1}^m e^{-t \frac{h'''(\theta)}{24} \o{v}_i^3}.$$
Set $\o{H}_m(\o{v},\o{v}')=m^{m/2} C_3^{m/2} \prod_{i=1}^m e^{-t \frac{h'''(\theta)}{24} \Re[\o{v}_i]^3}$, and set $L=\int_0^{\infty} e^{-\frac{h'''(\theta)}{24} x^3} dx< \infty$. Then
$$ \int_{(\o{C^{\epsilon}})^m} \o{H}_m(\o{v},\o{v}') \leq \int_{(\mathcal{C}_0)^m} \o{H}_m(\o{v},\o{v}') \leq m^{m/2}C_3^{m/2} L^m.$$
Thus
$$1+\sum_{m=1}^{\infty} \frac{1}{m!} \int_{(\mathcal{C}_0)^m} \o{H}_m(\o{v},\o{v}') 
\leq 1 + \sum_{m=1}^{\infty} \frac{m^{m/2}C_3^{m/2} L^m}{m!}.$$
So because $m! \geq \sqrt{\frac{2 \pi}{m}}\left(\frac{m}{e}\right)^{m}$, we have
$$1+\sum_{m=1}^{\infty} \frac{1}{m!} \int_{(\o{C^{\epsilon}})^m} \o{H}_m(\o{v},\o{v}') \leq 1+\sum_{m=1}^{\infty} \frac{1}{m!} \int_{(\mathcal{C}_{0})^m} \o{H}_m(\o{v},\o{v}')  < \infty.$$
\end{proof}

\bibliographystyle{alphaabbrGuillaumebettersorting}

\bibliography{LargeDeviationsForStickyBrownianMotionsArxiv3}

\end{document}